\documentclass[12pt]{amsart}
\usepackage{amssymb}
\usepackage{enumerate}
\DeclareMathAlphabet{\mathcal}{OMS}{cmsy}{m}{n}
\usepackage[all]{xy}
\usepackage{hyperref}   
\usepackage{mathtools}
\usepackage[top=25mm, outer=20mm, bottom=25mm, inner=25mm]{geometry}

\usepackage{mathptmx}

\usepackage{relsize}
\usepackage[bbgreekl]{mathbbol}
\DeclareSymbolFontAlphabet{\mathbb}{AMSb}
\DeclareSymbolFontAlphabet{\mathbbl}{bbold}

\newcommand{\doublecoset}[3]{%
	{\textstyle #1}
	\mkern-4mu\scalebox{1.5}{$\diagdown$}\mkern-5mu^{\textstyle #2}%
	\mkern-4mu\scalebox{1.5}{$\diagup$}\mkern-5mu{\textstyle #3} }

\newcommand{\prism}{{\mathlarger{\mathbbl{\Delta}}}}
\newcommand{\cris}{{\mathrm{cris}}}
\newcommand{\crispris}{{\prism_{\cris}}}
\newcommand{\BK}{{\mathrm{BK}}}
\newcommand{\BKprism}{{\prism_{\BK}}}
\newcommand{\BKcrispris}{{\prism_{\cris}^{\BK}}}

\newcommand{\Fil}{{\mathrm{Fil}}}
\newcommand{\K}{{\mathsf{K}}}
\newcommand{\perf}{{\mathrm{perf}}}
\newcommand{\Space}{{\mathsf{space}}}
\newcommand{\Spaces}{{\mathsf{spaces}}}
\newcommand{\Spf}{\ensuremath{\operatorname{Spf}\,}}
\newcommand{\sA}{{\mathcal A}}
\newcommand{\sE}{{\mathcal E}}
\newcommand{\sG}{{\mathcal G}}
\newcommand{\sS}{{\mathcal S}}
\newcommand{\Zp}{{\mathbb{Z}_p}}

\theoremstyle{remark}
\newtheorem{theorem}{\rm{\textbf{Theorem}}}[section]
\newtheorem{corollary}[theorem]{\rm{\textbf{\textbf{Corollary}}}}
\newtheorem{lemma}[theorem]{\rm{\textbf{Lemma}}}
\newtheorem{proposition}[theorem]{\rm{\textbf{\textbf{Proposition}}}}
\newtheorem{definition}[theorem]{\rm{\textbf{Definition}}}

\newtheorem{remark}[theorem]{\rm{\textbf{Remark}}}

\author[Q.~Yan]{Qijun Yan}
\title[Integral Frobenius period maps for Shimura variet{i}es]{On certain integral Frobenius period maps\\ for Shimura varieties and their reductions}
\address{Beijing Institute of Mathematical Sciences and Applications (BIMSA), Beijing, 101\,~408,  China}
\email{yanqmath\symbol{64}bimsa.cn}

\date{}
\begin{document}
	\begin{abstract}
		We formulate an integral Frobenius period map for the framed crystalline prismatization of the $p$-integral model $\mathcal{S}$ of a Shimura variety with good reduction. By analyzing reductions of this map, we derive a period map from the mod $p$ fiber $S$ of $\mathcal{S}$ to the moduli stack of 1-1 truncated local $G$-shtukas in the prismatic topology, which refines the zip period map of $S$ within this topology. Furthermore, we show that the pair $(\mathcal{S}, S)$ is associated with a double $G$-zip. Additionally, we introduce a framework of base reduction diagrams.
	\end{abstract}
	\maketitle
	\section{Introduction}
	Throughout, we fix a rational prime \( p \geq 3 \) and let \( \kappa \) denote a finite extension of the prime field \( \mathbb{F}_p = \mathbb{Z}/p\mathbb{Z} \). We denote its ring of Witt vectors by \( W = W(\kappa) \).
	
	\subsection{Background on Frobenius period maps}
	
	Consider the Kisin-Vasiu integral model \(\mathcal{S} = \mathcal{S}_{\mathsf{K}}\) over \( W \) of a Hodge type Shimura variety with level \( \mathsf{K} \), which is hyperspecial at \( p \) (\cite{KisinIntegralModels}). It is a smooth quasi-projective scheme over \( W \). Accompanying \( \mathcal{S} \) is a reductive group \(\mathcal{G}\) defined over \(\mathbb{Z}_p\) and a minuscule cocharacter \( \mu: \mathbb{G}_{m, \kappa} \to G \), where \( G := \mathcal{G} \otimes_W \kappa \) (we denote the base change \( \mathcal{G}_W \) by \(\mathcal{G}\) itself). Let \( S = S_{\mathsf{K}} := \mathcal{S} \otimes_{W} \kappa \) be the special fiber of \(\mathcal{S}\). Associated to the reductive pair \((G, \mu)\) is the stack \( G\text{-}\textsf{Zip}^{\mu} \) of \( G \)-zips of type \(\mu\)~\cite{PinkWedhornZiegler2}, which morally parametrizes \(1\)-truncated \( p \)-divisible groups of type \( \mu \). This stack serves as the period domain for the zip period map of \( S \),
	\begin{equation}\label{Eq:ZipMap}
		\zeta: S \to G\text{-}\textsf{Zip}^{\mu}.
	\end{equation}
	In this generality, it was f{i}rst constructed in \cite{ChaoZhangEOStratification}, generalizing previous construction of \cite{Moonen&WedhornDiscreteinvariants} and \cite{ViehmannWedhornEOPELtype} for PEL type Shimura varieties, to define and study the Ekedahl-Oort stratification of \( S \); see \cite{YanLocCon} for a local construction of \(\zeta\) which is more adapted to this work and an introduction to the history of it.  
	
	Our view of zip period maps has broadened: beyond their role in defining and analyzing the Ekedahl–Oort stratification of Shimura varieties, we now recognize them as instances of a wider family of maps, which we term \emph{Frobenius period maps}.These are typically defined by trivializing linear objects, such as vector bundles, with Frobenius structures.

	\subsection{Motivation for this work} \label{S:Motivation}
	Attached to the pair \((G, \mu)\) is also the \emph{algebraic} stack, \([\prescript{}{1}{\mathcal{C}_1^\mu}/\mathrm{Ad}_\varphi G]\) which is a function field analogue of \(G\text{-}\textsf{Zip}^{\mu}\). Here, \(\prescript{}{1}{\mathcal{C}_1^\mu}\) denotes Viehmann’s double coset space of type \(\mu\), a sub-quotient sheaf of the loop group \(\mathcal{L}G\), which turns out to be represented by a scheme \cite[Theorem \textbf{A}]{Yan23zip}, and \(\mathrm{Ad}_\varphi G\) denotes the \(\varphi\)-conjugation action of \( G \) on \(\prescript{}{1}{\mathcal{C}_1^\mu}\). The map \(\varphi: G \to G\) is the relative Frobenius of \( G \) (as \( G \) is defined over \(\mathbb{F}_p\) we have \( G = G^{\varphi} \)). Additionally, we established in~\cite{Yan23zip} a sequence of algebraic stacks
	\begin{equation}\label{Eq:Mapdelta}
		[\prescript{}{1}{\mathcal{C}_1^\mu}/\mathrm{Ad}_{\varphi} G]\ \xrightarrow{\delta}\ G\text{-}\mathsf{Zip}^{\mu}\ \rightarrow\ [\prescript{}{1}{\mathcal{C}_1^{\varphi(\mu)}}/\mathrm{Ad}_{\varphi} G]
	\end{equation}
	that become isomorphisms upon perfection. In particular, we have an isomorphism
	\begin{equation}\label{Eq:PerfIsom}
		\delta^{\perf}:[\prescript{}{1}{\mathcal{C}_{1}^{\mu}}/\mathrm{Ad}_{\varphi} G]^{\perf}\ \cong\ (G\text{-}\textsf{Zip}^{\mu})^{\perf}.
	\end{equation}
	This article is driven by the question: does the zip period map factor through the connecting map $\delta$?  A positive answer would not only refine the map, but-more importantly—establish a direct morphism from the number‑theoretic object $S$ to the function‑field object $[\prescript{}{1}{\mathcal{C}_{1}^{\mu}}/\mathrm{Ad}_{\varphi} G]$.
	
	Building upon\cite{Yan18}, our strategy is to construct the map from \( S \) to the moduli stack \( [\prescript{}{1}{\mathcal{C}_1^{\mu}}/\mathrm{Ad}_{\varphi} G] \) by studying the reduction of Breuil-Kisin prismatic cohomology of the universal abelian scheme \( \sA \) over \( \mathcal{S} \). However, a detailed analysis in \cite{YanLocCon} indicated that the desired factorization does not happen in the Zariski topology. This limitation arises from the non-uniqueness of Frobenius lifts on a \( p \)-adic \( W \)-algebra \( R \), unless \( R \) is perfect (in other words, the \( p \)-power Frobenius on \( R/pR \) is bijective). Experience gained while conducting the work \cite{YanLocCon} suggests that this refinement should be achievable within the prismatic topology of Bhatt and Scholze \cite{BSPrism}. This observation, together with the results of \cite{Yan18}, leads us to the guiding principle: \begin{center} \emph{The desired refinement should arise as the “reduction’’ of an \textbf{integral} Frobenius period.} \end{center}
	The first goal of this article is to construct such an integral Frobenius period map:
	
	\vspace{1em}
	\noindent \textbf{The first Main Result (informal statement).}
	In this article, we construct an \emph{integral Frobenius period map}
	\[
	\xi_{\BKcrispris} : \mathcal{S}_{\tilde{\prism}_{\cris}}
	\;\longrightarrow\;
	\Bigl[\mathbf{C}_{\BKcrispris}^{\tilde{\mu}} \big/ \mathrm{Ad}_\varphi \,\mathcal{L}^+_{\BKcrispris} \mathcal{G} \Bigr],
	\]
	over the framed crystalline prismatization \(\mathcal{S}_{\tilde{\prism}_{\cris}}\) of $\mathcal{S}$. We refer to it as the \emph{Breuil-Kisin Frobenius period map} of $(\mathcal{S}, S)$. The meaning and construction of this map will be explained in the next two subsections.
	
	\subsection{The source space for the integral Frobenius period map}
	
	The pair \( (W, p) \) constitutes a crystalline prism, as defined in \cite{BSPrism}, which shall be referred to as the \emph{base crystalline prism}. Attached to \( \underline{W} \) is the \emph{base Breuil-Kisin prism}
	\[
	\underline{\mathfrak{S}}\ :=\ (\mathfrak{S} = W[[t]], E(t) := t + p, \varphi_{\mathfrak{S}}).
	\]
	We write
	\[
	W_{\prism_{\cris}} = \kappa_{\crispris} = (\Spf W/\underline{W}) = (\operatorname{Spec} \kappa/\underline{W})
	\]
	for the sites of (crystalline) prisms over \( \underline{W} \), and
	\(
	W_{\BKprism} = (\Spf W/\underline{\mathfrak{S}})
	\)
	for the sites of (unramified Breuil-Kisin) prisms over \( \underline{\mathfrak{S}} \). The canonical projection \( \underline{\mathfrak{S}} \to \underline{W} \) induces a canonical inclusion functor \( W_{\crispris} \subseteq\, W_{\BKprism} \). The source space of our integral Frobenius period map is the following \emph{modified} prismatic site: 
	\begin{definition}
		Let \( \widehat{\mathcal{S}} \) denote the \( p \)-adic formal completion of \( \mathcal{S} \). The \emph{framed crystalline prismatization} of \( \mathcal{S} \) over \( \underline{W} \) is the site \( \mathcal{S}_{\tilde{\prism}_{\cris}} \) consisting of pairs \( (\, \underline{R},\, x: \Spf R \to \widehat{\mathcal{S}}\, ) \), where \( \underline{R} \in W_{\prism_{\cris}} \) is a crystalline prism over \( \underline{W} \), and \( x: \Spf R \to \widehat{\mathcal{S}} \) is a morphism of \( p \)-adic formal schemes over \( \Spf W \).
	\end{definition}
	We now justify the role of \( \mathcal{S}_{\tilde{\prism}_{\cris}} \) as the source space. As a category, its objects are the crystalline integral points of \( \widehat{\mathcal{S}} \) equipped with Frobenius structures. Moreover, its structural map to \( W_{\crispris} \) makes it a sheaf over \( W_{\crispris} \)  whence enabling a \emph{geometric} perspective. We will treat \( \mathcal{S}_{\tilde{\prism}_{\cris}} \) as a \( W_{\crispris} \)-\(\Space \) which is obtained from the \( W_{\BKprism}\!\)-\(\Space\ \mathcal{S}_{\BKprism} := (\widehat{\mathcal{S}}/\underline{\mathfrak{S}})_{\prism} \). Informally, a \(\Space\) is a category of \emph{points} in a general sense (see \S\,\ref{S:IntPerMapIntro}). 
	
	Thanks to the work of Imai, Kato and Youcis~\cite{IKYPrisReal}, we know that there exists a prismatic \( \mathcal{G} \)-torsor \( \mathbb{J}_{\prism} \) attached to \( \widehat{\mathcal{S}} \), which forms a key ingredient for this work. This torsor parametrizes trivializations of the pair
	\[
	\bigl(\, \mathrm{H}^1_{\prism}(\widehat{\sA}/\widehat{\sS}),\, \mathrm{T}_{\prism}\, \bigr)
	\]
	where \( \mathrm{H}^1_{\prism}(\widehat{\sA}/\widehat{\sS}) \) is the first relative prismatic cohomology of \( \widehat{\sA} \) over \( \widehat{\sS} \), and
	\(
	\mathrm{T}_{\prism}\ \subseteq\ \mathrm{H}^1_{\prism}(\widehat{\sA}/\widehat{\sS})^{\otimes}
	\)
	is the corresponding collection of prismatic tensors. We denote the restriction of \( \mathbb{J}_{\prism} \) to \( \mathcal{S}_{\BKprism} \) by \(\mathbb{J}_{\BKprism}\) and refer to it as the \emph{Breuil-Kisin prismatic} \( \mathcal{G} \)-torsor of \( \mathcal{S} \).
	
	From \( \mathbb{J}_{\BKprism} \), we derive a key ingredient for the construction of $\xi_{\BKcrispris}$. It is the sheaf \( \mathbb{J}_{\BKcrispris} \) over \( \mathcal{S}_{\tilde{\prism}_{\cris}}\!\) (Definition~\ref{Def:BkCrisTors}) which is a torsor under the Breuil-Kisin positive loop group \( \mathcal{L}^+_{\BKcrispris}\mathcal{G} \) and a subspace of \( \mathbb{J}_{\BKprism} \). Here, \( \mathcal{L}^+_{\BKcrispris}\mathcal{G} \) is a group functor over \( W_{\crispris} \) associating each crystalline prism \( \underline{R} \) with the group \( \mathcal{G}(R[[t]]) \); See~\S\,\ref{S:PrisLoopGrp} for more on this.
	
	\subsection{The domain of the integral Frobenius period map}\label{S:IntPerMapIntro}
	Define \( \tilde{\mu}: \mathbb{G}_{m, W} \to \mathcal{G} \) as a cocharacter lifting \( \mu \). Set \( \mathbf{C}^{\tilde{\mu}}_{\BKcrispris} \) to be the sheafification of the presheaf on \( W_{\crispris} \) which assigns to each $\underline{R} \in W_{\crispris}$ the set
	\[
	\mathcal{G}(R[[t]]) \tilde{\mu}(E(t)) \mathcal{G}(R[[t]]) 
	\]
	where \(E(t) = t + p\). It is a subspace of the group functor \( \mathcal{L}_{\BKcrispris}\mathcal{G} \) on $W_{\crispris}$ given by \( \underline{R} \mapsto \mathcal{G}(R((t))) \) (see~\S\,\ref{S:PrisLoopGrp}) and serves as the period domain for our integral Frobenius period map.
	
	\vspace{1em}
	\noindent \textbf{Idea of construction of $\xi_{\BKcrispris}$}: it is
	induced by an \( \mathcal{L}^+_{\BKcrispris}\mathcal{G} \)-equivariant map 
	\[
	\xi_{\BKcrispris}^{\sharp}: \mathbb{J}_{\BKcrispris} \to \mathbf{C}_{\BKcrispris}^{\tilde{\mu}}
	\]
	of \( W_{\crispris}\!\)-sheaves, obtained by trivializing the Frobenius endomorphism of \( \mathrm{H}^1_{\prism}(\widehat{\sA}/\widehat{\sS}) \) preserving tensors.
	
	Our first goal has thus been achieved; see Lemma \ref{Lem: DefBKPerMap} for details. Below we concentrate on analyzing its reductions.

	\subsection{Framework of Base Reduction Diagrams}
	To clarify the principle in \S~\ref{S:Motivation} and present our main results, we introduce the simple yet useful framework of \emph{base reduction diagrams}, which formalise what it means for a space or map to be the reduction of another via prescribed \emph{reduction maps}. 
	Readers are encouraged to skip this subsection on a first reading and to return to it as needed.
	
	\begin{definition}
		A \emph{crystalline \( W \)-algebra} is a flat \( W \)-algebra \( R \) that admits a crystalline prism structure for some Frobenius lift \( \varphi: R \to R \). A \emph{crystalline \( \kappa \)-algebra} is a \( \kappa \)-algebra admitting a crystalline frame, i.~e., there exists a crystalline prism \( \underline{R} \) such that \( \bar{R} = R/pR \).
	\end{definition}
	\begin{definition}
		We define \( W_{\cris} \) and \( \kappa_{\cris} \) as the sites of crystalline algebras over \( W \) and \( \kappa \), respectively. The sites \( W_{\BKprism}, W_{\crispris} = \kappa_{\crispris}, W_{\cris}, \) and \( \kappa_{\cris} \) are referred to as the \emph{base \(\mathsf{spaces}\)}. \emph{Maps} between these base \(\mathsf{spaces}\) are continuous functors.
	\end{definition}
	\begin{definition}
		A \(\mathsf{space}~X\) over a base \(\mathsf{space}~Y\) is a prestack of groupoids over \(Y\). We also say \(X\) is a \(\mathsf{space}\) \emph{fibered} over \(Y\). An object \( x \) of \( X \) is referred to as a \emph{point} and is denoted by \( x \in X \) after an abuse of notation. Alternatively, \( X \) can be described as a \( Y \)-\(\Space\) with the \emph{fiber \(\Space\)} at \( y \in Y \) given by the groupoid \( X(y) \). \emph{Maps} between \(\Spaces\) are morphisms of functors.
	\end{definition}
	As noted above, the framed crystalline prismatization \(\mathcal{S}_{\tilde{\prism}_{\cris}}\) is a sheaf over \(W_{\crispris}\) and hence can be viewed as a \( W_{\crispris}\!\)-\(\Space \).
	\begin{definition}[Base Reduction Diagram]
		The following sequence, referred to as the \emph{base reduction diagram}, describes natural maps between our base \(\mathsf{spaces}\)
		\begin{equation}\label{BRedDiagmIntro}
			W_{\BKprism} \xrightarrow{\pi} W_{\crispris} \xrightarrow{\pi_{\crispris}} W_{\cris} \xrightarrow{\pi_{\cris}} \kappa_{\cris}.
		\end{equation}
		Here, the functor \(\pi\) is given by  base change of prisms along the canonical projection \( \underline{\mathfrak{S}} \to \underline{W} \), \( \pi_{\crispris} \) denotes the defrobenius map \( (R, \varphi: R \to R) \mapsto R\) and \( \pi_{\cris} \) denotes the reduction modulo \(p\) map.
	\end{definition}
	In~\S\,\ref{S:RelBaseSpace}, we also define \emph{relative base \(\Spaces\)} for any smooth \( p \)-adic formal scheme over \( W \). More specifically, attached to the \( p \)-adic completion \( \widehat{\mathcal{S}} \) of \( \mathcal{S} \) are the relative base \(\Spaces\)
	\[
	\mathcal{S}_{\BKprism}, \quad \mathcal{S}_{\tilde{\prism}_{\cris}}, \quad \mathcal{S}_{\crispris} = S_{\crispris}, \quad \mathcal{S}_{\cris}, \quad S_{\widetilde{\cris}}\quad \text{and}\quad S_{\cris}
	\]
	which are naturally fibered over \( W_{\BKprism},\ W_{\tilde{\prism}_{\cris}}=W_{\crispris},\, W_{\crispris}, \, W_{\cris},\ W_{\cris}\) and \( \kappa_{\cris} \), respectively. For definitions of the last three \(\Spaces\), the reader is referred to~\S\,\ref{S:RelBaseSpace}.  We can further form a \emph{relative base reduction diagram}:
	\begin{equation}\label{RelBRedDiagmIntro}
		\xymatrix{
			& \mathcal{S}_{\tilde{\prism}_\cris} \ar[dd] \ar[dl]_{\mathrm{defrobenius}} \ar[dr]^{\partial\,\mathrm{Mod}_p} & \\
			\mathcal{S}_{\cris} \ar[dddr]_{\mathrm{Mod}_p} \ar[dr]^{\partial\mathrm{Mod}_p} && S_{\crispris} \ar[dl]^{\mathrm{defrobenius}} \ar[dddl]^{\mathrm{deprism}} \\
			& S_{\widetilde{\cris}} \ar[dd]^{\partial\,\mathrm{Mod}_p} & \\
			&& \\
			& S_{\cris}. &
		}
	\end{equation}
	Here each arrow is understood but see \S \ref{S:RelBasRedDiag} for more details. 
	For any map $f\colon Y \to Z$ in \eqref{BRedDiagmIntro} or \eqref{RelBRedDiagmIntro} and any $\mathsf{space}$ $\mathcal{F}$ over $Z$, we view $\mathcal{F}$ as a $\mathsf{space}$ over $Y$ via pullback.  We denote it by $f^{*}\mathcal{F}$, $\mathcal{F}|_{Y}$, or simply $\mathcal{F}$.  The fiber at $y \in Y$ is then $\mathcal{F}\bigl(f(y)\bigr)$.

	The term \emph{reduction} in ``reduction maps'' refers to reduction modulo an ideal or to the idea of the removal of something. The base reduction diagrams above allow us to make the following two definitions:

	\begin{definition}\label{Def:RedalongBasRedDiag}
		Let \( \mathrm{Red}: Y \to Z \) be a reduction map as in~\eqref{BRedDiagmIntro} or~\eqref{RelBRedDiagmIntro}, and \( \mathcal{F} \)  a \(\mathsf{space}\) over \( Y \). We say that \( \mathcal{F} \) is \emph{of reduction along \( \mathrm{Red} \)} (by \( \mathcal{G} \)) if there exists a \(Z\)-\(\mathsf{space}\ \mathcal{G} \), necessarily unique up to isomorphism, and an isomorphism \( \mathcal{F} \cong \mathcal{G}|_{Y} \) of \( Y \)-\(\mathsf{spaces} \).
	\end{definition}
	
	\begin{definition}\label{Def:RedMapalongBaseRedDiag}
		Let \( \mathrm{Red}: Y \to Z \) be a reduction map as in~\eqref{BRedDiagmIntro} or~\eqref{RelBRedDiagmIntro}, and \( f: \mathcal{F}_1 \to \mathcal{F}_2 \) a map of \(Y\)-\(\Spaces \). We say that \( f \) is \emph{of reduction along \( \mathrm{Red} \)} (by \( g \)) if there exist \( Z \)-\(\Spaces \) \( \mathcal{G}_1, \mathcal{G}_2 \) such that \(\mathcal{F}_i \cong \mathrm{Red}^* \mathcal{G}_i\) for \( i = 1, 2 \), and a map \( g: \mathcal{G}_1 \to \mathcal{G}_2 \) of \( Z \)-\(\Spaces \) intertwining \(f\) with \(\mathrm{Red}^*g\). In this case, \( g \) is called the \emph{reduction} of \( f \) (along \( \mathrm{Red} \)).
	\end{definition}
	
	\begin{remark}
		Theorems \ref{Thm:ExdRepIntd} and \ref{Thm:CryPrisTorRedIntd} below motivate our Definition~\ref{Def:RedalongBasRedDiag} above, while Theorem \ref{Thm:MainRedThm} motivates Definition \ref{Def:RedMapalongBaseRedDiag}. These definitions are modeled on the fact that a geometric object over \( W \), such as a \( W \)-scheme \( X \) with special fiber \( X_{\kappa} \), has characteristic \( p \) if and only if there exists a \( \kappa \)-scheme \( Y \) and an isomorphism \( X \cong \mathrm{Res}_{\kappa/W} Y \) of functors over \( (\mathbf{Alg}/W) \). This is equivalent to saying that the canonical reduction modulo \( p \) map from \( X \) to the Weil restriction functor \(\mathrm{Res}_{\kappa/W} X_{\kappa} \)  is an isomorphism. The framework of base reduction diagrams guides the step‑by‑step reduction procedure and is of independent interest in its own right.
	\end{remark}

	\subsection{Main results on the reductions of $\xi_{\BKcrispris}$}

	Let \( \mathbb{J}_{\crispris} \) denote the restriction of \( \mathbb{J}_{\prism} \) to \( \mathcal{S}_{\crispris} \). An immediate reduction of $\xi_{\BKcrispris}$ is the \emph{crystalline prismatic Frobenius period map} for \( S \):
	\[
	\xi_{\crispris}: S_{\crispris} \longrightarrow [\mathbf{C}_{\tilde{\prism}_\cris}^{\tilde{\mu}} / \mathrm{Ad}_{\varphi} \mathcal{L}^+_{\crispris} \mathcal{G}],
	\]
	corresponding to a $\mathcal{G}$-equivariant map 
	\(
	\xi_{\crispris}^{\sharp}: \mathbb{J}_{\crispris} \to \mathbf{C}_{\tilde{\prism}_\cris}^{\tilde{\mu}}
	\). We refer to Lemma \ref{Lem:DefCrisPerMap} for details and unexplained notations, and to \S~\ref{S:CompareClasCryPerMap} for the connection of $\xi_{\crispris}$ with classical Frobenius period map.
	
	Our main results on the reductions of the integral Frobenius period map are summarized as follows. 
	\begin{theorem}\label{Thm:MainRedThm} Write  \( \mathrm{J}_{\crispris} = \mathrm{J}_{\crispris} \otimes_W \kappa \) for the (mod $p$) crystalline prismatic $G$-torsor of $S$. 
		\begin{enumerate}
			\item The Breuil-Kisin Frobenius period map \( \xi_{\BKcrispris}^{\sharp} \) induces a map of \( W_{\crispris} \)-\(\Spaces \):
			\(
			\rho^{\sharp}_{\BKcrispris}: \mathrm{J}_{\crispris} \to (\prescript{}{1}{\mathcal{C}_{1}^{\mu}})_{\crispris},
			\)
			which in turn induces a map of \( W_{\crispris} \)-stacks:
			\[
			\rho_{\BKcrispris}: S_{\crispris} \longrightarrow [\prescript{}{1}{\mathcal{C}_1^{\mu}}/\mathrm{Ad}_{\varphi} G]_{\crispris}.
			\]
			
			\item The crystalline Frobenius period map \( \xi^{\sharp}_{\crispris} \) induces a map of \( W_{\crispris} \)-\(\Spaces \):
			\(
			\rho_{\crispris}^{\sharp}: \mathrm{J}_{\crispris} \to (\prescript{}{1}{ \mathbf{C}_{1}^{\tilde{\mu}}})_{\crispris},
			\)
			which in turn induces a map of \( W_{\crispris} \)-stacks:
			\[
			\rho_{\crispris}: S_{\crispris} \longrightarrow [\prescript{}{1}{ \mathbf{C}_{1}^{\tilde{\mu}}}/\mathrm{Ad}_{\varphi} G]_{\crispris},
			\]
			where $[\prescript{}{1}{ \mathbf{C}_{1}^{\tilde{\mu}}}/\mathrm{Ad}_{\varphi} G]_{\cris}$ is the mixed characteristic analogue of moduli stack $[\prescript{}{1}{\mathcal{C}_1^{\mu}}/\mathrm{Ad}_{\varphi} G]_{\cris}$. 
			\item The maps \( \rho_{\BKcrispris} \) and \( \rho_{\crispris} \) fit into the following commutative diagram of \( W_{\crispris} \)-stacks:
			\[
			\xymatrix{
				& S_{\crispris} \ar[dl]_{\rho_{\crispris}} \ar[dr]^{\rho_{\BKcrispris}} & \\
				[\prescript{}{1}{\mathbf{C}_1^{\tilde{\mu}}}/\mathrm{Ad}_{\varphi} G]_{\crispris} \ar[rr]_{\cong}^{\epsilon_{\crispris}} & & [\prescript{}{1}{\mathcal{C}_1^{\mu}}/\mathrm{Ad}_{\varphi} G]_{\crispris}
			}
			\]
			\item Let \( \delta_{\crispris} \) denote the restriction to \( W_{\crispris} \) of the morphism \( \delta \) in \eqref{Eq:Mapdelta}, and let \( \zeta_{\cris} \) denote the restriction of \( \zeta \) to \( \kappa_{\cris} \). Then the composition \( \delta_{\crispris} \circ \rho_{\crispris}: S_{\crispris} \to (G\text{-}\textsf{Zip}^{\mu})_{\crispris} \) is of reduction along \( W_{\crispris} \to \kappa_{\cris} \) by \( \zeta_{\cris} \).
		\end{enumerate}
	\end{theorem}
	
	\begin{remark}\begin{enumerate}
			\item  One consequence of this theorem is that it globalizes while simultaneously furnishing a geometric and conceptual understanding of an otherwise mysterious map:  \[ (\delta^{\perf})^{-1} \circ \zeta^{\perf}: S^{\perf}\longrightarrow [\prescript{}{1}{\mathcal{C}_1^{\mu}}/\mathrm{Ad}_{\varphi} G]^{\perf}, \]
			that is, a map from an arithmetic object from \emph{number field} to that from \emph{function field}.
			\item By the final item of the theorem, the maps \( \rho_{\crispris} = \rho_{\BKcrispris} \) refine the zip period map \( \zeta \) in the prismatic topology. In fact, it determines \( \zeta \), despite not being of reduction along \( W_{\crispris} \to \kappa_{\cris} \), since \( \zeta_{\cris} \)  determines \( \zeta \) as \( S_{\cris} \) has enough crystalline points.
			
			\item As established in \cite{Yan23zip}, the moduli stack \( [\prescript{}{1}{\mathcal{C}_1^{\mu}}/\mathrm{Ad}_{\varphi} G] \) admits a zip stack interpretation, specifically an isomorphism of \( \kappa \)-stacks,
			\(\epsilon':
			[\prescript{}{1}{\mathcal{C}_{1}^{\mu}}/\mathrm{Ad}_{\varphi} G] \cong [G/\mathsf{R}_{\varphi} \mathsf{E}_{\mu}].
			\)
			Correspondingly, the composition \( \epsilon'_{\crispris} \circ \rho_{\BKcrispris} \) yields a \( (\mathsf{E}_{\mu})_{\crispris} \)-torsor \( (\mathsf{Z}_{\mu})_{\crispris} \) equipped with a \( (\mathsf{Z}_{\mu})_{\crispris} \)-equivariant map $\eta_{\crispris}$. We describe the pair \( ((\mathsf{Z}_{\mu})_{\crispris}, \eta_{\crispris}) \) in \S \ref{S:PrisGauge}. 
			
			\item  The $\mathsf{E}_{\mu}$-torsor \( (\mathsf{Z}_{\mu})_{\crispris} \) exists already in the Zariski topology (see \S \ref{S:UnivZip}). We call the pair \( (\mathsf{Z}_\mu, \rho_{\crispris}) \) the \emph{prismatic zip gauge} attached to \( (\mathcal{S}, S) \); it determines a double $G$-zip,
			which closely resembles Drinfeld's version of $G$-zip \cite{Drinfeld2023shimurian} (c.f. \cite{ShenGauge}). See \ref{Rmk:Conn} for more details.
		\end{enumerate}
		
	\end{remark}

	\subsection{Main results on Reductions of \(\Spaces\)}\label{SS:res}
	The results of this subsection should logically be presented before Theorem \ref{Thm:MainRedThm}. For example, Theorem \ref{Thm:ExdRepIntd} is already involved there.  
	
	Define \( \prescript{}{1}{\mathbf{C}_{1}^{\tilde{\mu}}} \) as the \( W_{\cris} \)-\(\Space \) obtained by sheafifying the presheaf on \( W_{\cris} \) that assigns the set
	\[ 
	\doublecoset{\mathrm{K}_1(R)}{\mathcal{G}(R) \tilde{\mu}(p) \mathcal{G}(R)}{\mathrm{K}_1(R)}
	\]
	to each crystalline algebra \(R \in W_{\cris}\), where \( \mathrm{K}_1(R) \) is the kernel of the reduction map \( \mathcal{G}(R) \to G(R/pR) \). The theorem below, stated as Theorem \ref{Thm:ExtdRepthm} in the main text, extends \cite[{Theorem $\mathrm{A}'$}]{Yan23zip}.
	\begin{theorem}\label{Thm:ExdRepIntd}
		The \( W_{\cris} \)-\(\Space\ \prescript{}{1}{\mathbf{C}_{1}^{\tilde{\mu}}} \) is of reduction along the map \( W_{\cris} \to \kappa_{\cris} \) by \( (\prescript{}{1}{\mathcal{C}_1^{\mu}})_{\cris} \). In other words, there is a canonical isomorphism of \( W_{\cris} \)-\(\Spaces \)
		\begin{align*}
			\prescript{}{1}{\mathbf{C}_{1}^{\tilde{\mu}}} \xrightarrow[\cong]{\epsilon}\ (\prescript{}{1}{\mathcal{C}_1^{\mu}})_{\cris}, \quad
			g\tilde{\mu}h\ \mapsto\ \bar{g}\mu(t)\bar{h}.
		\end{align*}
		In particular, the $\Space$ $\prescript{}{1}{\mathbf{C}_{1}^{\tilde{\mu}}}$ is independent of the choice of lift $\tilde{\mu}$ of $\mu$. 
	\end{theorem}

	\begin{theorem}\label{Thm:CryPrisTorRedIntd}
		The \( S_{\crispris}\!\)-\(\Space\ \mathbb{J}_{\crispris}\) is of reduction along the defrobenius map \( S_{\crispris}\! \to S_{\widetilde{\cris}} \).
	\end{theorem}
	This is the content of Theorem~\ref{Thm:KeyRedSpace}. In \S\,\ref{S:CryPrisTorRed}, we construct a \( p \)-adic formal scheme \( \widehat{\mathbb{J}} \) over \( \widehat{\mathcal{S}} \) from the relative cohomology \( \mathrm{H}^1_{\mathrm{dR}}(\widehat{\mathcal{A}}/\widehat{\mathcal{S}}) \cong \mathrm{H}^1_{\cris}(\widehat{\mathcal{A}}/\widehat{\mathcal{S}}) \) with its structure morphism \( \widehat{\mathbb{J}} \to \widehat{\mathcal{S}} \) being a \( p \)-adic \'{e}tale \( \mathcal{G} \)-torsor (see Lemma~\ref{Lem:petaleTorsor}). The existence of $\widehat{\mathbb{J}}$ is an outcome of the study of the relation between prismatic crystalline cohomology $\mathrm{H}^1_{\crispris}(\widehat{\mathcal{A}}/\widehat{\mathcal{S}})$ and the classical crystalline cohomology $\mathrm{H}^1_{\cris}(\widehat{\mathcal{A}}/\widehat{\mathcal{S}})$. The reduction of the crystalline prismatic \( \mathcal{G} \)-torsor \( \mathbb{J}_{\crispris} \) along \( \mathcal{S}_{\crispris} \to \mathcal{S}_{\tilde{\prism}_{\cris}} \) is given by the restriction \( \widehat{\mathbb{J}}_{\cris} \) to \( \mathcal{S}_{\cris} \) of \( \widehat{\mathbb{J}} \), which we show is of reduction along \( \mathcal{S}_{\cris} \to S_{\widetilde{\cris}} \) (Lemma \ref{Lem:TorRed}).  In particular,  $\mathrm{J}_{\crispris}$ is of reduction along the deprism map \( S_{\crispris} \to S_{\cris} \) by the restriction \( \mathrm{J}_{\cris} \) to \( S_{\cris} \) of the \( \kappa \)-scheme \( \mathrm{J} := \widehat{\mathbb{J}}\otimes_W\kappa \).

	\subsection{Structure of the article}	
	In \S~\ref{S:EtaleReal+FCrystal}, we review foundational material on prismatic topology, the \'etale realization functor and prismatic $F$-crystals for smooth $p$-adic formal schemes. \S~\ref{LocSys+PrisTorsonShimuraVar} shifts to the setting of Shimura varieties, focusing on the integral \'etale realization $\omega_{\mathsf{K}, \text{\'et}}^{\mathrm{an}}$ and its associated prismatic $\mathcal{G}$-torsor, summarizing the necessary input from \cite{IKYPrisReal}. \S~\ref{S:Prismatization} discusses the framed crystalline prismatization $\mathcal{S}_{\tilde{\prism}_{\cris}}$ for a $p$-integral model $\mathcal{S}$ of a Shimura variety, i.e., the source space of our integral Frobenius period map. The formalism applies to any smooth $p$-adic formal scheme $\mathfrak{X}$. \S\S~\ref{S:LanguagueSpace}--\ref{S:ReductionTheory} introduce the language of $\Spaces$ and the framework of \emph{base reduction diagrams}, while \S~\ref{S:LoopGp} defines various versions of loop groups associated with $\mathcal{G}$. These sections are primarily devoted to definitions.  \S~\ref{S:RepProof} contains the proof of Theorem~\ref{Thm:ExdRepIntd}. In \S~\ref{S:FrobPerMap}, we formally define the prismatic Frobenius period maps $\xi_{\BKcrispris}$ and $\xi_{\crispris}$. \S~\ref{S:RedGTorsor} proves Theorem \ref{Thm:CryPrisTorRedIntd} and \S~\ref{S:RedFrobPerMap} proves Theorem~\ref{Thm:MainRedThm}. The final section, \S~\ref{S:ZipGauges}, expands on the prismatic zip gauges introduced earlier and gives a zip-style description of $\rho_{\BKcrispris}$.
	
	\subsection{Notation and conventions}
	
	\begin{enumerate}\label{Notation}
		\item Let \( \mathcal{C} \) and \( \mathcal{D} \) be categories. By an inclusion \( \mathcal{C} \subseteq \mathcal{D} \) or \( \mathcal{C} \hookrightarrow \mathcal{D} \), we mean a faithful functor from \( \mathcal{C} \) to \( \mathcal{D} \), typically understood from context.
		\item For the purpose of formulating particular reduction problems, we use (Grothendieck) sites, often variants of the prismatic sites of Bhatt-Scholze, as alternatives to the fundamental notion of ``spaces." Our usage of sites is simplified: for instance, we do not consider \emph{morphisms} between sites. A \emph{map} of sites \( A \to B \) is simply a continuous functor from \( A \) to \( B \). We use the term ``restriction" loosely: if \( f: X \to Y \) is a map and \( \mathcal{F} \) is a mathematical object over \( Y \), the restriction of \( \mathcal{F} \) to \( X \) along \( f \), denoted by $\mathcal{F}|_X$, refers to the pullback \( f^* \mathcal{F} \), whenever clear in context. If \( \mathcal{F} \) is a (set-valued or groupoid valued) sheaf, so is \(  \mathcal{F}_X \) by the continuity of \( f \).
		
		\item For a ring \( R \), we denote by \((\mathbf{Alg}/R)\), \((\mathbf{FlatAlg}/R)\), and \((\mathbf{PerfAlg}/R)\) the categories of \( R \)-algebras, flat \( R \)-algebras, and perfect \( R \)-algebras, respectively. When regarded as sites, these categories are equipped with the flat topology. For a site $\mathcal{C}$, we use $\mathbf{Shv}(\mathcal{C}) $ to denote the category of set-valued sheaves over $\mathcal{C}$. 
		\item For a reductive group scheme \( \sG \) over a Dedekind domain \( R \) (e.g., \( R = \Zp \)) and a faithful representation \( \rho_0: \sG \hookrightarrow \mathrm{GL}_R(\Lambda_0) \), with \( \Lambda_0 \) a projective \( R \)-module, there exists a set of tensors \( \mathrm{T}_0 \subseteq \Lambda_0^{\otimes} \) such that \( \sG \) is the scheme-theoretic stabilizer of \( \mathrm{T}_0 \) \cite{BroshiGTors}. We write \( \sG = \mathrm{Stab}(\rho_0, \Lambda_0) \) for such a realization.
		
		\item Let \( \mathcal{G} \) be a reductive group over \( \Zp \). For a \( \Zp \)-linear tensor category \( \mathcal{C} \), the notation \( \mathcal{G}\text{-}\mathcal{C} \) denotes the category of \( \mathcal{G} \)-objects in \( \mathcal{C} \), i.e., the category of \( \Zp \)-linear tensor functors \( \omega: \mathbf{Rep}_{\Zp}(\mathcal{G}) \to \mathcal{C} \).
		
	\end{enumerate}

	\subsection{Acknowledgment}
	I would like to thank Fabrizio Andreatta for sharing with me the idea of establishing a connection between Breuil-Kisin modules, zip period maps, and Viehmann's double coset space, an idea that initially informed my PhD thesis and form the foundation for this work. I am also grateful to Naoki Imai, Tong Liu,  Yong-Suk Moon, Mao Sheng, Angel Toledo, Alex Youcis, Chia-Fu Yu, and Chao Zhang for their valuable discussions and encouragement.
	
	\section{Etale realization functor and  prismatic $F$-crystals}\label{S:EtaleReal+FCrystal}
	We first introduce some notations that will be effective in this section. Let \(\kappa\) be a finite extension of $\mathbb{F}_p$, \(W := W(\kappa)\), and \(K_0 = W\left[\frac{1}{p}\right]\). Let \(K/K_0\) be a finite totally ramified extension of degree \(e\geq 1\), and \({\mathcal{O}}_K\) be its ring of integers.  Let $\pi$ be a uniformizer of $K$, and $E=E(t)\in W[t]$ be the minimal polynomial of $\pi$. Though for later application we will take \(K = K_0\), we do not assume this in this section.
	
	\subsection{Prisms and prismatic sites}\label{S:BSetupPrism}
	
	We first recall some definitions and facts on prisms, adapted to our needs, following \cite{BSPrism} and \cite[\S 1.1]{IKYTannak}.
	
	A ($p$-torsion free) \emph{prism} $\underline{A}$ is a triple $(A, \varphi, I)$, where $A$ is a $p$-torsion free ring together with a Frobenius lift $\varphi = \varphi_A$ of $A$, and $I \subseteq A$ is an invertible ideal such that $A$ is derived $(p, I)$-complete and $p \in I + \varphi(I)$. Often we write $\underline{A} = (A, I)$ when the Frobenius $\varphi$ is clear or irrelevant for the discussion, and similarly sometimes we write $\underline{A} = (A, \varphi)$ when $I$ is clear. 
	
	A prism $\underline{A} = (A, I)$ is \emph{bounded} if the quotient ring $A/I$ has bounded $p^\infty$-torsion, i.e., $A/I[p^\infty] = A/I[p^n]$ for some integer $n$.  For a bounded prism $\underline{A}$, the underlying ring $A$ is (classically) $(p, I)$-complete and $A/I$ is $p$-adically complete (see e.g., \cite[Lemma 1.2]{IKYTannak}). A prism $\underline{A} = (A, I)$ is \emph{crystalline} if we have $I = pA$. In particular, a crystalline prism is $p$-torsion free. In the remaining part of this article, by a prism we always mean a $p$-torsion free and bounded prism, unless otherwise specified. 
	
	A \emph{morphism of prisms} $(A, I) \to (B, J)$ is a homomorphism of rings $A \to B$, compatible with Frobenius structure and sending $I$ into $J$. By the rigidity property of morphisms between prisms \cite[Proposition 3.5]{BSPrism}, the existence of such a morphism implies that we necessarily have $J = I B$. A morphism of prisms $(A, I) \to (B, J)$ is \emph{$I$-completely flat} (resp. \emph{faithfully flat, smooth, \'etale}) if $B \otimes^{\mathrm{L}}_A A/I$ is concentrated in degree $0$ and the canonical homomorphism $A/I \to B \otimes^{\mathrm{L}}_A A/I$ is flat (resp. faithfully flat, smooth, \'etale).
	
	For each prism $\underline{A} = (A, I)$, the ideal $I$ is finitely generated.  Since $I$ is an invertible ideal, the open subset  $ U(I):=\mathrm{Spec} A - \mathrm{V}(I)$ is affine. Its ring of global sections is denoted by $A[1/I]$.  For each mathematical object $X$ (e.g., sheaves, schemes, modules) over $A$, we denote by $X[1/I]$ its restriction on $U(I)$; for example, for each $A$-module $M$, we write $M[1/I]$ for the base change $M \otimes_A A[1/I]$. We write $U(\underline{A})$ for the subset $\mathrm{Spec} A - V(p, I)$. Clearly we have $U(I)\subseteq U(\underline{A})$. Hence for a vector bundle $\sE$ (or other meaningful mathematical objects) over $U(\underline{A})$, the notation $\sE[1/I]$ is understood. 
	
	\begin{proposition}\label{Prop: ProduPrism}
		Let $\underline{A}=(A, I)$ be a (bounded) prism. \begin{enumerate}
			\item Let $B$ be a $(p, I)$-adically complete $A$-algebra, then $A \to B$ is $(p, I)$-completely flat (resp. faithfully flat, smooth, \'etale) if and only if the corresponding morphism $\Spf B\to \Spf A$ is $(p, I)$-adically flat (resp. faithfully flat, smooth, \'etale).
			\item Let $B$ be a $(p, I)$-adically complete $A$-algebra. If $A\to B$ is $(p, I)$-completely \'etale, there exists a unique Frobenius structure on $B$, which makes  $(B, IB)$ a (bounded $p$-torsion free) prism and the ring map $(A, I)\to (B, IB)$ is a morphism of prisms. 
		\end{enumerate}
	\end{proposition}
	\begin{proof} This is \cite[Lemma 1.3]{IKYTannak} and \cite[Proposition 1.4]{IKYTannak}, except the $p$-torsionness statement in (2), which we now address.  Since $ B$ is $(p, I)$-adically complete, it is $p$-adically complete (hence $p$-adically separated). On the other hand, since $A\to B$  is $(p, I)$-adically \'etale, hence is $p$-adically \'etale, in particular, $p$-adically flat. Hence for each $n$, $B/p^nB$ is flat over $A/p^nA$ for all $n\geq 1$. Using the flatness of these maps (considering the injective map $A/p^{n-1}A \hookrightarrow A/p^nA$ given by multiplication by $p$) and the fact that $B$ is $p$-adically separated, one concludes that $B$ is indeed $p$-torsion free. 
	\end{proof}

	Let $\mathfrak{X}$ be a $p$-adic formal scheme. The \emph{absolute prismatic site} $\mathfrak{X}_{\prism} = \mathfrak{X}_{\prism, \mathrm{flat}}$ of $\mathfrak{X}$ has as its underlying category the opposite category of triples $\underline{x} = (\underline{R}, x)$, where $\underline{R} = (R, I)$ is a bounded prism and $x: \Spf R/I \to \mathfrak{X}$ is a morphism of $p$-adic formal schemes. Morphisms in this category consist of morphisms of prisms $(\underline{R}, x) \to (\underline{A}, x')$ such that the induced morphisms $\Spf A \to \Spf R$ intertwine $x$ and~$x'$. 
	
	The topology of $\mathfrak{X}_{\prism}$ is given by declaring that a family $\{\underline{R} \to (A_i, J_i)\}$ of morphisms is a covering if the corresponding family $\{\Spf A_i/J_i \to \Spf R/I\}$ is a covering of $\Spf R/I$ in the flat topology, i.e., the corresponding map $\sqcup_i \Spf A_i/J_i \to \Spf R/I$ is $p$-adically faithfully flat and quasi-compact. It follows from \cite[Corollary 3.2]{BSPrism} that $\mathfrak{X}_{\prism}$ is indeed a site. When $\mathfrak{X}$ is affine, say $\mathfrak{X}=\mathrm{Spec} R$, we also write $(\Spf R)_{\prism}$ as $R_{\prism}$ and $x: \Spf (A/I) \to \Spf R$ as $x: R \to A/I$. 
	
	The \emph{structure sheaf} ${\mathcal{O}}_{\mathfrak{X}_{\prism}}$ and the reduced structure sheaf of $\overline{\mathcal{O}}_{\mathfrak{X}_\prism}$ are defined to be the presheaf by $\mathcal{O}_{\mathfrak{X}_{\prism}}(A, I, x) = A$, resp. $\overline{\mathcal{O}}_{\mathfrak{X}_{\prism}}(A, I, x)=A/I$. Indeed, they are \emph{sheaves} by \cite[Corollary 3.12]{BSPrism}. Moreover, so is the presheaf $\mathcal{O}_{\mathfrak{X}_\prism}[1/I_\prism]$ defined by $\mathcal{O}_{\mathfrak{X}_ \prism}[1/I_\prism](A, I, x)=A[1/I]$. All the Frobenius maps $\varphi_R$ in prisms $\underline{R} = (R, I)$ induce a morphism of sheaves of rings: $\varphi_{\mathfrak{X}_{\prism}}: \mathcal{O}_{\mathfrak{X}_{\prism}} \to \mathcal{O}_{\mathfrak{X}_{\prism}}$.
	
	For a morphism $f: \mathfrak{X} \to \mathfrak{Y}$ of $p$-adic formal schemes, the functor 
	$\mathfrak{X}_{\prism} \to \mathfrak{Y}_{\prism}$ is cocontinuous and hence induces a morphism of topoi,
	\[
	(f_{\prism, *}, f_{\prism}^*): \mathbf{Shv}(\mathfrak{X}_{\prism}) \to \mathbf{Shv}(\mathfrak{Y}_{\prism}).
	\]
	We write the $i$-th \emph{relative prismatic cohomology} of $\mathfrak{X}$ over $\mathfrak{Y}$ as
	\(
	\mathrm{H}^i_{\prism}(\mathfrak{X}/\mathfrak{Y}) := \mathbf{R}^i f_{\prism, *} (\mathcal{O}_{\mathfrak{X}_{\prism}}). 
	\)
	The Frobenius map $\varphi_{\mathfrak{X}_{\prism}}: \mathcal{O}_{\mathfrak{X}_{\prism}} \to \mathcal{O}_{\mathfrak{X}_{\prism}}$ induces  (linearized) morphisms between cohomologies: 
	\begin{equation}\label{Eq: PrisFCryExamp}
		\varphi_{\mathrm{H}^i_{\prism}(\mathfrak{X}/\mathfrak{Y})}:   \varphi^* \mathrm{H}^i_{\prism}(\mathfrak{X}/\mathfrak{Y})\to \mathrm{H}^i_{\prism}(\mathfrak{X}/\mathfrak{Y}). 
	\end{equation}
	
	\subsection{Small formal \({\mathcal{O}}_K\)-algebras}\label{S: SmallFormal}
	Following \cite{DLMSPrisCris}, by a \emph{small formal} ${\mathcal{O}}_K$-algebra we mean a Zariski connected $p$-adically complete ${\mathcal{O}}_K$-algebra $R$, such that there is a $p$-adic \'etale morphism
	\begin{equation}\label{Eq: FormalFraming}
		\iota: {\mathcal{O}}_K\langle t_1^{\pm}, \ldots, t_d^{\pm} \rangle \to R,
	\end{equation}
	for some $d \geq 0$. Here ${\mathcal{O}}_K\langle t_1^{\pm}, \ldots, t_d^{\pm} \rangle$ denotes the $p$-adic completion of the polynomial algebra ${\mathcal{O}}_K[t_1^{\pm}, \ldots, t_d^{\pm}]$. Such an $\iota$ is called a \emph{formal framing} of the small formal ${\mathcal{O}}_K$-algebra $R$. By the topological invariance of the \'etale site of a formal scheme, there exists a unique $p$-adic \'etale map \(\iota_0: W<t_1^{\pm}, \ldots, t_d^{\pm}> \to R_0,\) such that $R=R_0\otimes_W O_K$ and $\iota = \iota_0 \otimes_W \mathrm{id}_{{\mathcal{O}}_K}$. We also call $\iota_0$ a formal framing of $R$.
	
	Small formal ${\mathcal{O}}_K$-algebras fall into the category of \emph{base ${\mathcal{O}}_K$-algebras} in the sense of \cite[\S 1.1.5]{IKYPrisReal}. Hence all the results developed for base ${\mathcal{O}}_K$-algebras in \textit{loc. cit.} apply to small formal ${\mathcal{O}}_K$-algebras.
	
	Let \( \mathfrak{X} \) be a smooth \( p \)-adic formal scheme over \( \mathcal{O}_K \). Then by a result of Kedlaya \cite[Lemma 4.9]{BhattSpecializing}, $\mathfrak{X}$ admits a basis (in the Zariski topology) consisting of formal spectrum's of small formal $O_K$-algebras.
	
	\subsection{{Breuil-Kisin prisms attached to small formal ${\mathcal{O}}_K$-algebras}}\label{S: BKPrism}
	Let \(A\) be a small formal \({\mathcal{O}}_K\)-algebra with a formal framing \(\iota_0: W\langle t_1^{\pm}, \ldots, t_d^{\pm} \rangle \to A_0\). Denote by \(\varphi: W\langle t_1^{\pm}, \ldots, t_d^{\pm} \rangle \to W\langle t_1^{\pm}, \ldots, t_d^{\pm} \rangle\) the ``canonical" Frobenius lift of \(W\langle t_1^{\pm}, \ldots, t_d^{\pm} \rangle\), whose restriction on \(W_0\) is the unique Frobenius lift of \(W\) and sends \(t_i\) to \(t_i^p\). Then,
	\[
	\underline{W\langle t_1^{\pm}, \ldots, t_d^{\pm} \rangle} := \big(W\langle t_1^{\pm}, \ldots, t_d^{\pm} \rangle, \varphi, (p)\big)
	\]
	forms crystalline prism. 
	
	By Proposition \ref{Prop: ProduPrism}, the formal framing \(\iota\) induces a unique crystalline prism structure on \(A_0\), denoted by
	\(
	\underline{A_0} := \big(A_0, \varphi_{A_0}, (p)\big),
	\)
	together with a morphism of prisms \(\iota_0: \underline{W\langle t_1^{\pm}, \ldots, t_d^{\pm} \rangle} \to \underline{A_0}\). This in turn induces a \(p\)-torsion free prism,
	\(
	\underline{\mathfrak{S}_A} = \big( \mathfrak{S}_A, \varphi_{\mathfrak{S}_A}, (E) \big),
	\)
	where \(\mathfrak{S}_A := A_0[[t]]\) and \(\varphi_{\mathfrak{S}_A}: \mathfrak{S}_A \to \mathfrak{S}_A\) is given by \(\varphi_{A_0}\) on \(A_0\) and sends \(t\) to \(t^p\). Note that we have $\mathfrak{S}_A/(E)=A$. The prism \(\underline{\mathfrak{S}_A}\) is called the \emph{Breuil-Kisin prism} attached to the small formal \({\mathcal{O}}_K\)-algebra \(A\) (and the additional choice of formal framing \(\iota_0\)). The data \((\underline{\mathfrak{S}_A}, A = \mathfrak{S}_A / (E))\) induces an object of \((A)_{\prism}\), which we often still denote by \(\underline{\mathfrak{S}_A}\).

	\subsection{Prismatic {$F$}-crystals and Prismatic torsors}
	
	Let $\mathfrak{X}$  be a smooth $p$-adic formal ${\mathcal{O}}_K$-scheme. A \emph{prismatic $F$-crystal (in vector bundles)} over $\mathfrak{X}_{\prism} $ is an assignment which assigns each object \(\underline{A}\) of \(\mathfrak{X}_{\prism}\) a finite projective \(A\)-module \(M_A\), together with an isomorphism of \(A[1/I_A]\)-modules, called the \emph{Frobenius map} of \(M_A\),
	\begin{equation}\label{Eq: Frob of PirsCris}
		\varphi_{M_A}: \varphi_A^* M_A[1/I_{A}] \cong M_A[1/I_{A}],
	\end{equation}
	and the following compatibility data: for each morphism \(f: \underline{A} \to \underline{B}\) in \(\mathfrak{X}_{\prism}/T\), an isomorphism
	\(
	f^* M_A \cong M_B,
	\)
	compatible with Frobenius maps. The data is required to satisfy clear compatibility for each composition of morphisms \(\underline{A} \to \underline{B} \to \underline{C}\) in \(\mathfrak{X}_{\prism}\). 
	
	A prismatic $F$-crystal is said to be \emph{effective} if for each $\underline{A}$, the Frobenius map $\varphi_{M_A}$ is induced by a map $\varphi_{M_A,0}: \varphi^*M_A \to M_A$ of $A$-modules such that $\varphi_{M_A}=\varphi_{M_A, 0}[1/I_{A}]$; cf. \cite[Definition 4.1]{BSPrisCrys}. Due to the $I_A$-torsion freeness of $M_A$, such $\varphi_{M_A,0}$, if exists, is necessarily injective and unique. Since $M_A$ is finitely generated and the ideal $I_A$ is locally principal, there exists a smallest integer $r(M_A)$ such that the cokernel of $\varphi_{M_A,0}$ is killed by $r$. This number is called the \emph{height} of \( \underline{M_A}:=(M_A, \varphi_{M_A}) \). An effective prismatic $F$-crystal is said to be \emph{minuscule} if  \( \underline{M_A} \) has height $\leq 1$ (i.e., $\mathrm{coker}(\varphi_{M_A})$ is killed by $I_A$) for all $\underline{A}$. In practice, given an effective  prismatic $F$-crystal, we shall only specify the data: $\varphi_{M_A}: \varphi^*M_A\to M_A$ before inverting $I_A$. 
	
	\begin{definition}[{\cite{GRPrisCris}, \cite{AnschutzLeBrasPrisDieud}}]  Let $\mathfrak{X}$  be a smooth $p$-adic formal ${\mathcal{O}}_K$-scheme. A \emph{prismatic Dieudonn\'e crystal (in vector bundles)} is an effective minuscule prismatic $F$-crystal over $\mathfrak{X}_{\prism}$. 
	\end{definition}
	
	Denote by \(\mathbf{Vect}^{\varphi}(\mathfrak{X}_{\prism})\) the category of prismatic \(F\)-crystals (in vector bundles) over \(\mathfrak{X}_{\prism}\), and by \(\mathbf{Vect}^{\mathrm{an}, \varphi}(\mathfrak{X}_{\prism})\) the category of \emph{analytic prismatic \(F\)-crystals (in vector bundles)} over \(\mathfrak{X}_{\prism}\). Here, the definition of an analytic prismatic \(F\)-crystal is just like that of a prismatic \(F\)-crystal with the following change: \(M_A\) is replaced by a vector bundle \(\sE_A\) over \(U(\underline{A})\) and \(\sE[1/I]\) is understood as the pullback to the affine open \(U(I)\).
	
	Alternatively, we can write an object in \(\mathbf{Vect}^{\varphi}(\mathfrak{X}_{\prism})\), resp. \(\mathbf{Vect}^{\mathrm{an}, \varphi}(\mathfrak{X}_{\prism})\) as a pair,
	\[
	\big(\sE, \varphi^*\sE[1/I_{\prism}] \cong \sE[1/I_{\prism}]\big), \text{ with } \sE \in \mathbf{Vect}(\mathfrak{X}_{\prism}), \text{ resp. } \sE \in \mathbf{Vect}^{\mathrm{an}}(\mathfrak{X}_{\prism}).
	\]
	Here \(\mathbf{Vect}(\mathfrak{X}_{\prism})\), resp. \(\mathbf{Vect}^{\mathrm{an}}(\mathfrak{X}_{\prism})\), denotes the category of \emph{prismatic crystals} resp. \emph{analytic prismatic crystals (in vector bundles)} over \(\mathfrak{X}_{\prism}\) whose objects can be described like those in \(\mathbf{Vect}^{\varphi}(\mathfrak{X}_{\prism})\), resp. \(\mathbf{Vect}^{\mathrm{an}, \varphi}(\mathfrak{X}_{\prism})\), with the data of Frobenius maps in \eqref{Eq: Frob of PirsCris} omitted.
	
	Each of the categories \(\mathbf{Vect}(\mathfrak{X})\), \(\mathbf{Vect}^{\varphi}(\mathfrak{X})\),  \(\mathbf{Vect}^{\mathrm{an}}(\mathfrak{X})\), and \(\mathbf{Vect}^{\mathrm{an}, \varphi}(\mathfrak{X})\) admits a natural \(\mathbb{Z}_p\)-linear \(\otimes\) structure where tensor products and exactness are defined termwise. 
	\begin{proposition}[{\cite[Proposition 3.7, 3.8]{GRPrisCris}}]\label{Prop: FullFaithfulRestr}
		The natural $\Zp$-linear $\otimes$ functor, 
		\[\mathbf{Vect}(\mathfrak{X}_{\prism})\to \mathbf{Vect}^{\mathrm{an}}(\mathfrak{X}_{\prism}), \]
		given by restriction is fully faithful. When $\mathfrak{X}=\Spf {\mathcal{O}}_K$, it induces an equivalence of categories. 
	\end{proposition}

	\begin{theorem}\label{Thm: PrisFCrysAbeSchem}
		Let $\mathcal{X}$ be a smooth scheme of finite type over $W$ and $\sA \to \mathcal{X} $ an abelian scheme, with $\widehat{\sA} \to \widehat{\mathcal{X}} $ denoting its $p$-adic completion. Then the relative prismatic cohomology  \( \mathrm{H}^1_{\prism}(\widehat{\sA}/\widehat{\mathcal{X}})\) is a prismatic crystal (in vector bundles, by definition) of rank $2\dim(\sA/\mathcal{X})$, where $\dim(\sA/\mathcal{X})$ denotes the relative dimension of $\sA$ over $\mathcal{X}$. Moreover, the pair
		\[
		( \mathrm{H}^1_{\prism}(\widehat{\sA}/\widehat{\mathcal{X}}), \, \varphi:\mathrm{H}^1_{\prism}(\widehat{\sA}/\widehat{\mathcal{X}})\to \mathrm{H}^1_{\prism}(\widehat{\sA}/\widehat{\mathcal{X}})  ) \] 
		forms a prismatic Dieudonn\'e crystal over $\widehat{\mathcal{X}}_{\prism}$. 
	\end{theorem}
	\begin{proof}
		This follows from \cite[Theorem 4.62]{AnschutzLeBrasPrisDieud}. 
	\end{proof}

	\subsection{Etale Realization Functors}

	Let \(\mathfrak{X}\) be a smooth \(p\)-adic formal \({\mathcal{O}}_K\)-scheme, and denote its rigid generic fibre by
	\(
	X := \mathfrak{X} \times_{\Spf {\mathcal{O}}_K} \mathrm{Spa} (K, {\mathcal{O}}_K).
	\)
	Denote by \(\mathbf{Loc}_{\Zp}(X)\) the category of \'etale \(\Zp\) local systems on \(X\), and by \(\mathbf{Vect}^{\varphi}(\mathfrak{X}_{\prism})[1/I_{\prism}]_p^{\wedge}\) the category of \emph{prismatic Laurent \(F\)-crystals} \cite[Definition 3.2]{BSPrisCrys}, whose objects consist of pairs \((\mathcal{L}, \varphi)\), with \(\mathcal{L}\) being an object in \(\mathbf{Vect}(\mathfrak{X}_{\prism})[1/I_{\prism}]_p^{\wedge}\) and \(\varphi_{\mathcal{L}}: \varphi^* \mathcal{L} \cong \mathcal{L}\) an isomorphism in \(\mathbf{Vect}(\mathfrak{X}_{\prism})[1/I_{\prism}]_p^{\wedge}\). Here, \(\varphi: \mathcal{O}_{\mathfrak{X}_{\prism}}[1/I_{\prism}] \to \mathcal{O}_{\mathfrak{X}_{\prism}}[1/I_{\prism}]\) is the Frobenius map induced by \(\varphi: \mathcal{O}_{\mathfrak{X}_{\prism}} \to \mathcal{O}_{\mathfrak{X}_{\prism}}\), and the category \(\mathbf{Vect}(\mathfrak{X}_{\prism})[1/I_{\prism}]_p^{\wedge}\), known as \emph{prismatic Laurent crystals}, is defined similarly to \(\mathbf{Vect}(\mathfrak{X}_{\prism})\), with the following change: for each prism \(\underline{A} = (A, I)\), the \(A\)-module \(M_A\) is replaced by an \(A[1/I_{\prism}]\)-module.  Both \(\mathbf{Vect}^{\varphi}(\mathfrak{X}_{\prism})[1/I_{\prism}]_p^{\wedge}\)  and \(\mathbf{Vect}(\mathfrak{X}_{\prism})[1/I_{\prism}]_p^{\wedge}\) admit the structure of an exact $\Zp$-linear $\otimes$-category. 
	
	Bhatt and Scholze establish in \cite[Corollary 3.8]{BSPrisCrys} a category equivalence: 
	\[
	\mathrm{T}_{\mathrm{\text{\'e}t}}: \mathbf{Vect}^{\varphi}(\mathfrak{X}_{\prism})[1/I_{\prism}]_p^{\wedge} \cong \mathbf{Loc}_{\Zp}(X),
	\]
	which is called an \emph{\'etale realization functor} in \cite{IKYTannak}. This functor is a bi-exact \(\Zp\)-linear \(\otimes\)-equivalence.
	
	Denote by \(\mathbf{Loc}_{\Zp}^{\cris}(X)\) the full subcategory of crystalline \(\Zp\) local systems on \(X\), and consider the (exact \(\Zp\)-linear \(\otimes\)) \'etale realization functor:
	\begin{equation}\label{Eq: EtaleRealFun1}
		\mathrm{T}_{\mathrm{\text{\'e}t}}: \mathbf{Vect}^{\mathrm{an}, \varphi}(\mathfrak{X}_{\prism}) \xrightarrow{\mathrm{restriction}} \mathbf{Vect}^{\varphi}(\mathfrak{X}_{\prism})[1/I_{\prism}]_p^{\wedge} \cong \mathbf{Loc}_{\Zp}(X).
	\end{equation}
	
	\begin{theorem}\label{Thm: BiExEquiv}
		Let \(\mathfrak{X}\) be a smooth \(p\)-adic formal scheme over \({\mathcal{O}}_K\). Then the \'etale realization functor in \eqref{Eq: EtaleRealFun1} induces a bi-exact \(\Zp\)-linear \(\otimes\)-equivalence:
		\begin{equation}\label{Eq: EtaleRealFun2}
			\mathrm{T}_{\mathrm{\text{\'e}t}}: \mathbf{Vect}^{\mathrm{an}, \varphi}(\mathfrak{X}_{\prism}) \to \mathbf{Loc}_{\Zp}^{\cris}(X).
		\end{equation}
		In particular, for each reductive group \(\sG\) over \(\Zp\), it induces a bi-exact \(\Zp\)-linear \(\otimes\)-equivalence:
		\[
		\sG\text{-}\mathbf{Vect}^{\mathrm{an}, \varphi}(\mathfrak{X}_{\prism}) \cong \sG\text{-}\mathbf{Loc}_{\Zp}^{\cris}(X).
		\]
	\end{theorem}
	
	\begin{proof}
		The assertion that \(\mathrm{T}_{\mathrm{\text{\'e}t}}\) is a category equivalence is proved in \cite[\S 4]{GRPrisCris} (see also \cite{DLMSPrisCris}), and the assertion that \(\mathrm{T}_{\mathrm{\text{\'e}t}}\) is bi-exact is proved in \cite[Proposition 2.22]{IKYTannak}.
	\end{proof}
	
	Let \(\mathbf{Loc}_{\Zp}^{\prism-\mathrm{gr}}(X) := \mathrm{T}_{\mathrm{\text{\'e}t}}(\mathbf{Vect}^{\varphi}(\mathfrak{X}_{\prism}))\) be the essential image of \(\mathbf{Vect}^{\varphi}(\mathfrak{X}_{\prism})\) under the \'etale realization functor \eqref{Eq: EtaleRealFun2}. It is a full exact \(\Zp\)-linear \(\otimes\)-subcategory of \(\mathbf{Loc}_{\Zp}^{\cris}(X)\), whose objects are called \emph{prismatically good reduction \(\Zp\) local systems} \cite[Definition 2.27]{IKYTannak}. Clearly, \(\mathrm{T}_{\mathrm{\text{\'e}t}}\) induces the following \emph{exact} \(\Zp\)-linear \(\otimes\)-equivalence:
	\begin{equation}\label{Eq: EtaleRealFun3}
		\mathrm{T}_{\mathrm{\text{\'e}t}}: \mathbf{Vect}^{\varphi}(\mathfrak{X}_{\prism}) \cong \mathbf{Loc}_{\Zp}^{\prism-\mathrm{gr}}(X).
	\end{equation}
	This, in turn, induces a faithful functor between \(\sG\)-objects for each reductive \(\Zp\)-group scheme \(\sG\):
	\begin{equation}\label{Eq: EtaleRealFun4}
		\mathrm{T}_{\mathrm{\text{\'e}t}}: \sG\text{-}\mathbf{Vect}^{\varphi}(\mathfrak{X}_{\prism}) \to \sG\text{-}\mathbf{Loc}_{\Zp}^{\prism-\mathrm{gr}}(X).
	\end{equation}
	
	Unlike the situation in Theorem \ref{Thm: BiExEquiv}, it is not a priori clear that \eqref{Eq: EtaleRealFun4} is an equivalence of categories due to the following reason: the category equivalence in \eqref{Eq: EtaleRealFun3} is not \emph{bi-exact} (albeit being \emph{exact}) because the quasi-inverse of this category equivalence is not exact. In fact, this is already the case for \(\mathfrak{X} = \Spf {\mathcal{O}}_K\). Indeed, though in this special case, the \(\Zp\)-linear \(\otimes\)-functor \(\mathbf{Vect}^{\varphi}(\Spf {\mathcal{O}}_K) \to \mathbf{Vect}^{\varphi, \mathrm{an}}(\Spf {\mathcal{O}}_K)\) is an equivalence of categories, its quasi-inverse is not exact. However, the authors in \cite{IKYTannak} manage to show that \eqref{Eq: EtaleRealFun4} is still an equivalence of categories.
	
	\begin{theorem}[{\cite[Theorem 2.28]{IKYTannak}}] \label{Thm: GMainEquiv}
		Let $\mathfrak{X}$ be a smooth $p$-adic formal scheme over ${\mathcal{O}}_K$ and $\sG$ a reductive group scheme over $\Zp$. 
		The faithful \'etale realization functor in \eqref{Eq: EtaleRealFun4} is an equivalence of categories:
		\begin{equation*}\label{Eq: EtaleRealFun5}
			\mathrm{T}_{\mathrm{\text{\'e}t}}: \sG\text{-}\mathbf{Vect}^{\varphi}(\mathfrak{X}_{\prism}) \cong \sG\text{-}\mathbf{Loc}_{\Zp}^{\prism-\mathrm{gr}}(X).
		\end{equation*}
	\end{theorem}

	\section{Local Systems on Shimura Varieties and prismatic $\mathcal{G}$-torsors}\label{S:ShVLoc}\label{LocSys+PrisTorsonShimuraVar}
	In this section, we introduce the prismatic $\mathcal{G}$\nobreakdash-torsor $\mathbb{J}_{\prism}$ over the integral model $\mathcal{S}$ of a Shimura variety. In the first two subsections, we recall facts on Shimura varieties to fix data and notation. We then discuss $\mathbb{J}_{\prism}$ in the remaining two subsections.
	\subsection{Integral models of Shimura Varieties of Hodge Type}\label{S: ShimVar}
	Our basic setup of Shimura varieties is the same as \cite[\S 7]{Yan23zip} and \cite[\S 3]{YanLocCon}. Specifically, we consider a Shimura variety $\mathrm{Sh}_{\mathsf{K}}=\mathrm{Sh}_{\mathsf{K}}(\mathbf{G}, \mathbf{X})$  attached to a Shimura datum $\big(\mathbf{G}, \mathbf{X})$ of Hodge type, where $\mathsf{K}=\mathsf{K}_0\mathsf{K}^p$, with $\mathsf{K}_0\subseteq \mathbf{G}(\mathbb{Q}_p)$, $\mathsf{K}^p\subseteq \mathbf{G}(\mathbb{A}_f^p)$ open compact subgroups. It is a smooth quasi-projective variety over a number field (a.k.a. the reflective field) $E=E(\mathbf{G}, \mathbf{X})$. Here that $\mathrm{Sh}_{\mathsf{K}}$ is of Hodge type means that it can be embedding into some Siegel Shimura variety; we will fix such an embedding shortly in \eqref{Eq:EmbedShiDat} below. Assume that the level structure $\mathsf{K}$ is hyperspecial at $p$. This means that we have $\mathsf{K}_0=\mathcal{G}(\Zp)$ for some reductive model $\mathcal{G}$ over $\Zp$ of $\mathbf{G}_{\mathbb{Q}_p}$. Fix a place $v$ of $E$ above $p$. Then the hyperspecial condition also implies that $E_v$ is unramified over $\mathbb{Q}_p$ and hence we have $\mathcal{O}_{E_v}= W$, where $W=W(\kappa)$ and $\kappa$ denotes the residue field $\mathcal{O}_{E_v}/p\mathcal{O}_{E_v}$. Let $\mathcal{S}=\mathcal{S}_{\mathsf{K}}$ be the Kisin-Vasiu integral model over $W$ of $\mathrm{Sh}_{\mathsf{K}}$ \cite{KisinIntegralModels}, and write $S=S_{\mathsf{K}}:=\mathcal{S}\otimes_{W}\kappa$ for the special fiber of $\mathcal{S}$. 
	
	We provide a more detailed description of the integral model $\mathcal{S}$ to fix data and notation. For this, we fix a morphism of unramified Shimura data as defined in \cite[\S 4.1]{IKYPrisReal},
	\begin{equation}\label{Eq:EmbedShiDat}
		\iota: (\mathbf{G}, \mathbf{X}, \sG) \to (\mathbf{GSp}(V_0), \mathbb{H}_{\mathrm{g}}^{\pm}, \mathbf{GSp}(\Lambda_0)),
	\end{equation}
	along with neat open compact subgroups,
	\[
	\mathsf{K} = \mathsf{K}_0\mathsf{K}^p \subseteq \sG(\mathbb{A}_f), \quad \mathsf{L} = \mathsf{L}_0\mathsf{L}^p \subseteq \mathbf{GSp}(V_0)(\mathbb{A}_f), \text{ where }
	\]
	\[
	\mathsf{K}_0 := \sG(\Zp), \quad \mathsf{L}_0 := \mathbf{GSp}(\Lambda_0)(\Zp), \quad \mathsf{K}^p \subseteq \sG(\mathbb{A}_f^p), \text{ and} \quad \mathsf{L}^p \subseteq \mathbf{GSp}(V_0)(\mathbb{A}_f^p).
	\]
	satisfying \(\iota(\mathsf{K}^p) \subseteq \mathsf{L}^p\). Then \(\mathrm{Sh}_{\mathsf{K}}\) resp. \(\mathrm{Sh}_{\mathsf{L}}\)  is a quasi-projective \(E\)-scheme resp.  \(\mathbb{Q}\)-scheme.
	
	For the Siegel Shimura datum \((\mathbf{GSp}(V_0), \mathbb{H}_{\mathrm{g}}^{\pm})\) and \(\mathsf{L}\) as specified above, its integral model \(\mathcal{S}_{\mathsf{L}}\) admits a moduli interpretation as a moduli space of principally polarized abelian varieties with additional structures; see, for example, \cite[Corollary 3.8]{MoonenInteModel} for more details. We denote by \(\mathcal{A}=\mathcal{A}_{\mathsf{L}}\) the universal abelian scheme over \(\mathcal{S}_{\mathsf{L}}\).
	The embedding of unramified Shimura data \(\iota\) induces a finite morphism of \(\mathcal{O}_{E_{v}}\)-schemes,
	\(
	\iota: \mathcal{S}_{\mathsf{K}} \to \mathcal{S}_{\mathsf{L}}.
	\)
	Denote by \(\mathcal{A}=\mathcal{A}_{\mathsf{K}}\) the pullback to \(\mathcal{S}_{\mathsf{K}}\) of  \(\mathcal{A}_{\mathsf{L}}\).
	\subsection{Hodge cocharacters}\label{S: HodgeCoch}
	Write $G=\mathcal{G}\otimes_{\Zp}\mathbb{F}_p$ for the special fiber of $\mathcal{G}$ and still denote by $G$ its base change to $\kappa$. Let $\mu: \mathbb{G}_{m, \kappa}\to G_{\kappa}=G$ be a representative of the reduction  $[\mu]_{\kappa}$ over $\kappa$ of the $\mathbf{G}(\mathbb{C})$-conjugacy class $[\mu]_{\mathbb{C}}$ determined by $(\mathbf{G}, \mathbf{X})$. Denote by $ P_{\pm}\subseteq G$ the opposite parabolic subgroups of $G$ defined by $ \mu $ and  $U_{\pm}\subseteq P_{\pm}$ the corresponding unipotent radicals; see for example \cite[\S 3.4]{YanLocCon} for the normalization.  
	
	We also fix a cocharacter $\tilde{\mu}: \mathbb{G}_{m, W}\to \sG_{W}=\sG$ over $W$, which lifts the character $\mu: \mathbb{G}_{m, \kappa}\to G$. Such a cocharacter, coming from the Shimura datum, is necessarily minuscule, in the sense that the weight decomposition $\mathbf{Lie}\sG=\oplus_{i\in \mathbb{Z}}(\mathbf{Lie}\sG)_i$ induced by $\tilde{\mu}$ is such that $(\mathbf{Lie}\sG)_i=0$ for all $i\geq 2$. 
	
	\subsection{Etale realization functor for {$\mathrm{Sh}_{\mathsf{K}}$}}\label{S:PrisTor}
	In this subsection, we record the \'etale realization functor for our Shimura variety \(\mathrm{Sh}_{\mathsf{K}}\) introduced above, following \cite[\S 2.3]{IKYPrisReal}. Note that in loc. cit. Shimura varieties of more general type and with more general level structures are allowed. In contrast, we are in the case where $\mathbf{G}=\mathbf{G}^c$ and $\sG$ is reductive over $\Zp$. 
	
	We retain the setting as in the previous subsection. As our $\K^p$ is fixed, we simply write $\sS$ for $\sS_{\K}$ when no confusions shall arise. For our fixed neat compact open subgroup \(\K=\K_0\K^p\subseteq \mathbf{G}(\mathbb{A}_p) \), 
	the map
	\[
	\mathrm{Sh}_{\K^p}:=\varprojlim_{\K_p'\subseteq \K_0}\mathrm{Sh}_{\K_p'\K^p} \to \mathrm{Sh}_{\K},
	\]
	viewed as a morphism in the pro\'etale site \((\mathrm{Sh}_{\K})_{\mathrm{{pro\text{\'e}t}}}\),  is a $\K_0=\sG(\Zp)$-torsor\footnote{Here \(\sG(\Zp)=\varprojlim_{i}\sG(\Zp/p^i\Zp)\) is a profinite group.}. This $\sG(\Zp)$-torsor induces the following \emph{integral \'etale realization functor}, 
	\begin{equation}\label{Eq: EtalRealShimVar}
		\omega_{\K, \mathrm{\text{\'e}t}}:  \mathbf{Rep}_{\Zp}(\sG)\to \mathbf{Loc}_{\Zp}(\mathrm{Sh}_{\K}). 
	\end{equation}
	It is given by sending an object $\rho: \sG\to \mathrm{GL}_{\Zp}(\Lambda)$ to the push forward of the $\K_0$-torsor $\mathrm{Sh}_{\K_0}$ along $\rho$ (viewed as a homomorphism of reductive $\Zp$-group schemes), i.e., 
	\[
	\omega_{\K, \mathrm{\text{\'e}t}}(\rho):=\omega_{\K, \mathrm{\text{\'e}t}}(\Lambda): = \mathrm{Sh}_{\K^p}\times^{\sG, \rho}\mathrm{GL}_{\Zp}(\Lambda)=\mathrm{Sh}_{\K^p}\times^{\sG, \rho}\Lambda,
	\]
	where the contracted product is the usual construction. The functor $\omega_{\K, \mathrm{\text{\'e}t}}$ has nice functorial properties. For example, given a morphism of unramified Shimura data, $\delta: (\mathbf{G}, \mathbf{X}, \sG)\to (\mathbf{G}', \mathbf{X}', \sG')$, together with neat compact open subgroups, $\K=\K_0\K^p\subseteq \mathbf{G}(\mathbb{A}_f)$ and $\K'=\K'_0\K'^p\subseteq \mathbf{G}'(\mathbb{A}_f)$ such that $\delta(\K)\subseteq \K'$, we have the following commutative diagram.
	\begin{equation}\label{Eq: EtalRealFunctorial}
		\xymatrix{
			\mathbf{Rep}_{\Zp}(\sG')\ar[r]^{\delta^*}\ar[d]^{\omega_{\K, \mathrm{\text{\'e}t}}}&\mathbf{Rep}_{\Zp}(\sG)\ar[d]^{\omega_{\K', \mathrm{\text{\'e}t}}}\\
			\mathbf{Loc}_{\Zp}(\mathrm{Sh}_{\K'})\ar[r]^{\delta^*}&\mathbf{Loc}_{\Zp}(\mathrm{Sh}_{\K}).
		}
	\end{equation}
	Here note that the reflex field $E'$ of $(\mathbf{G}', \mathbf{X}')$ is necessarily a subfield of $E$, hence the bottom pull-back functor is well-understood.

	Write $\Lambda_0\in \mathbf{Rep}_{\Zp}\mathcal{G}$ for the faithful $\Zp$-representation $\mathcal{G}\hookrightarrow \mathrm{GL}(\Lambda)$ induced by the fixed embedding $\iota: \mathcal{G}\to \mathbf{GSp}(\Lambda_0) $, and $\Lambda_0^*$ for its contragredient dual. Let $\mathrm{T}_0\subseteq \Lambda_0^{\otimes}$ be the set of tensors cutting out $\mathcal{G}$ from $\mathrm{GL}(\Lambda_0^*)$(\S~\ref{Notation}). Write $\mathrm{H}^1_{\mathbb{Q}_p}(\sA/ \mathrm{Sh}_{\K}): = \mathrm{R}^1\pi_{*}\mathbb{Q}_p$ and $\mathrm{H}^1_{\mathbb{Z}_p}(\sA/ \mathrm{Sh}_{\K}): = \mathrm{R}^1\pi_{*}\mathbb{Z}_p$, with $\pi: \mathcal{A}\to \mathrm{Sh}_{\K}$ the structure morphism.  Let $\mathrm{T}_{0, p, \mathrm{\text{\'e}t}}\subseteq \mathrm{H}^1_{\mathbb{Q}_p}(\sA/ \mathrm{Sh}_{\K})^{\otimes}$ be the set of tensors constructed in \cite[Lemma 3.1.6]{KimRapoportUniformisation} (taking the $p$-component of $t_{\text{\'et}}^{\mathrm{univ}}$ therein) from $\mathrm{T}_{0}\otimes 1\subseteq V_0^{\otimes}$. 
	
	The following proposition is extracted from \cite[\S 2.3]{IKYPrisReal}. 
	\begin{proposition}\label{Prop: TrivEtalTors}
		There is an isomorphism of $\Zp$-local systems inside \(\mathbf{Loc}_{\Zp}(\mathrm{Sh}_{\K})\),
		\[
		\omega_{\K, \mathrm{\text{\'e}t}}(\Lambda_0^*)\cong \mathrm{H}^1_{\Zp}(\sA/\mathrm{Sh}_{\K}),
		\]
		which carries the tensor $\omega_{\K, \mathrm{\text{\'e}t}}(\mathrm{T}_0)\otimes 1$ to $\mathrm{T}_{0, p, \mathrm{\text{\'e}t}}\subseteq \mathrm{H}^1_{\mathbb{Q}_p}(\sA/\mathrm{Sh}_{\K^p})^{\otimes}$. Consequently, we have
		\[
		\mathrm{T}_{0, p, \mathrm{\text{\'e}t}}\subseteq \mathrm{H}^1_{\Zp}(\sA/\mathrm{Sh}_{\K})^{\otimes}\subseteq  \mathrm{H}^1_{\mathbb{Q}_p}(\sA/\mathrm{Sh}_{\K})^{\otimes},
		\] 
		and an isomorphism (in pro\'etale topology) of $\sG(\Zp)$-torsors over $\mathrm{Sh}_{\K}$, 
		\begin{equation}\label{Eq: EtalTorsor}
			\mathrm{Sh}_{\K^p}\cong \underline{\mathrm{Isom}}\big(\Lambda_0^*\otimes_{\Zp}\underline{\Zp}, \mathrm{T}_0), (\mathrm{H}^1_{\Zp}(\sA/\mathrm{Sh}_{\K}), \mathrm{T}_{0, p, \mathrm{\text{\'e}t}})\big).
		\end{equation}
	\end{proposition}
	
	Write \(\widehat{\sA}, \widehat{\sS}, \) etc., for the $p$-adic completion of \(\sA, \sS, \) etc., and \(\widehat{\sA}_{\eta}, \widehat{\sS}_{\eta}, \) etc. for their rigid generic fibers. Applying Theorem \ref{Thm: BiExEquiv} and \ref{Thm: GMainEquiv}, we obtain the following equivalences of categories:
	\begin{equation}\label{Eq: GMainEquivShimVar}
		\mathrm{T}_{\mathrm{\text{\'e}t}}: \mathbf{Vect}^{\mathrm{an}, \varphi}(\widehat{\sS}_{\prism})\cong \mathbf{Loc}_{\Zp}^{\cris}(\widehat{\sS}_{\eta}), \quad	\sG\text{-}\mathrm{T}_{\mathrm{\text{\'e}t}}: \sG\text{-}\mathbf{Vect}^{\varphi}(\widehat{\sS}_{\prism}) \cong \sG\text{-}\mathbf{Loc}_{\Zp}^{\prism-\mathrm{gr}}(\widehat{\sS}_{\eta}).
	\end{equation}
	
	Denote by $\mathrm{Sh}_{\K}^{\mathrm{an}}$ for the adic analytification of $E_v$-scheme $\mathrm{Sh}_{\K}$. Then we have an open embedding $\widehat{\sS}_{\eta}\hookrightarrow \mathrm{Sh}_{\K}^{\mathrm{an}}$ of adic spaces over $\mathrm{Spa} E_v$. Note that we have the following analytification-restriction functor, which is an exact $\Zp$-linear $\otimes$-functor, 
	\begin{equation}\label{Eq: EtalRestr}
		\mathbf{Loc}_{\Zp}(\mathrm{Sh}_{\K})\xrightarrow{(\cdot)^{\mathrm{an}}} \mathbf{Loc}_{\Zp}(\mathrm{Sh}_{\K}^{\mathrm{an}})\xrightarrow{\text{restriction}} \mathbf{Loc}_{\Zp}(\widehat{\sS}_{\eta}).
	\end{equation}
	
	The \emph{integral \'etale realization functor for \(\widehat{\sS}_{\eta}\)},  
	\(
	\omega^{\mathrm{an}}_{\K, \mathrm{\text{\'e}t}}: \mathbf{Rep}_{\Zp}(\sG)\to \mathbf{Loc}_{\Zp}(\widehat{\sS}_{\eta}),
	\)
	is defined as the composition of $\omega_{\K, \mathrm{\text{\'e}t}}$ in \eqref{Eq: EtalRealShimVar} with the canonical functor in \eqref{Eq: EtalRestr}. 
	
	It follows from the rigid analytic GAGA theorem that under the functor in \eqref{Eq: EtalRestr}, the $\Zp$-local system $\mathrm{H}^1_{\Zp}(\sA/\mathrm{Sh}_{\K})$ is sent to the local system $\mathrm{H}^1_{\Zp}(\widehat{\sA}_{\eta}/\widehat{\sS}_{\eta})$. Hence applying this functor for the following isomorphism in \(\mathbf{Loc}_{\Zp}(\mathrm{Sh}_{\K})\) (with tensors understood as morphisms), 
	\[
	( \omega_{\K, \mathrm{\text{\'e}t}}(\Lambda_0^*), \omega_{\K, \mathrm{\text{\'e}t}}(\mathrm{T}_0) )\cong (\mathrm{H}^1_{\Zp}(\sA/\mathrm{Sh}_{\K}),  \mathrm{T}_{0, p, \mathrm{\text{\'e}t}} ),
	\]
	one obtains in $\mathbf{Loc}_{\Zp}(\widehat{\sS}_{\eta})$ an identification: 
	\[
	( \omega^{\mathrm{an}}_{\K, \mathrm{\text{\'e}t}}(\Lambda_0^*),   \omega^{\mathrm{an}}_{\K, \mathrm{\text{\'e}t}}(\mathrm{T}_0) )\cong (\mathrm{H}^1_{\Zp}(\widehat{\sA}_{\eta}/\widehat{\sS}_{\eta}),  \mathrm{T}^{\mathrm{an}}_{0, p, \mathrm{\text{\'e}t}} ),
	\]
	where $\mathrm{T}^{\mathrm{an}}_{0, p, \mathrm{\text{\'e}t}}$ denotes the image of $\mathrm{T}_{0, p, \mathrm{\text{\'e}t}}$ under the functor \eqref{Eq: EtalRestr}.
	Since it is known that $\mathrm{H}^1_{\Zp}(\widehat{\sA}_{\eta}/\widehat{\sS}_{\eta})$ is a crystalline \'etale local system on $\widehat{\mathcal{S}}_{\eta}$, one sees that the integral \'etale realization functor $\omega^{\mathrm{an}}_{\K, \mathrm{\text{\'e}t}}$ defines an object in $\mathcal{G}\text{-}\mathbf{Loc}^{\cris}_{\Zp}(\widehat{\sS}_{\eta})$.

	\subsection{The prismatic $\mathcal{G}$-torsor $\mathbb{J}_{\prism}$} 
	Write \(
	(\mathfrak{M}_{\widehat{\sA}/\widehat{\sS}},  \varphi_{\mathfrak{M}_{\widehat{\sA}/\widehat{\sS}}}),
	\) for the pair \( (\mathrm{H}^1_{\prism}(\widehat{\sA}/\widehat{\sS}), \varphi) \) as in Theorem \ref{Thm: PrisFCrysAbeSchem}. Let $K_0$ be an unramified extension of $E_v$, and let $K/K_0$ be a finite totally ramified extension. Given a point $x \in \mathrm{Sh}_{\K}(K)$, let $E_0(t) \in {\mathcal{O}}_{K_0}[t]$ be the minimal polynomial of a uniformizer $\pi \in K$. This yields a Breuil-Kisin prism \(\underline{\mathfrak{S}_0} := (\mathcal{O}_{K_0}[[t]], E_0(t))\) and thus an object $(\underline{\mathfrak{S}_0}, x)$ in $\widehat{\mathcal{S}}_{\prism}$. The evaluation of $(\mathfrak{M}_{\widehat{\sA}/\widehat{\sS}}, \varphi_{\mathfrak{M}_{\widehat{\sA}/\widehat{\sS}}})$ at $(\underline{\mathfrak{S}_0}, x)$ corresponds to the Breuil-Kisin module, say $\mathfrak{M}_x$, associated with the $\mathrm{Gal}(\overline{K}/K)$-stable $\mathbb{Z}_p$-lattice $\mathrm{H}^1_{\text{\'et}}(\mathcal{A}_x \otimes_K \overline{K}, \mathbb{Z}_p)$ of the crystalline representation $\mathrm{H}^1_{\text{\'et}}(\mathcal{A}_x \otimes_K \overline{K}, \mathbb{Q}_p)$ of $\mathrm{Gal}(\overline{K}/K)$. Recall that $\mathfrak{M}_x$ is obtained by applying the functor $\mathfrak{M}$ in \cite[Theorem 1.2.1]{KisinIntegralModels} (cf. \cite[Lemma 2.1.15]{KisinCrystalineRepresentations}).
	
	\begin{theorem}\label{Thm: PrisGTors} 
		\begin{enumerate}[(1).]
			\item Under the equivalence $\mathrm{T}_{\mathrm{\text{\'e}t}}$ in \eqref{Eq: GMainEquivShimVar}, we have the following identification in  $\mathbf{Loc}^{\cris}_{\Zp}(\widehat{\sS}_{\eta})$,
			\[
			\mathrm{T}_{\mathrm{\text{\'e}t}}((\mathfrak{M}_{\widehat{\sA}/\widehat{\sS}}, \varphi_{\mathfrak{M}_{\widehat{\sA}/\widehat{\sS}}}))= \mathrm{H}^1_{\Zp}(\widehat{\sA}_{\eta}/\widehat{\sS}_{\eta}). 
			\]
			\item  Denote by $\mathrm{T}_{\prism}\subseteq \mathfrak{M}_{\widehat{\sA}/\widehat{\sS}}^{\otimes} $ the set of Frobenius invariant prismatic tensors corresponding to $\mathrm{T}^{\mathrm{an}}_{0, p, \mathrm{\text{\'e}t}}\subseteq \mathrm{H}^1_{\Zp}(\widehat{\sA}_{\eta}/\widehat{\sS}_{\eta})^\otimes$ via $\mathrm{T}_{\text{\'et}}$. Then \begin{equation}\label{Eq: PrisGTor}
				\mathbb{J}_{\prism}:=	\underline{\mathrm{Isom}}((\Lambda_0^* \otimes_{\Zp} \mathcal{O}_{\prism}, \mathrm{T}_0\otimes 1),  (\mathfrak{M}_{\widehat{\sA}/\widehat{\sS}}, \mathrm{T}_\prism))
			\end{equation}
			is $\mathcal{G}$-torsor over $\mathcal{S}_{\prism}$. The Frobenius of $\mathfrak{M}_{\widehat{\sA}/\widehat{\sS}}$ induces a Frobenius map $\varphi_{\mathbb{J}_{\prism}}: \varphi^*\mathbb{J}_{\prism}\to \mathbb{J}_{\prism}$ such that $\varphi_{\mathbb{J}_{\prism}}[1/I_{\prism}]$ is an isomorphism.   
		\end{enumerate}
	\end{theorem}

	\begin{proof}
		
		(1). By \cite[Remark 9.16]{GRPrisCris}, we have 
		\(
		\mathrm{T}_{\mathrm{\text{\'e}t}}((\mathfrak{M}_{\widehat{\sA}/\widehat{\sS}}, \varphi_{\mathfrak{M}_{\widehat{\sA}/\widehat{\sS}}})_{\mathrm{tl}})= \mathrm{H}^1_{\Zp}(\widehat{\sA}_{\eta}/\widehat{\sS}_{\eta})_{\mathrm{tl}},
		\)
		where $(\cdot)_{\mathrm{tf}}$ in loc. cit. means the operation of taking quotient by $p$-power torsion. But in our case, this operation is unnecessary. Indeed, by Theorem \ref{Thm: PrisFCrysAbeSchem}, \( \mathfrak{M}_{\widehat{\sA}/\widehat{\sS}} \) is $p$-torsion free. Hence we have \( \mathrm{H}^1_{\prism}(\widehat{\sA}/\widehat{\sS})=\mathrm{H}^1_{\prism}(\widehat{\sA}/\widehat{\sS})_{\mathrm{tl}}.
		\)
		On the other hand, it is a well-known fact that pro\'etale locally the $\Zp$-local system $\mathrm{H}^1_{\Zp}(\widehat{\sA}_{\eta}/\widehat{\sS}_{\eta})$ is a free $\Zp$-module of rank $2\mathrm{g}$.

		(2). This is \cite[Theorem 2.12]{IKYPrisReal}; c.f. \cite[Theorem 2.28]{IKYTannak}.	\end{proof}

	We shall refer to either $\mathbb{J}_{\prism}$ or the pair $(\mathbb{J}_{\prism}, \varphi_{\mathbb{J}_{\prism}})$ as the \emph{prismatic $\mathcal{G}$-torsor} associated with $\sS=\sS_{\K}$.
	
	Recall that we have fixed a minuscule cocharacter $\tilde{\mu}: \mathbb{G}_{m, W}\to \sG_{W}$ over $W$ in \ref{S: HodgeCoch}. Following \cite[Definition 3.12]{IKYTannak}), we say a prismatic Frobenius $\sG$-torsor $(\sE, \varphi_{\sE})$ is \emph{of type $\tilde{\mu}$} if for each $\underline{R}=(R, I)$, there is a flat cover $\underline{A}=(A, IA)$ such that: 
	(1). $IA$ is principal, say $IA=(d)$;
	(2).  $\sE(\underline{A})\neq \emptyset$; equivalently, the restriction of $\sE$ on the slice category $\prism/\underline{A}:=(\widehat{\sS}/\underline{A})_{\prism} $ admits a trivialization $\beta: \sE|_{\prism/\underline{A}} \cong \sG_{\prism/\underline{A}}$ of $\sG$-torsors\footnote{We use the notation \(\sG_{\prism/\underline{A}}\) from \cite{IKYTannak}, which would correspond to $\mathcal{G}_{\tilde{\prism}/\underline{A}}$ in our notational system to be introduced later.};
	(3). The trivialized Frobenius $\sG_{{\prism/\underline{A}}[1/d]}\cong \sG_{{\prism/\underline{A}}[1/d]}$ under $\beta$, corresponds to the multiplication (on the left) by an element inside $\sG(A)\tilde{\mu}(d)\sG(A)$.

	\begin{theorem}\label{Thm: PrismticTypeMu}
		The prismatic $\sG$-torsor $(\mathbb{J}_{\prism}, \varphi_{\mathbb{J}_{\prism}})$ over $\widehat{\sS}$ is of type $\tilde{\mu}$.  
	\end{theorem}
	\begin{proof}
		This is part of \cite[Theorem 2.14]{IKYPrisReal}, which is proved in \cite[Corollary 3.7]{IKYPrisReal}.
	\end{proof}
	
	\section{Framed crystalline prismatization and reduction maps}\label{S:Prismatization}
	
	\subsection{Basic setup} \label{S:BasicSetup}
	Set \( K_0 = K = E_v=W[1/p]\), where $E_v$ is as in the previous section.  As in \ref{S: BKPrism}, attached to the \emph{base crystalline prism} \( \underline{W} = (W, p,\varphi) \) is the \emph{base Breuil-Kisin prism} \( \underline{\mathfrak{S}} = (\mathfrak{S}, E, \varphi_{\mathfrak{S}}) \), where \( \mathfrak{S} = \mathfrak{S}_W = W[[t]] \)  and \( E = E(t) = t + p \).  
	
	Fix a smooth \( p \)-adic formal scheme \( \mathfrak{X} \) over \( W \), with special fiber \( X := \mathfrak{X} \otimes_W \kappa \).  Denote \( \mathfrak{X}_{\prism_{\BK}} := (\mathfrak{X}/\underline{\mathfrak{S}})_{\prism} \) and \( \mathfrak{X}_{\prism_{\cris}} := (\mathfrak{X}/\underline{W})_{\prism} = (X/\underline{W})_{\prism} \).We introduce in this section the framed crystalline prismatization $\mathfrak{X}_{\tilde{\prism}_{\cris}}$ of $\mathfrak{X}$. When $\mathfrak{X}=\widehat{\mathcal{S}}$, the $p$‑adic completion of the $p$‑integral model $\mathcal{S}$ of a Shimura variety, this prismatization serves as the source space of the integral Frobenius period map named in the title of the present article.
	
	\subsection{Reduction map for sheaves over $\mathfrak{X}_{\BKprism}$}
	Note that there is an obvious morphism of prisms, 
	\begin{equation}\label{Eq:PrismModt}
		\underline{\mathfrak{S}}\xrightarrow{t\mapsto 0} \underline{W}. 
	\end{equation}
	Each object $\underline{R}$ in $\mathfrak{X}_{\prism_{\cris}}$ can be viewed as an object in $\mathfrak{X}_{\prism_{\BK}}$
	by precomposing with this morphism. This induces a continuous inclusion functor, $\iota_0: X_{\prism_{\cris}}\to \mathfrak{X}_{\prism_{\BK}}$ given by $(\underline{R}, x)\mapsto ( \underline{R}, x) $, whose composition with the structure map $\mathfrak{X}_{\BKprism}\to \mathfrak{X}_{\prism}$ is the structure morphism $\mathfrak{X}_{\crispris}\to \mathfrak{X}_{\prism}$. This functor induces the pullback functor on sheaves:  	
	\begin{equation} \label{Eq: ResToCrisPris}
		(\cdot)_{\prism_{\cris}}:	\mathbf{Shv}(\mathfrak{X}_{\prism_{\BK}})\to \mathbf{Shv}(X_{\prism_{\cris}}), \quad  \mathcal{F}\mapsto \mathcal{F}_{\prism_{\cris}}:=\iota_0^* \mathcal{F}.
	\end{equation}
	
	On the other hand, we have a continuous functor 
	\(
	\pi: \mathfrak{X}_{\prism_{\BK}}\to X_{\prism_{\cris}}, 
	\) 
	which is given by base change along the morphism $\underline{\mathfrak{S}}\to \underline{W}$ in \eqref{Eq:PrismModt}, i.e.,  
	\begin{equation}\label{Eq:PrismBC}
		(\underline{A},\, x: \Spf A/EA \to \widehat{\mathfrak{X}}) \mapsto (\underline{A/tA}:=(A/ tA, (p)), \,\bar{x}: \mathrm{Spec} A/(t, p)\to X), 
	\end{equation}
	giving a section of the inclusion function $\iota_0$, i.e., we have  $\pi\circ \iota_0\cong \mathrm{id}_{X_{\prism_{\cris}}}$. Denote by $\pi^*:\mathbf{Shv}(X_{\prism_{\cris}}) \to \mathbf{Shv}(\mathfrak{X}_{\prism_{\BK}})$ the induced map on sheaves. We have for each $\mathcal{F}\in \mathbf{Shv}(\mathfrak{X}_{\prism_{\BK}})$ a canonical reduction map: 
	\begin{equation}\label{Eq: RedMap}
		\mathcal{F} \to \pi^*\mathcal{F}_{\prism_{\cris}},
	\end{equation}
	whose evaluation on $(\underline{A}, x)\in \mathfrak{X}_{\BKprism}$ is given by the canonical reduction map $\mathcal{F}(\underline{A}, x)\to \mathcal{F}(\underline{A/tA}, \bar{x})$ induced by the morphism in \eqref{Eq:PrismBC}
	
	\subsection{Framed crystalline prismatization}
	\begin{definition}
		The \emph{framed crystalline prismatization} of $\mathfrak{X}$ over $\underline{W}$ is the site $ \mathfrak{X}_{\tilde{\prism}_\cris}$ of pairs $(\underline{R}, x: \Spf R\to \widehat{\mathfrak{X}})$, where $\underline{R}$ is a  prism over $\underline{W}$, and $x: \Spf R \to \widehat{\mathfrak{X}}$ is a morphism of $p$-adic formal schemes over $\Spf W$; morphisms in this category are morphisms between crystalline prisms preserving the sections $x\in \mathrm{Hom}_{\Spf W}(\Spf R, \widehat{\mathfrak{X}})$; it is equipped with the flat topology as that of $W_{\crispris}$.  
	\end{definition}
	A geometric definition of  $\mathfrak{X}_{\tilde{\prism}{\cris}}$ will be given in \S~\ref{S:LanguagueSpace}.
	
	\subsection{Crystalline prismatic reduction map}
	There is an obvious continuous functor from $\mathfrak{X}_{\tilde{\prism}_\cris}$ to $\mathfrak{X}_{\crispris}=X_{\crispris}$, which we call the \emph{crystalline prismatic reduction map of $\mathfrak{X}$}, 
	\begin{equation}\label{Eq:CrisPrisRedMap}
		\mathrm{Red}: \mathfrak{X}_{\tilde{\prism}_\cris} \to X_{\crispris}, \quad (\underline{R}, x)\mapsto (\underline{R}, \bar{x}).
	\end{equation}
	It is an example of $\partial \mathrm{Mod}_p$ map as in Diagram \eqref{RelBRedDiagmIntro}.
	One observes that if \(\mathfrak{X}\) is a \(\kappa\)-scheme (viewed as a formal scheme over \(W\)), then \(\mathfrak{X}_{\tilde{\prism}_\cris} = \emptyset\). Thus, in this case, the reduction map is not meaningful. 
	
	\subsection{Reduction maps for sheaves over $\mathfrak{X}_{\tilde{\prism}_{\cris}}$}
	As usual, for any crystalline prism $\underline{R}=(R, pR)$, denote by 
	\( 
	\underline{\mathfrak{S}_R}=(\mathfrak{S}_R, E \mathfrak{S}_R, \varphi_{\mathfrak{S}_R}),
	\)
	the Breuil-Kisin prism over $\underline{\mathfrak{S}}$, as in \S~\ref{S: BKPrism}.  The association 
	\begin{equation}\label{Eq:TilCris2BK}
		( \underline{R}, x: \Spf R\to \mathfrak{X})\mapsto (\underline{\mathfrak{S}_R}, x: \Spf \mathfrak{S}_R/E=\Spf R \to \mathfrak{X}) 
	\end{equation}
	defines a continuous functor 
	\(
	\iota: \mathfrak{X}_{\tilde{\prism}_{\cris}} \to \mathfrak{X}_{\prism_{\BK}},\)
	inducing via pullback functor between sheaves, 
	\begin{equation}\label{Eq: BKResCris2}
		(\cdot)_{\prism^{\BK}_{\cris}}=\iota^*:  \mathbf{Shv}(\mathfrak{X}_{\prism_{\BK}})\to \mathbf{Shv}(\mathfrak{X}_{\tilde{\prism}_{\cris}}), \quad  \mathcal{F}\mapsto \mathcal{F}_{\prism^{\BK}_{\cris}}:=\iota^* \mathcal{F}. 
	\end{equation}
	
	One verifies readily that the crystalline prismatic reduction map of $\mathfrak{X}$ defined in \eqref{Eq:CrisPrisRedMap} is equal to the composition $\pi \circ\iota$. Then applying the functor $(\cdot)_{\prism^{\BK}_{\cris}}$ for the reduction map in \eqref{Eq: RedMap}, we obtain a reduction map in $\mathfrak{X}_{\prism^{\BK}_{\cris}}$, 
	\[
	\mathrm{Red}^{\BK}_{\cris}: \mathcal{F}_{\prism^{\BK}_{\cris}} \to \iota^* \pi^* \mathcal{F}_{\prism_{\cris}}=\mathrm{Red}^* \mathcal{F}_{\crispris}.
	\]
	Its evaluation on a point $(\underline{R}, x)$ is given by the morphism 
	\(
	\mathcal{F}(\underline{\mathfrak{S}_R}, x)\to \mathcal{F}(\underline{R}, \bar{x}) 
	\)
	induced by the morphism $(\underline{\mathfrak{S}_R}, x)\to (\underline{R}, \bar{x})$ inside $\mathfrak{X}_{\prism_{\BK}}$. As for each object $(\underline{R}, x)$ of $\mathfrak{X}_{\tilde{\prism}_\cris}$, the evaluation of $\mathrm{Red}^*\mathcal{F}_{\crispris}$ is determined by its reduction $(\underline{R}, \bar{x})$ \footnote{With Definition \ref{Def:RedSpace}, we describe this by saying that the sheaf $\mathrm{Red}^*\mathcal{F}_{\crispris}$ \emph{is of reduction} along $\mathrm{Red}: \mathfrak{X}_{\tilde{\prism}_\cris}\to \mathfrak{X}_{\crispris}$.}, we abuse notation and write  $\mathcal{F}_{\crispris}$ for $\mathrm{Red}^* \mathcal{F}_{\crispris}$. Accordingly, the reduction map $\mathrm{Red}_{\crispris}^{\BK}$ takes the form, 
	\begin{equation}\label{Eq: BKCrisRedMap}
		\mathrm{Red}^{\BK}_{\cris}: \mathcal{F}_{\prism^{\BK}_{\cris}} \to  \mathcal{F}_{\crispris}.
	\end{equation}

	\section{$\Spaces$}\label{S:LanguagueSpace}
	
	Let $\mathfrak{X}$ be a fixed smooth $p$-adic formal scheme over $\Spf W$ as in the previous section.
	
	\subsection{Prismatic sites as sheaves}\label{S:SitesAsSheaves} 
	Set $W_{\prism}=(\Spf W)_{\prism}$. Unwinding definition, it is the site of prisms $(A, I)$, with $A$ a $W$-algebra, equipped with the flat topology. 
	
	One observes that the structure functor $\mathfrak{X}_{\prism}\to W_{\prism}$ can be interpreted as a sheaf over $W_{\prism}$, which associates each object $(A, I)$ in $W_{\prism}$ the set $\mathrm{Hom}_{\Spf W} (\Spf A/I, \mathfrak{X})$. Similarly,  if we set $W_{\BKprism}=(\Spf W/\underline{\mathfrak{S}})$ and $W_{\crispris}=(\Spf W/\underline{W})$, the sites $\mathfrak{X}_{\BKprism}$ and $\mathfrak{X}_{\crispris}=X_{\crispris}$ can be interpreted as sheaves over $W_{\BKprism}$ and $W_{\crispris}$ respectively. Such an observation leads us to view the sites $\mathfrak{X}_{\prism}$, $\mathfrak{X}_{\BKcrispris}$ and $X_{\crispris}$ as \emph{geometric spaces}. 
	
	In addition, we have the following identifications of fibered categories (equivalently of sheaves):  
	\[
	\mathfrak{X}_{\BKprism}=\mathfrak{X}_{\prism}\times_{W_{\prism}}W_{\BKprism},  \quad  \mathfrak{X}_{\crispris}=\mathfrak{X}_{\prism}\times_{W_{\prism}}W_{\crispris}
	\]
	over $W_{\BKprism}$ and $W_{\crispris}$ respectively. Here $W_{\crispris}, W_{\BKprism}$ can be viewed as slice categories of $W_{\prism}$, thus coming with structure morphisms to $W_{\prism}$, and the fiber products above are fiber product of categories (equivalently, pull back of sheaves since we know $\mathfrak{X}_{\prism}$ is a sheaf over $W_{\prism}$).
	
	Note that we have $W_{\crispris}=W_{\tilde{\prism}_{\cris}}$  and that the functor $\iota: W_{\crispris}\to W_{\BKprism}$ defined as in \eqref{Eq:TilCris2BK} (see also the next subsection) is a continuous functor. The identification of sheaves over $W_{\crispris}$, 
	\[
	\mathfrak{X}_{\tilde{\prism}_{\cris}}=\mathfrak{X}_{\BKprism}\times_{W_{\BKprism}, \iota} W_{\crispris},
	\]
	gives a \emph{geometric} definition of our framed crystalline prismatization $\mathfrak{X}_{\tilde{\prism}_{\cris}}$. 
	
	\subsection{Base $\Spaces$} \label{BaseSpace} 
	
	The observation in \S \ref{S:SitesAsSheaves} motivates the definitions introduced in this subsection to emphasize the geometric perspective of sites. 
	
	\begin{definition}
		A \emph{crystalline $W$-algebra} is defined as a flat $W$-algebra $R$ that admits a crystalline prism structure for some Frobenius lift $\varphi: R \to R$. A \emph{crystalline $\kappa$-algebra} is defined as a $\kappa$-algebra which admits a crystalline frame, i.e., there exists a crystalline prism $\underline{R}$ such that $\bar{R} = R/pR$. Write $W_{\cris}$ and $\kappa_{\cris}$ for the categories of crystalline $W$-algebras and $\kappa$-algebras, respectively.
	\end{definition}  
	\textbf{Convention/Definition:} A base $\Space$ attached to $W$ is one of the following sites:
	\[
	W_{\prism}, \quad W_{\BKprism},\quad  W_{\crispris}=\kappa_{\crispris}, \quad  W_{\cris}, \quad \kappa_{\cris},
	\] of which $W_{\prism}$ is the \emph{absolute} base $\Space$. \emph{Maps} between these base $\mathsf{spaces}$ are continuous functors.  
	
	\begin{definition}
		\begin{enumerate}
			\item A \(\mathsf{space}\) \(X\) over a base \(\mathsf{space}\) \(Y\) (i.e., one of the sites \(W_{\prism}, W_{\BKprism}, W_{\crispris}, W_{\cris}\), or \(\kappa_{\cris}\)) is a prestack of groupoids over \(Y\). We also say $X$ is a $\mathsf{space}$ fibered over $Y$. 
			
			\item If \(Y'\) is another base \(\mathsf{space}\) and \(X' \to Y'\) is a \(\mathsf{space}\) over \(X'\), a \emph{map} \(X \to X'\) is a continuous functor intertwining a base \(\mathsf{space}\) map $Y'\to Y$. We often omit the bases when they are clear in context.
			
			\item Given a map of $\Spaces$ $f: Y\to X$, we say $Y$ is a $\mathsf{subspace}$ of $X$ if the functor $f$ is faithful.
			
			\item An object in a $\mathsf{space}$ \(X\) is referred to as a \emph{point} of \(X\), and we write \( x \in X\) by abuse of notation. 
			
			\item Let \(X\) be a $\Space$ with  structure map \(f: X \to Z\) to the base $\Space$ $Z$. For each \(z\in Z\), set \(X(z)=f^{-1}(z)\) to be the site of points \(x\in X\) such that \(f(x) \cong z\), with the induced topology from \(X\). It is referred to as the \emph{fiber} of \(z\) in \(X\).
		\end{enumerate}
	\end{definition}
	
	\subsection{Prismatic sites as $\Spaces$}\label{S:SitesAsSpaces}
	
	The site \(\mathfrak{X}_{\prism}\) is a $\Space$ over \(W_{\prism}\); \(\mathfrak{X}_{\BKprism}\) is a $\Space$ over \(W_{\BKprism}\); \(\mathfrak{X}_{\tilde{\prism}_{\cris}}\) and \(X_{\crispris}\) are $\Spaces$ over \(W_{\crispris}\). Explicitly, \(\mathfrak{X}_{\BKprism}\) is a $\Space$ fibered over \(W_{\BKprism}\) with fibers given by 
	\((A, I) \mapsto \mathrm{Hom}_{\Spf W}(\Spf A/I, \mathfrak{X})\); 
	\(\mathfrak{X}_{\tilde{\prism}_\cris}\) is a $\Space$ fibered over \(W_{\crispris}\) 
	with fibers given by \(\underline{R} \mapsto \mathrm{Hom}_{\Spf W}(\Spf R, \mathfrak{X})\); 
	\(\mathfrak{X}_{\crispris} = X_{\crispris}\) is a $\Space$ fibered over \(W_{\crispris}\) with fibers given by \(\underline{R} \mapsto X(\bar{R})\). Henceforth, we use \(\bar{R}\) to denote \( R/pR \).
	
	The functors \(\iota_0: \mathfrak{X}_{\crispris} \to \mathfrak{X}_{\BKprism}\) and \(\pi: \mathfrak{X}_{\BKprism} \to \mathfrak{X}_{\crispris}\) are maps of $\Spaces$, with compatibility with the corresponding maps between their base $\Spaces$ arising from the functoriality for the structure morphism \(\mathfrak{X} \to \Spf W\). To be more precise, they are over the inclusion map \(\iota_0: W_{\prism_\cris} \to W_{\prism_\BK}\), and the map \(\pi: W_{\prism_{\BK}} \to W_{\prism_{\cris}}\) is defined by \( \underline{A} \mapsto \underline{A} \otimes_{\underline{\mathfrak{S}}} \underline{W} = (A/tA, (p))\). 
	
	The functors \(\iota: \mathfrak{X}_{\tilde{\prism}_{\cris}} \to \mathfrak{X}_{\BKprism}\) and \(\mathrm{Red}: \mathfrak{X}_{\tilde{\prism}_{\cris}} \to X_{\crispris}\) are maps of $\Spaces$, with underlying base $\Space$ maps given by:
	\[
	\iota: W_{\crispris}=W_{\tilde{\prism}_{\cris}} \to W_{\BKprism}, \quad \underline{R} \mapsto \underline{\mathfrak{S}_R},
	\]
	and \(\iota_0=\mathrm{id}: W_{\crispris}=W_{\tilde{\prism}_{\cris}}\to W_{\crispris}\) respectively. 
	
	The base $\Space$ maps $\pi, \iota_0$ and $\iota$ satisfy obvious relations: \(\pi \circ \iota_0 = \mathrm{id}_{W_{\crispris}}\) and \(\pi \circ \iota \cong \mathrm{id}_{W_{\crispris}}\).
	
	\subsection{Crystalline Prismatization Functors}\label{S:CrysPrisFunct}
	
	Let \((\mathbf{SmFSch}/W)\) denote the category of smooth $p$-adic formal schemes over $W$. One can define two functors from \((\mathbf{SmFSch}/W)\) to the category of \(W_{\crispris}\)-spaces: the \emph{framed crystalline prismatization functor} \(\mathfrak{X} \mapsto \mathfrak{X}_{\tilde{\prism}_\cris}\), and the \emph{crystalline prismatization functor} \(\mathfrak{X} \mapsto \mathfrak{X}_{\crispris}\).
	
	These functors are linked by a natural transformation induced by reduction. By definition, \(X_{\crispris} = (\mathrm{Res}_{\kappa/W} X)_{\tilde{\prism}_\cris}\), and the canonical map \(\mathrm{Red}: \mathfrak{X} \to \mathrm{Res}_{\kappa/W} X\) induces a map
	\(
	\mathrm{Red}: \mathfrak{X}_{\tilde{\prism}_\cris} \to X_{\crispris},
	\)
	which coincides with the reduction map in \eqref{Eq:CrisPrisRedMap}.

	\section{Base reduction diagrams}\label{S:ReductionTheory}
	
	\subsection{Base reduction diagram}
	We are interested in the following canonical \emph{reduction} maps  
	\begin{equation}\label{Eq:BaseRedMap}
		W_{\BKprism}\xrightarrow{\pi} W_{\crispris}\xrightarrow{\pi_{\crispris}}W_{\cris}\xrightarrow{\pi_{\cris}}\kappa_{\cris}.
	\end{equation}
	between base $\Spaces$,
	where $\pi_{\crispris}$ and $\pi_{\cris}$ are given by
	$(R, \varphi)\mapsto R$ \footnote{We shall refer to such maps as \emph{defrobenius maps}.} and $R\mapsto \bar{R}$ respectively. We refer to the sequence in \eqref{Eq:BaseRedMap} as the \emph{base reduction diagram}, with a relative version given in \S\ref{S:RelBasRedDiag}.

	\subsection{A notational convention}\label{S:subnotationconven} 
	
	Let us return to the setting of \S \ref{S:CrysPrisFunct}. Each $\mathfrak{X} \in (\mathbf{SmFSch}/W)$ also defines an object in $\mathbf{Shv}(W_{\cris})$, denoted by $\mathfrak{X}_{\cris}$, which associates to each $R \in W_{\cris}$ the set $\mathrm{Hom}_{\Spf W}(\Spf R, \mathfrak{X})$. Alternatively, it can be viewed as the restriction of $\mathcal{X}$ as a functor on $(\mathbf{Alg}/W)$ to $W_{\cris}$. We consider $\mathfrak{X}_{\cris}$ as a $W_{\cris}$-$\mathsf{space}$. Additionally, starting from a $\kappa$-scheme (or a $\kappa$-prestack), we obtain a $\kappa_{\cris}$-$\mathsf{space}$ defined as the restriction of $X$, viewed as a sheaf over $(\mathbf{Alg}/\kappa)$, to $\kappa_{\cris} \subseteq (\mathbf{Alg}/\kappa)$. 
	
	A conflict arises when we take $\mathcal{X} = X$: the object $X_{\cris} = \mathfrak{X}_{\cris} \in (W_{\cris} \textbf{-} \mathsf{space})$ is empty, whereas $X_{\cris} \in (\kappa_{\cris} \textbf{-} \mathsf{space})$ is typically non-empty. To avoid confusion, we establish the following convention:
	
	\begin{center}
		\fbox{
			\begin{minipage}{0.8\textwidth}
				\centering
				\textbf{Convention} \label{convention:crys} : \\
				For a prestack $X$ over $\kappa$, we set $X_{\cris} \in (W_{\cris} \textbf{-} \mathsf{space})$ as the pullback of $X_{\cris} \in (\kappa_{\cris} \textbf{-} \mathsf{space})$ via the base reduction map $\pi_{\cris}: W_{\cris} \to \kappa_{\cris}$.
			\end{minipage}
		}
	\end{center}

	\subsection{Relative base $\Spaces$}\label{S:RelBaseSpace}
	Let $\mathfrak{X}\in (\mathbf{SmFSch}/W)$, with special fiber $X$. 
	Apart from the $\Spaces$ $\mathfrak{X}_{\prism}, \mathfrak{X}_{\BKprism}, X_{\crispris}$, we also consider the following \(\mathsf{spaces}\) attached to \(\mathfrak{X}\):
	\begin{enumerate}
		\item The \(W_{\cris}\)-\(\mathsf{space}\) \(\mathfrak{X}_{\cris}\subseteq \mathfrak{X}_{\mathrm{ZAR}}\) of points \((R, x: \Spf R \to \mathfrak{X})\), with \(R \in W_{\cris}\).
		\item The \(W_{\cris}\)-\(\mathsf{space}\) \(X_{\widetilde{\cris}}\) of  points \((R, \bar{x}: \mathrm{Spec} \bar{R} \to X)\), with \(R \in W_{\cris}\).
		\item The \(\kappa_{\cris}\)-\(\mathsf{space}\) \(X_{\cris}\subseteq X_{\mathrm{ZAR}}\) of points \((\bar{R}, \bar{x}: \mathrm{Spec} \bar{R} \to X)\), with \(\bar{R} \in \kappa_{\cris}\).
	\end{enumerate}
	
	\textbf{Definition/Convention}: A \emph{relative base \(\mathsf{space}\)} (attached to $\mathfrak{X}$) is one of the following sites:
	\begin{equation}\label{Eq: RelBasSpace}
		\mathfrak{X}_{\BKprism}, \quad \mathfrak{X}_{\tilde{\prism}_\cris}, \quad X_{\crispris}, \quad \mathfrak{X}_{\cris}, \quad X_{\widetilde{\cris}}, \quad X_{\cris}.
	\end{equation}
	\emph{Maps} between these relative base $\Spaces$ are continuous functors.
	
	\subsection{Relative base reduction diagram}\label{S:RelBasRedDiag}
	The commutative diagram of canonical reduction maps between relative base $\mathsf{spaces}$—which we call the \emph{relative base reduction diagram}—already appeared in the introduction. Concretely, we refer to
	\begin{equation}\label{RelBRedDiagm}
		\text{The same diagram as in \eqref{BRedDiagmIntro}, but with } \mathfrak{X}\text{ and }X\text{ replacing }\mathcal{S}\text{ and }S\text{, respectively.}
	\end{equation}
	To avoid redundancy, we do not redraw the diagram from \eqref{BRedDiagmIntro}. Arrows labelled by \(\partial\,\mathrm{Mod}_p\) represent \emph{partial} reduction modulo \(p\) maps, and the arrow labelled by \( \mathrm{Mod}_p \) represents a \emph{full} reduction modulo \(p\) map. The maps in the diagram are clearly understood from the definitions of these \textsf{spaces}. For example,  \(\partial\mathrm{Mod}_p: \mathfrak{X} \to X_{\widetilde{\cris}}\) is given by sending \((R, x:\Spf R \to \mathfrak{X})\) to \((R, \bar{x}: \mathrm{Spec} \bar{R} \to X)\), while  \(\partial\mathrm{Mod}_p: X_{\widetilde{\cris}} \to X_{\cris}\) is given by sending \((R, \bar{x}: \mathrm{Spec} \bar{R} \to X)\) to \(\bar{x}: \mathrm{Spec} \bar{R} \to X\).
	
	As indicated in the introduction, if \(\mathrm{Red}: Y \to Z\) is one of the reduction maps in either \eqref{Eq:BaseRedMap} or \eqref{RelBRedDiagm}, each \(Z\)-\(\mathsf{space}\) \(\mathcal{F}\) can be viewed as a \(Y\)-\(\mathsf{space}\), with the restriction denoted by \(\mathcal{F}|_Y\) or \(\mathrm{Red}^*\mathcal{F}\) when there is a risk of confusion. Since $\mathrm{Red}$ is essentially surjective, viewing $\mathcal{F}$ as a $Y$-$\mathsf{space}$ does not lose information of $\mathcal{F}$.
	
	
	\begin{definition}
		Let $Y$ be one of the relative base $\mathsf{spaces} $ in \eqref{Eq: RelBasSpace}. A \emph{$\mathsf{space}$} over $Y$ is a prestack of groupoids over $Y$. Maps between $Y$-\textsf{spaces} are morphisms between prestacks over $Y$. 
	\end{definition} 
	
	\subsection{Reductions of $\Spaces$ and reductions of maps between $\Spaces$}
	We consider the following definitions as the most important ones in this article.
	
	\begin{definition}\label{Def:RedSpace}
		Let \( \mathrm{Red}: Y \to Z \) be one of the reduction maps in either \eqref{Eq:BaseRedMap} or \eqref{RelBRedDiagm}, and let \( \mathcal{F} \) be a \(\mathsf{space}\) over \( Y \). We say that \( \mathcal{F} \) is \emph{of reduction along} \( \mathrm{Red}: Y \to Z \) (by \( \mathcal{G} \)) if there exists a unique \( Z \)-\(\mathsf{space}\) \( \sG \) (up to isomorphism) and an isomorphism of \( Y \)-\(\mathsf{spaces}\) \( \mathcal{F} \cong \mathrm{Red}^* \mathcal{G} \). In this case, we will refer to \( \mathcal{G} \) as the \emph{reduction} of \( \mathcal{F} \) (along \( \mathrm{Red}: Y \to Z \)), and  identify \( \mathcal{F} \) and \( \mathcal{G} \), viewing \( \mathcal{F} \) as a \(\Space\) over \( Z \).
	\end{definition}
	
	\begin{definition}\label{Def:RedMap}
		Let \( \mathrm{Red}: Y \to Z \) be one of the reduction maps in either \eqref{Eq:BaseRedMap} or \eqref{RelBRedDiagm}, and let \( f: \mathcal{F}_1 \to \mathcal{F}_2 \) be a map of \( Y \)-\(\Spaces\). We say that \( f \) is \emph{of reduction along} \( \mathrm{Red} \) (by \( g \)) if there exist \( Z \)-\(\Spaces\) \( \mathcal{G}_1, \mathcal{G}_2 \) and a map \( g: \mathcal{G}_1 \to \mathcal{G}_2 \) such that \( \mathcal{F}_i \cong \mathrm{Red}^* \mathcal{G}_i \) for \( i = 1, 2 \), and these identifications induce a \(2\)-isomorphism \( f \simeq \mathrm{Red}^* g \). In this case, we refer to \( g \) as the \emph{reduction} of \( f \) along \( \mathrm{Red} \), and may identify \( f \) with \( g \), viewing \( f \) as a map of \( Z \)-\(\Spaces\).
	\end{definition}

	\section{Group {$\Spaces$} attached to $\mathcal{G}$}\label{S:LoopGp}
	We retain the basic setup as in \S \ref{S:BasicSetup}, and let $\sG, \tilde{\mu}: \mathbb{G}_{m, W}\to \mathcal{G}_W$ be as in \S \ref{S: HodgeCoch}.
	
	\subsection{Prismatic loop groups}\label{S:PrisLoopGrp}
	Let $H$ be an affine smooth group scheme over $W$. We set $H_{\prism}$, $H_{\tilde{\prism}}=\mathcal{L}_{\prism}^+H$, and $\mathcal{L}_\prism H=H_{\tilde{\prism}[1/I_{\prism}]}$, and $\mathcal{L}^1_{\prism}H$ for the presheaves over $W_{\prism}$, by associating $(A, I)$ with the set $H(A/I)$, $H(A)$, $H(A[1/I_A])$, and the kernel of the canonical reduction map $\mathcal{L}_{\prism}^+H \to H_{\prism}$ respectively. In fact, they are group \emph{sheaves} over $W_{\prism}$ in the flat topology, as can derived from the fact that $\mathcal{O}_{\mathfrak{X}_\prism}$, $\overline{\mathcal{O}}_{\mathfrak{X}_{\prism}}$, and $\mathcal{O}_{\mathfrak{X}_{\prism}}[1/I_{\prism}]$ are sheaves. We refer to $\mathcal{L}_{\prism}^+H$ as the \emph{positive prismatic loop group of $H$}, and $\mathcal{L}_{\prism} H$ as the \emph{prismatic loop group of $H$}.  
	
	We shall write $H_{\prism_{\BK}}$, $\mathcal{L}_{\BKprism}^+ H =H_{\tilde{\prism}_{\BK}}$, $\mathcal{L}_{\BKprism}H=H_{\tilde{\prism}_{\BK}[1/E]}$, and $\mathcal{L}^1_{\BKprism}H$ for their restriction on $W_{\prism_{\BK}}$, and $H_{\prism_{\cris}}$, $\mathcal{L}_{\crispris}^+ H =H_{\tilde{\prism}_{\cris}}$, $\mathcal{L}_{\crispris}H=H_{\tilde{\prism}_{\cris}[1/p]}$, and $\mathcal{L}^1_{\crispris}H$ for their restriction on $W_{\prism_{\cris}}$ respectively. Similarly, we write $H_{\BKcrispris}, \mathcal{L}^+_{\BKcrispris}H$, $\mathcal{L}_{\BKcrispris}H$ for the pullback of $H_{\prism_{\BK}}$, $\mathcal{L}_{\BKprism}^+ H$, $\mathcal{L}_{\BKprism}H$ along the map $\iota: W_{\crispris}\to W_{\BKprism}$ between base $\Spaces$\footnote{Warning: here we do not define $\mathcal{L}^1_{\BKcrispris}H$ as the pullback along $\iota$ of $\mathcal{L}^1_{\BKprism}H$. Its definition will be given in \S~\ref{S:RedofBKPerMap}.}.
	
	The following lemma is immediate, so we omit its proof.
	\begin{lemma}
		The $W_{\BKprism}$-$\Space$ $H_{\BKprism}$ is of reduction along the composition $ W_{\BKprism}\xrightarrow{\pi} W_{\crispris}\xrightarrow{\pi_{\crispris}} W_{\cris}$. The $W_{\crispris}$-$\Space$ $H_{\tilde{\prism}_{\cris}}$ and $H_{\tilde{\prism}_{\cris}[1/p]}$ are of reduction along the canonical map $W_{\crispris}\to W_{\cris}$. Moreover, the $W_{\crispris}$-$\Space$ $H_{\crispris}$ is of reduction along the composition $W_{\crispris}\xrightarrow{\pi_{\crispris}}W_{\cris}\xrightarrow{\pi_{\cris}}\kappa_{\cris}$.  
	\end{lemma}

	Let $\widehat{H}$ denote the $p$-adic completion of $H$, which is $p$-adic formal (group) scheme over $\Spf W$. Then the prismatizations $\hat{H}_{\tilde{\prism}_{\cris}}$ and $\hat{H}_{\crispris}$ of $\hat{H}$ defined in \S \ref{S:Prismatization} and the sheaves $H_{\tilde{\prism}_{\cris}}$  and $H_{\prism_{\cris}}$ defined above can be identified as $W_{\crispris}$-$\Spaces$. Indeed, for each crystalline prism $\underline{A}$ we have the identification $\mathrm{Hom}_{\Spf W}(\Spf A, \widehat{H}) = \mathrm{Hom}_{\mathrm{Spec} W}(\mathrm{Spec} A, H)$ as $A$ is $p$-adically complete. 
	
	From now on, we consider the case $H=\sG=\sG_{W}$, the base change to $W$ of a $\Zp$-reductive group scheme $\mathcal{G}$ as in \S\ref{S: HodgeCoch}.

	\subsection{The {$\Spaces$} {$\mathbf{C}^{\tilde{\mu}}_{\tilde{\prism}_\cris}$} and {$\mathbf{C}^{\tilde{\mu}}_{\BKcrispris}$} over {$W_{\crispris}$}}
	Denote by $(\mathbf{C}_{\tilde{\prism}_{\BK}}^{\tilde{\mu}})^{\mathrm{p}}$ the presheaf on $W_{\prism_{\BK}}$ given by associating each object $(A, EA)$ in $W_{\prism_{\BK}}$ the subset 
	\(
	\mathcal{G}(A)\tilde{\mu}(E)\mathcal{G}(A)
	\) of $\sG_{\tilde{\prism}_{\BK}[1/E]}(\underline{A})=\mathcal{G}(A[1/E])$, 
	and $\mathbf{C}_{\tilde{\prism}_{\BK}}^{\tilde{\mu}}$ its sheafification in the flat topology. We have the following inclusions of \( W_{\BKprism} \)-\(\Spaces\),
	\begin{equation}\label{Eq:BKCellmutil}
		(\mathbf{C}_{\tilde{\prism}_{\BK}}^{\tilde{\mu}})^{\mathrm{p}}\subseteq \mathbf{C}_{\tilde{\prism}_{\BK}}^{\tilde{\mu}}\subseteq \mathcal{L}_{\BKprism}\mathcal{G}.
	\end{equation}
	
	Applying $\iota_0^*$ in \eqref{Eq: ResToCrisPris} and $\iota^*$ in \eqref{Eq: BKResCris2} to \eqref{Eq:BKCellmutil} yields the following inclusions of $W_{\crispris}$-$\Spaces$:
	\[
	(\mathbf{C}_{\tilde{\prism}_{\cris}}^{\tilde{\mu}})^{\mathrm{p}}
	\subseteq
	\mathbf{C}_{\tilde{\prism}_{\cris}}^{\tilde{\mu}}
	\subseteq
	\mathcal{L}_{\crispris}\mathcal{G},
	\qquad
	(\mathbf{C}_{\prism_{\cris}^{\BK}}^{\tilde{\mu}})^{\mathrm{p}}
	\subseteq
	\mathbf{C}_{\prism_{\cris}^{\BK}}^{\tilde{\mu}}
	\subseteq
	\mathcal{L}_{\BKcrispris}\mathcal{G}.
	\]

	The next lemma will be used in our application later. 
	\begin{lemma}  \label{Lem: BKFrob}
		Given $\underline{A}\in W_{\prism_{\BK}}$ and $g\in \mathcal{L}_{\BKprism}\mathcal{G}(\underline{A})$. If there exists a flat cover $\underline{A}\to \underline{B}$ such that the image of $g$ in $\mathcal{L}_{\BKprism}\mathcal{G}(\underline{B})$ lies in $\mathbf{C}^{\tilde{\mu}}_{\tilde{\prism}_{\BK}}(\underline{B})$, we must have $g\in \mathbf{C}_{\tilde{\prism}_{\BK}}^{\tilde{\mu}}(\underline{A})$. 
	\end{lemma}
	\begin{proof}
		This follows from the fact that $\mathcal{L}_{\BKprism}\mathcal{G}$ is a sheaf and $\mathbf{C}_{\tilde{\prism}_{\BK}}^{\tilde{\mu}}\subseteq \mathcal{L}_{\BKprism}\mathcal{G}$ is a subsheaf. 
	\end{proof}
	
	Since our reductive group $\sG$ is defined over $\Zp$, for each object $(A, I)$ in $W_{\prism}$, the Frobenius map $\varphi_A: A \to A$ is a $\Zp$ homomorphism. Thus, the Frobenius $\varphi: \mathcal{O}_{\prism} \to \mathcal{O}_{\prism}$ induces the Frobenius map
	\(
	\varphi: \mathcal{L}_{\prism}\mathcal{G} \to \mathcal{L}_{\prism}\mathcal{G}. 
	\)
	There is a natural $\varphi$-conjugation action of $\mathcal{L}^+_{\BKprism}\mathcal{G}$ on $\mathcal{L}_{\BKprism}\mathcal{G}$:
	\[
	\mathcal{L}_{\BKprism}\mathcal{G}\times \mathcal{L}^+_{\BKprism}\mathcal{G}\to \mathcal{L}_{\BKprism}\mathcal{G}, \quad h\cdot g= g^{-1}h\varphi(g),
	\]
	and denote by $[\mathcal{L}_{\BKprism}\mathcal{G}/\mathrm{Ad}_{\varphi} \mathcal{L}^+_{\BKprism}\mathcal{G}]$ the resulting quotient stack over $W_{\BKprism}$. 
	Clearly,  $(\mathbf{C}_{\tilde{\prism}_{\BK}}^{\tilde{\mu}})^{\mathrm{p}}$, and hence also $\mathbf{C}^{\tilde{\mu}}_{\tilde{\prism}_{\BK}}$, are stable under the $\varphi$-conjugation action of $\mathcal{L}^+_{\BKprism}\mathcal{G}$ on $\mathcal{L}_{\BKprism}\mathcal{G}$. In particular,  we have an inclusion of stacks over $W_{\BKprism}$, \([\mathbf{C}_{\tilde{\prism}_\BK}^{\tilde{\mu}}/\mathrm{Ad}_{\varphi}\mathcal{L}^+_{\prism_\BK}\mathcal{G}]\subseteq [\mathcal{L}_{\prism_\BK}\mathcal{G}/\mathrm{Ad}_{\varphi} \mathcal{L}^+_{\prism_\BK}\mathcal{G}].  
	\)
	Similarly, we have the following inclusion of stacks over $W_{\crispris}$,
	\begin{align}[\mathbf{C}_{\tilde{\prism}_{\cris}}^{\tilde{\mu}}/\mathrm{Ad}_{\varphi}\mathcal{L}^+_{\crispris}\mathcal{G}]\subseteq [\mathcal{L}_{\crispris}\mathcal{G}/\mathrm{Ad}_{\varphi} \mathcal{L}^+_{\crispris}\mathcal{G}], \quad [\mathbf{C}_{\BKcrispris}^{\tilde{\mu}}/\mathrm{Ad}_{\varphi}\mathcal{L}^+_{\BKcrispris}\mathcal{G}]\subseteq [\mathcal{L}_{\BKcrispris}\mathcal{G}/\mathrm{Ad}_{\varphi} \mathcal{L}^+_{\BKcrispris}\mathcal{G}].
	\end{align}

	
	\subsection{Breuil-Kisin loop groups}\label{S:BKLpGp} 
	Define the sheaves $\mathcal{L}_{\BK}^+\mathcal{G}$ and $\mathcal{L}_{\BK}\mathcal{G} $ over $(\mathbf{Alg}/W)$ as follows: $\mathcal{L}_{\BK}^+\mathcal{G}(R) = \mathcal{G}(\mathfrak{S}_R)$ and $\mathcal{L}\mathcal{G}(R) = \mathcal{G}(\mathfrak{S}_R[1/E])$. Define $\mathcal{L}_{\BK}^1\mathcal{G}$ to be the kernel of the canonical reduction modulo $E$ map $\mathcal{L}_{\BK}^+\mathcal{G} \to \mathcal{G}$, and let $\mathbf{C}_{\BK}^{\tilde{\mu}} = \mathbf{C}_{\BK}^{\tilde{\mu}(E)}$ be the subsheaf of $\mathcal{L}_{\BK}\mathcal{G}$ on $(\mathbf{Alg}/W)$ defined by the sheafification of the functor
	\( R\mapsto \mathcal{L}^+_{\BK}\mathcal{G}(R) \tilde{\mu}(E) \mathcal{L}^+_{\BK}\mathcal{G}(R)\). We use the same notions to denote the restriction to $W_{\cris}\subseteq (\mathbf{Alg}/W)$ of these sheaves defined above.  We refer to $\mathcal{L}_{\BK} \mathcal{G}$ as the \emph{Breuil-Kisin loop group} of $\mathcal{G}$ and $\mathcal{L}^+_{\BK} \mathcal{G}$ as the \emph{positive Breuil-Kisin loop group} of $\mathcal{G}$.
	
	The next lemma and its analogue, Lemma \ref{Lem:CrisLpRed}, are immediate; proofs are omitted.
	\begin{lemma} \label{Lem:BKLpGpRed}The $W_{\crispris}$-$\Spaces$  $\mathbf{C}_{\prism_{\cris}^{\BK}}^{\tilde{\mu}}$ is of reduction by $\mathbf{C}^{\tilde{\mu}}_{\BK}$ along the the defrobenius map $W_{\crispris}\to W_{\cris}$. For each $\epsilon\in \{+, \emptyset, 1\}$, $\mathcal{L}^{\epsilon}_{\BKcrispris}\mathcal{G}$ is of reduction  by $\mathcal{L}_{\BK}^{\epsilon}\mathcal{G}$ along $W_{\crispris}\to W_{\cris}$.
	\end{lemma}
	
	\subsection{Crystalline loop groups}
	Define the sheaves $\mathcal{L}_{\cris}^+\mathcal{G}$ and $\mathcal{L}_{\cris}\mathcal{G} $ over $(\mathbf{FlatAlg}/W)$ \footnote{The flatness restriction is only needed for the definition of $\mathcal{L}_{\cris}\mathcal{G}$ where one needs to invert $p$.} as the following functors,
	\(
	R\mapsto \mathcal{G}(R)\) and \( R\mapsto \mathcal{G}(R[1/p])
	\)
	respectively. Define $\mathcal{L}^1_{\cris}\mathcal{G}$ to be the kernel of the canonical reduction modulo $p$ map \(
	\mathrm{Red}: \mathcal{L}_{\cris}^+\mathcal{G} \to \mathrm{Res}_{\kappa/W}G
	\). Define the subsheaf $\mathbf{C}^{\tilde{\mu}}_{\cris}\subseteq \mathcal{L}_{\cris}\mathcal{G}$ to be the sheafification of the presheaf 
	\(R \mapsto  \mathcal{L}_{\cris}^+\mathcal{G}(R) \tilde{\mu}(p) \mathcal{L}_{\cris}^+\mathcal{G}(R)\). We use the same notions to denote the restriction to $W_{\cris}\subseteq (\mathbf{FlatAlg}/W)$ of these sheaves defined above. We refer to $\mathcal{L}_{\cris} \mathcal{G}$ as the \emph{crystalline loop group} of $\mathcal{G}$ and $\mathcal{L}^+_{\cris} \mathcal{G}$ as the \emph{positive crystalline loop group} of $\mathcal{G}$.
	
	We need to explain our seemingly strange definition of $\mathcal{L}_{\cris}^+\mathcal{G}$. For this, let us consider the embedding from 
	\(
	(\mathbf{PerfAlg}/\kappa)\) to $W_{\cris}$ given by \(\bar{R}\mapsto W(\bar{R}).
	\)
	Then the restriction to $ (\mathbf{PerfAlg}/\kappa)$ of $\mathcal{L}_{\cris}^+ \mathcal{G}$ resp. $\mathcal{L}_{\cris}\mathcal{G}$ becomes the classical mixed characteristic positive loop group resp. loop group attached to $\mathcal{G}$. 
	
	The next lemma is parallel to Lemma \ref{Lem:BKLpGpRed} and requires no proof.
	\begin{lemma} \label{Lem:CrisLpRed}
		The $W_{\crispris}$-$\Spaces$ $\mathbf{C}_{\tilde{\prism}_{\cris}}^{\tilde{\mu}}$ is of reduction by $\mathbf{C}^{\tilde{\mu}}_{\cris}$ along the defrobenius map $W_{\crispris}\to W_{\cris}$. For each $\epsilon\in \{+, \emptyset, 1\}$, $\mathcal{L}^{\epsilon}_{\tilde{\prism}_\cris}\mathcal{G}$ is of reduction by $\mathcal{L}_{\cris}^{\epsilon}\mathcal{G}$ along $W_{\crispris}\to W_{\cris}$.
	\end{lemma}

	\section{A representability result}\label{S:RepProof}
	In this section we let $\mathcal{G}, G=\mathcal{G}_\kappa$, $\tilde{\mu}:\mathbb{G}_{m, W}\to \mathcal{G}$ and $\mu=\tilde{\mu}\otimes_W\kappa$ be as in \S \ref{S: HodgeCoch}. 
	
	Recall that the loop group $\mathcal{L}G$ and the positive loop group $\mathcal{L}^+G$ of $G$ are the group functors on $(\mathbf{Alg}/\kappa)$, given by 
	\( R\mapsto G(R((t))\) and \( R\mapsto G(R[[t]]) \)
	respectively. Set $\mathcal{L}^{1}G$ to be the kernel of the canonical projection $\mathcal{L}^+G\to G$, and $\mathcal{C}^{\mu}\subseteq \mathcal{L}G$ to be the sheafification of the presheaf 
	\[
	R\mapsto \mathcal{L}^+G(R) \mu(t) \mathcal{L}^+G(R)\subseteq \mathcal{L}G,
	\]
	and set
	$ \prescript{}{1}{\mathcal{C}^{\mu}_1}:= \mathcal{L}^1G \backslash \mathcal{C}^{\mu}/\mathcal{L}^1G$ to be the (set-valued) quotient sheaf of $\mathcal{C}^\mu$ by the actions of $\mathcal{L}^1 G\times \mathcal{L}^1G$ by  multiplication on the left and right componentwise. Set $\mathsf{E}_{\mu}\subseteq P_-\times P_+ $ to be the zip group attached to $\mu$ consisting of elements $(p_-, p_+)$ in $P_-\times P_+$ such that $p_-$ and $p_+$ share the same Levi component (see \cite[Definition 4.1]{Yan23zip}), and consider the action of $\mathsf{E}_{\mu}$ on $G$ given by $(g, h)\cdot (p_-, p_+)\mapsto (p_-g, p_+ h)$. Denote by $\mathsf{E}_{\mu}\backslash G^2$ the resulting quotient scheme. 
	Recall:
	\begin{theorem}[{\cite[Theorem 4.11]{Yan23zip}}]
		We have a canonical isomorphism of sheaves over $(\mathbf{Alg}/k)$:
		\[\prescript{}{1}{\mathcal{C}_1^{\mu}} \cong \mathsf{E}_{\mu}\backslash G^2, \quad g \mu(t) h \mapsto \big(\bar{g}^{-1}, \bar{h}\big),\]
		where $\bar{g} \in G(R)$ denotes the image in $G(R)$ of $g\in \mathcal{L}^+G(R)$ along the projection $\mathcal{L}^+G \to G$. 
	\end{theorem}
	
	Set $\prescript{}{1}{\mathbf{C}_1^{\tilde{\mu}}}$ to be the sheafification of the presheaf $(\prescript{}{1}{\mathbf{C}_1^{\tilde{\mu}}})^{\mathrm{p}}$ on $(\mathbf{FlatAlg}/W)$ given by 
	\[
	R \mapsto \doublecoset{\mathcal{L}_{\cris}^1\mathcal{G}(R)}{\mathcal{L}_{\cris}^+\mathcal{G}(R)\tilde{\mu}(p)\mathcal{L}_{\cris}^+\mathcal{G}(R)}{\mathcal{L}_{\cris}^1\mathcal{G}(R)}, 
	\]
	and write $(\prescript{}{1}{\mathbf{C}_1^{\tilde{\mu}}})_{\cris}$ for its restriction on $W_{\cris}\subseteq (\mathbf{FlatAlg}/W)$. The following theorem is an enhancement of \cite[Theorem 4.11]{Yan23zip}, where the functor $\prescript{}{1}{\mathbf{C}_1^{\tilde{\mu}}}$ is simply the restriction on $(\mathbf{PerfAlg}/\kappa)\subseteq W_{\cris}$ of our $\prescript{}{1}{\mathbf{C}_1^{\tilde{\mu}}}$ defined here.
	\begin{theorem}\label{Thm:ExtdRepthm}
		For each $R \in (\mathbf{FlatAlg}/W)$, there are canonical bijections, 
		\[(\prescript{}{1}{\mathbf{C}_1^{\tilde{\mu}}})^{\mathrm{p}}(R) \cong \mathsf{E}_{\mu}(\bar{R})\backslash G^2(\bar{R}), \quad g\tilde{\mu}(p)h \mapsto \big(\bar{g}^{-1}, \bar{h}\big),\]
		where $\bar{g} \in G(\bar{R})$ is the canonical reduction modulo $p$ of $g$. This leads to a canonical isomorphism 
		\(\prescript{}{1}{\mathbf{C}_1^{\tilde{\mu}}} \cong \mathrm{Res}_{\kappa/W}(\mathsf{E}_{\mu}\backslash G^2)\) of sheaves on $(\mathbf{FlatAlg}/W)$,
		and hence a canonical isomorphism of $W_{\cris}$-$\mathsf{spaces}$, 
		\begin{equation}\label{Eq:2CC1Isom}
			(\prescript{}{1}{\mathbf{C}_1^{\tilde{\mu}}})_{\cris}\cong  (\prescript{}{1}{\mathcal{C}_1^{\mu}})_{\cris}. 
		\end{equation}
		Said in another way, $ (\prescript{}{1}{\mathbf{C}_1^{\tilde{\mu}}})_{\cris}$ is of reduction along $W_{\cris} \to \kappa_{\cris}$ by $(\prescript{}{1}{\mathcal{C}_1^{\mu}})_{\cris}$; see \S \ref{S:subnotationconven} for our convention of the notation $(\prescript{}{1}{\mathcal{C}_1^{\mu}})_{\cris}$ as a $W_{\cris}$-$\Space$. 
	\end{theorem}

	\begin{proof}
		The proof follows a similar approach as in \cite[Theorem 4.11]{Yan23zip}. We summarize the essential components in Lemma \ref{Lem:KeyLemma} below. 
	\end{proof}
	\begin{remark} 
		Theorem \ref{Thm:ExtdRepthm} implies that the sheaf $\prescript{}{1}{\mathbf{C}_1^{\tilde{\mu}}}$ depends only on the mod $p$ cocharacter $\mu$ and not on the choice of its lift $\tilde{\mu}$.
	\end{remark}

	Let $\mathcal{P}^{\pm} \subseteq \mathcal{G}$ be the parabolic subgroups defined by $\tilde{\mu}$, with their unipotent radicals denoted by $\mathcal{U}_{\pm}$ and their common Levi subgroup by $\mathcal{M}$, which is also the centralizer in $\mathcal{G}$ of $\tilde{\mu}$. We write $P_{\pm}$ and $ U_{\pm}$ for the special fibre of $\mathcal{P}_{\pm}$ and $\mathcal{U}_{\pm}$ respectively. 
	
	Consider the following subsheaves of $\mathcal{L}_{\cris}^+\mathcal{G} = \mathcal{G}$ over $(\mathbf{FlatAlg}/W)$:
	\begin{align*}
		\mathrm{H}^{\pm} & = \mathrm{Red}^{-1}(\mathrm{Res}_{\kappa/W}P_{\pm}), \,
		{}^+\mathrm{H} & = \mathcal{L}_{\cris}^+\mathcal{G} \cap \tilde{\mu}(p)^{-1} \mathcal{L}_{\cris}^+\mathcal{G} \tilde{\mu}(p), \,
		{}^-\mathrm{H} & = \mathcal{L}_{\cris}^+\mathcal{G} \cap \tilde{\mu}(p) \mathcal{L}_{\cris}^+\mathcal{G} \tilde{\mu}^{-1}(p),
	\end{align*}
	where $\mathrm{Red}: \mathcal{G}\to \mathrm{Res}_{\kappa/W}G$ denotes the canonical reduction map. These are, in fact, sheaves over $(\mathbf{FlatAlg}/W)$, as is $\mathcal{L}_{\cris}^+\mathcal{G} = \mathcal{G}$. Define $\mathrm{E}_{\tilde{\mu}}^{\dagger} \subseteq \mathrm{H}^{-} \times \mathrm{H}^{+}$ as the preimage of the subgroup $\mathsf{E}_{\mu} \subseteq P_- \times P_+$ along the canonical reduction map $\mathrm{H}^- \times \mathrm{H}^+ \to P_- \times P_+$. The following lemma can be viewed as a mixed characteristic analogue of (parts of) \cite[Lemmata 3.1, 3.2]{Yan23zip}.
	\begin{lemma}\label{Lem:KeyLemma} 
		\begin{enumerate}
			\item We have the following inclusion:
			\begin{equation*}
				\begin{array}{ccc}
					\tilde{\mu}(p) \mathcal{L}_{\cris}^{+} \mathcal{P}_+ \tilde{\mu}(p)^{-1} \subseteq \mathcal{L}_{\cris}^{+} \mathcal{G}, & \tilde{\mu}(p) \mathcal{L}_{\cris}^{+} \mathcal{U}_+ \tilde{\mu}(p)^{-1} \subseteq \mathcal{L}_{\cris}^{1} \mathcal{U}_+ \subseteq \mathcal{L}_{\cris}^1 \mathcal{G}, \\
					\tilde{\mu}(p)^{-1} \mathcal{L}_{\cris}^{+} \mathcal{P}_- \tilde{\mu}(p) \subseteq \mathcal{L}_{\cris}^{+} \mathcal{G}, & \tilde{\mu}(p)^{-1} \mathcal{L}_{\cris}^{+} \mathcal{U}_- \tilde{\mu}(p) \subseteq \mathcal{L}_{\cris}^{1} \mathcal{U}_- \subseteq \mathcal{L}_{\cris}^1 \mathcal{G}.
				\end{array}
			\end{equation*}
			\item We always have ${}^{\pm}\mathrm{H} \subseteq \mathrm{H}^{\pm}$.
			\item The image of the embedding $\mathrm{H}^{-} \hookrightarrow \mathrm{H}^{-} \times \mathrm{H}^{+}$ given by $h \mapsto (h, \tilde{\mu}(p)^{-1} h \tilde{\mu}(p))$ lies in $\mathrm{E}_{\tilde{\mu}}$.
		\end{enumerate}
	\end{lemma}
	
	\begin{remark}
		\begin{enumerate}
			\item Lemma \ref{Lem:KeyLemma} holds without requiring that $\tilde{\mu}$ is minuscule. 
			\item Another way to state the theorem is that the association \(\bar{R} \mapsto \prescript{}{1}{\mathbf{C}_1^{\tilde{\mu}}}(R)\), for any lift \(R \in W_{\cris}\) of \(\bar{R}\), is well-defined and defines a sheaf over \(\kappa_{\cris}\), represented by the \(\kappa\)-scheme \(\prescript{}{1}{\mathcal{C}_1^{\mu}} \cong G^2 / \mathsf{E}_{\mu}\). This motivates our use of base reduction diagrams: without them, defining \(\prescript{}{1}{\mathbf{C}_1^{\tilde{\mu}}}\) before Theorem~\ref{Thm:ExtdRepthm}, or even stating the theorem, would be unnatural.
			
		\end{enumerate}
	\end{remark}
	
	\section{Prismatic Frobenius period maps}  \label{S:FrobPerMap}
	\subsection{Two restrictions of the prismatic $\sG$-torsor $\mathbb{J}_{\prism}$}
	In this section, we work in the setting of \S~\ref{S:ShVLoc}, adopting the notation and conventions of \S~\ref{S:BasicSetup} with $\mathfrak{X}=\sS$. The goal is to define prismatic Frobenius period maps for the pair $(\mathcal{S},S)$.

	Recall from Theorem \ref{Thm: PrisGTors} that over the absolute prismatic site $\sS_{\prism}$ is a $\sG$-torsor $\mathbb{J}_{\prism}$, 
	trivializing the cohomology $\mathfrak{M}_{\widehat{\sA}/\widehat{\sS}}:=\mathrm{H}^1_{\prism}(\widehat{\sA}/\widehat{\sS})$ and Frobenius invariant prismatic tensors $\mathrm{T}_{\prism}\subseteq \mathfrak{M}_{\widehat{\sA}/\widehat{\sS}}^{\otimes}$. We denote by $\mathbb{J}_{\BKprism}$ and $\mathbb{J}_{\crispris}$ for the restrictions of $\mathbb{J}_{\prism}$ to $ \sS_{\BKprism}$ and to $\sS_{\crispris}$ respectively.
	\begin{definition}\label{Def:BkCrisTors}
		Define $\mathbb{J}_{\BKcrispris}$ to be the pullback to $\mathcal{S}_{\tilde{\prism}_{\cris}}$ of the $\mathbb{J}_{\BKprism}$ along  $\iota: \mathcal{S}_{\tilde{\prism}_{\cris}}\to \mathcal{S}_{\BKprism}$. 
	\end{definition}
	We have the following commutative diagram of reduction maps,
	\begin{equation}\label{Eq: ComDiagRedMap}
		\xymatrix{	\mathbb{J}_{\BKcrispris}\ar[d]_{\mathcal{L}^+_{\BKcrispris}\mathcal{G}}\ar[rr]^{\mathrm{Red}_{\cris}^{\BK}}&&\mathbb{J }_{\crispris}\ar[d]^{\mathcal{L}^+_{\crispris}\mathcal{G}}\\
			\sS_{\tilde{\prism}_\cris}\ar[rr]^{\mathrm{Red}}&& S_{\crispris}, 
		}
	\end{equation}
	where the vertical arrows are torsors under the corresponding group $\Spaces$.
	
	\subsection{Breuil-Kisin prismatic Frobenius period map}\label{S: BKPerMap}
	For each point $\underline{x}=(\underline{R}, x: \Spf R \to \widehat{\sS})$ of $\sS_{\tilde{\prism}_\cris}$, we denote by 
	\(
	\varphi_{\underline{x}}: \varphi_{\mathfrak{S}_R}^* \mathfrak{M}_{\underline{x}} \to \mathfrak{M}_{\underline{x}}
	\)
	the Frobenius map of $\mathfrak{M}_{\underline{x}}:= \mathrm{H}_{\prism}^1(\mathcal{A}_x/\underline{\mathfrak{S}_R})$, the first prismatic cohomology of $\sA_x$ over the Breuil-Kisin prism $\underline{\mathfrak{S}_R}=(\mathfrak{S}_R, \varphi_{\mathfrak{S}_R}, E)$ as in \eqref{Eq:CrisPrisRedMap}. The fibre $\Space$ $\mathbb{J}_{\BKcrispris}(\underline{x})$ of $\underline{x}$, is the set (i.e., discrete groupoid) of \emph{global sections} of the following $\mathfrak{S}_R$-scheme,
	\[
	\mathbb{J}_{\BKcrispris, \underline{x}}:= \underline{\mathrm{Isom}}\big((\Lambda_0^* \otimes_{\Zp}\mathfrak{S}_R, \mathrm{T}_0\otimes 1),  (\mathrm{H}^1_{\prism}(\sA_x/\underline{\mathfrak{S}_R}), \mathrm{T}_\prism)\big).
	\]
	For a section $\tilde{\beta}\in \mathbb{J}_{\BKcrispris, \underline{x}}(\mathfrak{S}_R)$, write $\tilde{\beta}^*(\varphi_{\underline{x}})$ for the composition,
	\(
	\tilde{\beta}^{-1} \varphi_{\underline{x}}\varphi^*(\tilde{\beta}) 
	\), which is an element in $\mathcal{L}_{\BKcrispris}\mathcal{G}(\underline{R})= \sG(\mathfrak{S}_R[1/E])$.
	Here we have used the identification $\varphi_{\mathfrak{S}_R}^*(\Lambda_0^*\otimes_{\Zp}\mathfrak{S}_R) = \Lambda_0^*\otimes_{\Zp}\mathfrak{S}_R$, noting that $\Lambda_0^*$ is defined over $\Zp$.
	
	\begin{lemma}\label{Lem: DefBKPerMap}
		The association $\tilde{\beta} \mapsto \tilde{\beta}^*(\varphi_{\underline{x}})$ defines a morphism,
		\(
		\xi_{\BKcrispris}^{\sharp}: \mathbb{J}_{\BKcrispris} \to \mathbf{C}_{\BKcrispris}^{\tilde{\mu}},
		\)
		which is $\mathcal{L}^+_{\BKcrispris}\mathcal{G}$-equivariant and hence 
		induces a morphism $\xi_{\BKcrispris}: \sS_{\tilde{\prism}_\cris} \to [\mathbf{C}_{\BKcrispris}^{\tilde{\mu}}/\mathrm{Ad}_{\varphi}\mathcal{L}^+_{\BKcrispris}\mathcal{G}]$ of $W_{\crispris}$-stacks.
	\end{lemma}
	\begin{proof}
		We only need to show that the trivialization
		\(
		\tilde{\beta}^*(\varphi_{\underline{x}})\) lies in \( \mathbf{C}^{\tilde{\mu}}_{\BKcrispris}(\underline{R})\).
		But this follows from the combination of Theorem \ref{Thm: PrismticTypeMu} and Lemma \ref{Lem: BKFrob}.
	\end{proof}
	We will refer to $\xi_{\BKcrispris}$ or $\xi_{\BKcrispris}^{\sharp}$  as the \emph{Breuil-Kisin (prismatic) Frobenius period map} for $\sS=\sS_{\K}$.
	
	\subsection{Crystalline prismatic Frobenius period map} \label{S: CryPerMap}
	This subsection is parallel to the previous subsection \S \ref{S: BKPerMap} and can indeed be deduced from the discussion therein. However, we explicate it here for the convenience of future reference.
	
	For each point $\underline{\bar{x}}=(\underline{R}, \bar{x}: \mathrm{Spec} \bar{R} \to S)$ of $S_{\crispris}$, we denote by
	\(
	\varphi_{\underline{\bar{x}}}: \varphi_R^* \mathrm{H}_{\prism}^1(\mathcal{A}_{\underline{\bar{x}}}/\underline{R}) \to \mathrm{H}_{\prism}^1(\mathcal{A}_{\underline{\bar{x}}}/\underline{R}),
	\)
	the Frobenius map of $ \mathrm{H}_{\prism}^1(\mathcal{A}_{\underline{\bar{x}}}/\underline{R})$. The fiber $\Space$ $\mathbb{J}_{\crispris}(\underline{\bar{x}})$ of $\underline{\bar{x}}$, is the set of global sections of the following $R$-scheme,
	\[
	\mathbb{J}_{\crispris, \, \underline{\bar{x}}}:= \underline{\mathrm{Isom}}\big((\Lambda_0^* \otimes_{\Zp}R, \mathrm{T}_0\otimes 1), (\mathrm{H}^1_{\prism}(\sA_{\bar{x}}/\underline{R}), \mathrm{T}_\prism)\big).
	\]
	For a section $\beta \in \mathbb{J}_{\crispris, \, \underline{\bar{x}}}(R)$, write $\beta^*(\varphi_{\underline{\bar{x}}})$ for the composition
	\(
	\beta^{-1} \varphi_{\underline{\bar{x}}} \varphi^*(\beta) 
	\), which is an element of $\sG(R[1/p]) = \mathcal{L}_{\crispris} \mathcal{G}(\underline{R})$. 
	
	\begin{lemma}\label{Lem:DefCrisPerMap}
		The association $\beta \mapsto \beta^*(\varphi_{\underline{\bar{x}}})$ defines a morphism,
		\(
		\xi_{\crispris}^{\sharp}: \mathbb{J}_{\crispris} \to \mathbf{C}_{\tilde{\prism}_\cris}^{\tilde{\mu}}, 
		\) which is $\mathcal{L}^+_{\crispris}\mathcal{G}$-equivariant and hence
		induces a map $\xi_{\crispris}: S_{\crispris} \to [\mathbf{C}_{\tilde{\prism}_\cris}^{\tilde{\mu}}/\mathrm{Ad}_{\varphi}\mathcal{L}^+_{\crispris}\mathcal{G}]$ of $W_{\crispris}$-stacks.
	\end{lemma}
	\begin{proof}
		Same as the proof of Lemma \ref{Lem: DefBKPerMap}.
	\end{proof}
	We will refer to $\xi_{\crispris}$ (equivalently $\xi^{\sharp}_{\crispris}$) as the \emph{crystalline (prismatic) Frobenius period map} for the pair $(\mathcal{S}, S)$ (or simply for $S$).  
	
	\subsection{Connection with Classical Crystalline Period Maps}\label{S:CompareClasCryPerMap}
	
	Let \( \varphi(\tilde{\mu}) \) denote the base change of \( \tilde{\mu} \) along the Frobenius \( \varphi: W \to W \). Recall that the classical map,
	\[
	\xi_{\mathrm{cl}}: S(\bar{\mathbb{F}}_p) \longrightarrow \mathbf{C}^{\varphi(\tilde{\mu})}_{\tilde{\prism}_\cris}(\underline{\Breve{\mathbb{Z}}_p})/\mathrm{Ad}_{\varphi} \mathcal{G}(\underline{\Breve{\mathbb{Z}}_p}),
	\]
	where \( \Breve{\mathbb{Z}}_p := W(\bar{\mathbb{F}}_p) \), is defined by trivializations of the Frobenius map on the crystalline cohomology \( \mathrm{H}^1_{\cris}(\mathcal{A}_{\bar{x}}/\Breve{\mathbb{Z}}_p) \), preserving tensors, with \( \bar{x} \) a point of \( S(\bar{\mathbb{F}}_p) \).
	
	The deprism map \( S_{\crispris} \to S_\cris \), when restricted to \( S_{\cris}^{\perf} \subseteq S_{\cris} \), admits a section \( S_{\cris}^\perf \hookrightarrow S_{\crispris} \), identifying \( S_{\cris}^{\perf} \) with \( S_{\crispris}^{\perf} \subseteq S_{\crispris} \), where the subspace  \( S_{\crispris}^{\perf} \) consists of perfect points, i.e., those points \( ((R, \varphi_R), \bar{x}) \) where \( \varphi_R \) is a bijection. Thus, the map \( \xi_{\mathrm{cl}} \) extends naturally to a \( W_{\crispris} \)-stack map:
	\[
	\xi_{\mathrm{cl}}^{\perf}: S_{\cris}^{\perf} = S_{\crispris}^{\perf} \longrightarrow [\mathbf{C}^{\varphi(\tilde{\mu})}_{\tilde{\prism}_\cris}/\mathrm{Ad}_{\varphi} \mathcal{L}^+_{\crispris} \mathcal{G}]^{\perf},
	\]
	which we term as the \emph{classical crystalline Frobenius period map} for \( S \). This map can be recovered from $\xi_{\crispris}^{\perf}$ from the relation, 
	\[
	\xi_{\mathrm{cl}}^{\perf} = [\varphi]^{\perf} \circ \xi_{\crispris}^{\perf},
	\]
	where \( \xi_{\crispris}^{\perf} \) is the restriction to \( S_{\cris}^{\perf} \) of \( \xi_{\crispris} \)  and \([\varphi]^{\perf}\) is the restriction to \( S_{\cris}^{\perf} \) of
	\[
	[\varphi]: [ \mathbf{C}^{\tilde{\mu}}_{\tilde{\prism}_\cris}/\mathrm{Ad}_{\varphi} \mathcal{L}^+_{\crispris} \mathcal{G} ] \to [\mathbf{C}^{\varphi(\tilde{\mu})}_{\tilde{\prism}_\cris}/\mathrm{Ad}_{\varphi} \mathcal{L}^+_{\crispris} \mathcal{G}],  
	\] which is induced by the Frobenius map \( \varphi: \mathbf{C}^{\tilde{\mu}}_{\tilde{\prism}_{\cris}} \to \mathbf{C}^{\varphi(\tilde{\mu})}_{\tilde{\prism}_{\cris}} \).

	\section{Reductions of prismatic $\mathcal{G}$-torsors}\label{S:RedGTorsor}
	In this section, we study the reductions of various prismatic $\mathcal{G}$-torsors. 
	
	\subsection{Reduction of Breuil-Kisin prismatic $\sG$-torsors}
	\label{S:RedofBKPerMap}
	Denote by	\(
	\mathbb{J}_{\BKcrispris}\otimes_{\underline{\mathfrak{S}}}\underline{W}
	\) 
	the presheaf on $\sS_{\tilde{\prism}_\cris}$  given by associating a point  $\underline{x}=(\underline{R}, x: \Spf R\to \widehat{\sS})$ the set of global sections of the $R$-scheme, 
	\[
	\mathbb{J}_{\BKcrispris, \underline{x}}\otimes _{\mathfrak{S}_R, t\mapsto 0} R:=	\underline{\mathrm{Isom}}\big((\Lambda_0 \otimes_{\Zp}\mathfrak{S}_R, \mathrm{T}_0\otimes 1)\otimes _{\mathfrak{S}_R, t\mapsto 0} R,  (\mathrm{H}^1_{\prism}(\sA_x/R), \mathrm{T}_\prism)\otimes _{\mathfrak{S}_R, t\mapsto 0} R\big). 
	\]
	
	\begin{lemma}\label{Lem:ModtIdentftion}
		The $\sS_{\tilde{\prism}_\cris}$-$\Space$
		$\mathbb{J}_{\BKcrispris}\otimes_{\underline{\mathfrak{S}}}\underline{W}$ is of reduction by $\mathbb{J}_{\crispris}$ along the crystalline prismatic reduction map $\mathrm{Red}: \sS_{\tilde{\prism}_\cris}\to S_{\crispris}$. 
	\end{lemma}
	\begin{proof} This is clear: as for each point $\underline{x}=(\underline{R}, x)$ of  $\mathcal{\sS}_{\tilde{\prism}_\cris}$, we have a canonical isomorphism 
		\[(\mathrm{H}_{\prism}^1(\mathcal{A}_x/\underline{\mathfrak{S}_R}), \mathrm{T}_{\prism})\otimes_{\mathfrak{S}_R}\mathfrak{S}_R/(t)\cong (\mathrm{H}_{\prism}^1(\mathcal{A}_{\bar{x}}/\underline{R}), \mathrm{T}_{\prism}),\]
		since the reduction modulo $t$ map $(\underline{\mathfrak{S}_R}, x) \to (\underline{R}, \bar{x})$ is a morphism of the site $\sS_{\BKcrispris}$.  
	\end{proof}
	
	With the identification in Lemma \ref{Lem:ModtIdentftion}, the reduction map $\mathrm{Red}^{\BK}_{\cris}: \mathbb{J}_{\BKcrispris}\to \mathbb{J}_{\crispris}$ defined in \eqref{Eq: BKCrisRedMap} is identified with the canonical reduction map $\mathbb{J}_{\BKcrispris}\to \mathbb{J}_{\BKcrispris}\otimes_{\underline{\mathfrak{S}}}\underline{W}$, and hence forms a $\mathcal{L}^1_{\BKcrispris}\mathcal{G}$-torsor, when viewed as a map of $W_{\crispris}$-$\Spaces$. Here $\mathcal{L}^1_{\BKcrispris}\mathcal{G}$ denotes the kernel of the reduction map $\mathrm{Red}_{\cris}^{\BK}: \mathcal{L}_{\BKcrispris}^+\mathcal{G} \to \mathcal{L}_{\crispris}^+\mathcal{G} $. Note that it is NOT the pullback along $\iota: W_{\crispris}\to W_{\BKprism}$ of $\mathcal{L}^1_{\BKprism}\mathcal{G}$. 
	
	\begin{lemma}\label{Lem:  PrisPerMapDiagm}
		The following commutative diagram of maps between $W_{\crispris}$-$\mathsf{spaces}$, 
		\begin{align}\label{Eq: PrisPerMapDiagm}
			\begin{array}{c@{\hspace{0.5cm}}c}
				\xymatrix{
					\mathbb{J}_{\BKcrispris} \ar[r]^{\xi^{\sharp}_{\BKcrispris}} \ar[d]  & \mathbf{C}^{\tilde{\mu}}_{\BKcrispris} \ar[d]\ar[dr]^{\text{mod } t}& \\
					\mathbb{J}_{\crispris} \ar[r] \ar@/_1.5pc/[rr]_{\quad \xi^{\sharp}_{\crispris}} & [\mathcal{L}^1_{\BKcrispris}\mathcal{G}\backslash \mathbf{C}_{\BKcrispris}^{\tilde{\mu}}/\mathcal{L}^1_{\BKcrispris}\mathcal{G}]\ \ar[r]^{\quad \quad \text{mod } t} & \mathbf{C}^{\tilde{\mu}}_{\crispris},
				}
				&
				\xymatrix{
					\sS_{\tilde{\prism}_\cris} \ar[rr]^{\xi_{\BKcrispris}\quad \quad \quad} \ar[d] && [\mathbf{C}_{\BKcrispris}^{\tilde{\mu}}/\mathrm{Ad}_{\varphi}\mathcal{L}_{\BKcrispris}^+\mathcal{G}] \ar[d]^{\text{mod } t} \\
					S_{\crispris} \ar[rr]^{\xi_{\crispris}\quad \quad \quad} && [\mathbf{C}_{\crispris}^{\tilde{\mu}}/\mathrm{Ad}_{\varphi}\mathcal{L}_{\crispris}^+\mathcal{G}].
				}
			\end{array}
		\end{align}
	\end{lemma}
	\begin{proof}
		We only need to show that the composition of  $\xi_{\BKcrispris}^{\sharp}$ with the canonical projection from  $\mathbf{C}^{\tilde{\mu}}_{\BKcrispris}$ to the quotient stack $[\mathcal{L}^1_{\BKcrispris}\mathcal{G}\backslash \mathbf{C}_{\BKcrispris}^{\tilde{\mu}}/\mathcal{L}^1_{\BKcrispris}\mathcal{G}]$ factors through the  reduction map $\mathbb{J}_{\BKcrispris}\to \mathbb{J}_{\crispris}$. Indeed, this follows from the combination of the following three facts: (1) for each point $\underline{x}=(\underline{R}, x)$ of $\sS_{\tilde{\prism}_\cris}$, with reduction $\underline{\bar{x}}$ in $S_{\crispris}$,  the canonical map $\mathbb{J}_{\BKcrispris, \underline{x}}\to  \mathbb{J}_{\crispris, \underline{\bar{x}}}$ is a $\mathcal{L}^1_{\BKcrispris}\mathcal{G}(\underline{R})$-torsor. (2) The map $\xi_{\BKcrispris}^{\sharp}$ is $\mathcal{L}^+_{\BKcrispris}\mathcal{G}$- $\varphi$-equivariant. (3) The subgroup functor $\mathcal{L}^1_{\BKcrispris}\mathcal{G}\subseteq \mathcal{L}^+_{\BKcrispris}\mathcal{G}$ is $\varphi$-stable. 
	\end{proof}
	
	\subsection{Reduction of crystalline prismatic $\sG$-torsor}\label{S:CryPrisTorRed}
	Our main goal in this subsection is to show that the crystalline prismatic $\sG$-torsor $\mathbb{J}_{\crispris}$ is of reduction along the defrobenius deduction map $S_{\crispris}\to S_{\widetilde{\cris}}$, that is, Theorem \ref{Thm:KeyRedSpace}. We need some preparation for this theorem.  
	
	We write $\mathrm{H}^1_{\cris}(\widehat{\sA}/\widehat{\sS})$ for the first relative crystalline cohomology of $\widehat{\sA}$ over $\widehat{\sS}$, by which we mean
	\[
	\mathrm{H}^1_{\cris}(\widehat{\sA}/\widehat{\sS}):= (\mathbf{R}^1\pi_{\mathrm{CRIS}, *}\mathcal{O}_{\sA}^{\cris} )_{\widehat{\sS}}\cong \mathbb{D}^*(\sA/S)_{\widehat{\sS}},
	\]
	the evaluation at the (pro) PD-thickening of the crystal $\mathbf{R}^1\pi_{\mathrm{CRIS}, *}\mathcal{O}_{\sA}^{\cris} $, which is canonically isomorphic to the Dieudonn\'e crystal $\mathbb{D}^*(\sA/S)$ (see \cite[{Proposition 3.3.7, Th\'eor\`em 2.5.6}]{BBMTheoriedeDieudonnecristalline}), where $\pi: \sA \to S$ is the structure moprhism. 
	Likewise, we write $\mathrm{H}^1_{\mathrm{dR}}(\widehat{\sA}/\widehat{\sS})$ for the first de Rham cohomology of $\widehat{\sA}$ relative to $\widehat{\sS}$. Then we have the canonical Berthelot-Ogus isomorphism,
	\begin{equation}\label{Eq: BOCompar}
		(\mathrm{H}^1_{\cris}(\widehat{\sA}/\widehat{\sS}), \mathrm{T}_{\cris})\cong (\mathrm{H}_{\mathrm{dR}}^1(\widehat{\sA}/\widehat{\sS}), \mathrm{T}_{\mathrm{dR}}).
	\end{equation}
	The matching of tensors is inherent in their construction; see for example \cite[Equation (3.3.6)]{KimRapoportUniformisation} for the construction of the crystalline tensors $\mathrm{T}_{\mathrm{cris}}$, and see \cite{KisinIntegralModels} for the construction of $\mathrm{T}_{\mathrm{dR}}$. 
	Define the following formal $\widehat{\sS}$-scheme, 
	\[    
	\mathbb{I}_{\cris}:=\underline{\mathrm{Isom}}_{\mathcal{O}_{\widehat{\mathcal{S}}}}\big((\Lambda_0^* \otimes_{\Zp}\mathcal{O}_{\widehat{\sS}}, \mathrm{T}_0\otimes 1),  (\mathrm{H}^1_{\cris}(\widehat{\sA}/\widehat{\sS}), \mathrm{T}_\cris)\big),
	\]
	which via \eqref{Eq: BOCompar}, is nothing but the $p$-adic completion of the following $\sG$-torsor $\mathbb{I}$ (see for example \cite[Lemma 3.4]{YanLocCon}),
	\[
	\mathbb{I}_{\mathrm{dR}}:=\underline{\mathrm{Isom}}_{\mathcal{O}_{\mathcal{S}}}\big((\Lambda_0^* \otimes_{\Zp}\mathcal{O}_\sS, \mathrm{T}_0\otimes 1),  (\mathrm{H}^1_{\mathrm{dR}}(\sA/\sS), \mathrm{T}_\mathrm{dR})\big). 
	\]
	Denote by $\Fil\mathrm{H}^1_{\cris}(\widehat{\sA}/\widehat{\sS})\subseteq \mathrm{H}^1_{\cris}(\widehat{\sA}/\widehat{\sS})\cong \mathrm{H}_{\mathrm{dR}}^1(\widehat{\sA}/\widehat{\sS})$ the preimage of the Hodge filtration $\omega_{\sA/S}\subseteq \mathrm{H}^1_{\mathrm{dR}}(\sA/S)$, under the canonical projection $\mathrm{H}_{\mathrm{dR}}^1(\widehat{\sA}/\widehat{\sS})\to \mathrm{H}_{\mathrm{dR}}^1(\sA/S) $. As can be seen by working Zariski locally (or Lemma \ref{Lem: CrisTens} below), the filtration $\Fil\mathrm{H}^1_{\cris}(\widehat{\sA}/\widehat{\sS})$ is locally free of the same rank as that of $\mathrm{H}^1_{\cris}(\widehat{\sA}/\widehat{\sS})$. Note that  inverting $p$ makes the inclusion
	\(
	\Fil\mathrm{H}^1_{\cris}(\widehat{\sA}/\widehat{\sS})\subseteq \mathrm{H}^1_{\cris}(\widehat{\sA}/\widehat{\sS})
	\) an identity
	\[
	\Fil\mathrm{H}^1_{\cris}(\widehat{\sA}/\widehat{\sS})[1/p]= \mathrm{H}^1_{\cris}(\widehat{\sA}/\widehat{\sS})[1/p].
	\]
	
	\begin{lemma} \label{Lem: CrisTens}
		We have 
		\(
		\mathrm{T}_{\cris}\subseteq \big(\Fil\mathrm{H}^1_{\cris}(\widehat{\sA}/\widehat{\sS})\big)^{\otimes}. 
		\)
	\end{lemma}
	\begin{proof}
		This can be verified locally: we may assume $\widehat{\sS}=\Spf R$ with $R$ a small formal $W$-algebra such that there exists a Frobenius lift $\varphi_R: R\to R$ making $(R, \varphi_R)$ a crystalline prism. 
		
		Taking global sections of the locally free sheaves \(
		\Fil\mathrm{H}^1_{\cris}(\widehat{\sA}/\widehat{\sS})\subseteq \mathrm{H}^1_{\cris}(\widehat{\sA}/\widehat{\sS})
		\)
		we obtain an inclusion $\Fil M \subseteq M$ of locally free $R$-module, where $M=\mathbb{D}^*(\sA_{\bar{R}})(R\twoheadrightarrow \bar{R})$, and $\Fil M \subseteq M$ is the pullback along the canonical projection $M\to \bar{M}: =M/pM$ of the filtration $\bar{M}^0\subseteq \bar{M}$, where $\bar{M}^0$ is a direct summand of $\bar{M}$ such that $\varphi^* \bar{M}^0\subseteq \bar{M}^{\varphi}:= \varphi^*\bar{M}$ is equally to the kernel of the Frobenius map $\mathrm{F}: \bar{M}^{\varphi} \to \bar{M}$. Alternatively, it is the direct summand of $\bar{M}$, which corresponds to the Hodge filtration $\omega_{\sA/\bar{R}}$ under the canonical isomorphism $\bar{M}\cong \mathrm{H}_{\mathrm{dR}}^1(\sA/\bar{R})$. 
		
		After a possible $p$-adic \'etale extension of $R$, we may assume that $\mathbb{I}_{\cris}(R)\neq \emptyset$. Choose an arbitrary trivialization $\beta\in \mathbb{I}_{\cris}(R)$, which induces an isomorphism of $R$-reductive group schemes, $\mathrm{GL}(M)\cong \mathrm{GL}(\Lambda_0^*)$ and thus realizes $\sG_R\subseteq \mathrm{GL}(\Lambda_0^*)_R$ as a reductive subgroup of $\mathrm{GL}(M)$, cut out by the tensors $\mathrm{T}_{\cris}\subseteq M^\otimes$. Moreover, if we let $M=M^1\oplus M^0$ be the weight decomposition induced by the cocharacter $\tilde{\mu}$ and $\beta$.  Then we have $\Fil M=M^1\oplus p M^0$, which is nothing but the image of $M$ under the tensor preserving isomorphism $\tilde{\mu}(p): M[1/p]\to M[1/p]$. Thus, we must have $\mathrm{T}_{\cris}=\tilde{\mu}(p)(\mathrm{T}_{\cris}) \subseteq (\Fil M)^{\otimes}$. 
	\end{proof}
	
	For the statement of Proposition \ref{Thm:KeyRedSpace} below, we define the following formal $\widehat{\sS}$-scheme: 
	\[
	\widehat{\mathbb{J}}: = \mathrm{Isom}_{\mathcal{O}_{\widehat{\sS}}}\big( (\Lambda_0^* \otimes_{\Zp}\mathcal{O}_{\widehat{\sS}}, \mathrm{T}_{0}), (\Fil\mathrm{H}^1_{\cris}(\widehat{\sA}/\widehat{\sS}), \mathrm{T}_{\cris}) \big).
	\]
	
	\begin{lemma}\label{Lem:petaleTorsor}
		The $p$-adic formal scheme $\widehat{\mathbb{J}}$ over $\widehat{\sS}$ is a $\sG$-torsor in the $p$-adic étale topology.
	\end{lemma}
	
	\begin{proof}
		As $\mathbb{I}_\cris$ is a $p$-adic $\sG$-torsor over $\widehat{\mathcal{S}}$, we may choose a $p$-adic étale covering of $\widehat{S}$ consisting of a family of morphisms $\Spf R \to 
		\widehat{S}$, where $R$ is a crystalline $W$-algebra, such that $\mathbb{I}(R)\neq \emptyset$. Clearly, the lemma follows from the following claim: for each such $\Spf R$, we have $\mathbb{J}(R)\neq \emptyset$.
		
		To see the claim, let $M, \Fil M$ be as in the proof of Lemma \ref{Lem: CrisTens}, and $\beta \in \mathbb{I}_{\cris}(R)$. As in loc. cit., via the trivialization $\beta$, we may view $\tilde{\mu}(p)$ as an element in $\mathrm{GL}_R (M[1/p])$, inducing an isomorphism $\tilde{\mu}(p): (M, \mathrm{T}_{\cris}) \cong (\Fil M, \mathrm{T}_\cris)$. Thus $\tilde{\mu}(p) \circ \beta \in \mathbb{J}(R)$.
	\end{proof}
	
	
	Denote by $\widehat{\mathbb{J}}_{\cris}$ its restriction on $\sS_{\cris}$ along the inclusion $\sS_{\cris} \subseteq \widehat{\sS}_{\mathrm{ZAR}}$.
	
	\begin{lemma}\label{Lem:TorRed}
		The $\sS_{\cris}$-$\mathsf{space}$ $\widehat{\mathbb{J}}_{\cris}$ is of reduction along the defrobenius map $\sS_{\cris} \to S_{\widetilde{\cris}}$. 
	\end{lemma}
	
	\begin{proof}
		As our integral model $\mathcal{S}$ is fixed, the $\sS_{\crispris}$-$\mathsf{space}$ $\Fil\mathrm{H}_{\cris}^1(\widehat{\sA}/\widehat{\sS})|_{\sS_{\cris}}$ is determined by the Dieudonn\'e crystal $\mathbb{D}^* (\sA/S)$ of the universal abelian scheme $\sA$ over $S$, and the crystalline tensors $\mathrm{T}_{\cris}$ are determined by the pair $\sA/S$ as well: see e.g. cite[3.6]{Yan18} for a slightly detailed discussion.
	\end{proof}
	We still denote by $\widehat{\mathbb{J}}_{\cris}$ the reduction to $S_{\widetilde{\cris}}$ of $\widehat{\mathbb{J}}_{\cris}$ along $\sS_{\cris} \to S_{\widetilde{\cris}}$. 
	
	\begin{theorem}
		\label{Thm:KeyRedSpace}
		The $S_{\crispris}$-$\Spaces$
		\(
		\mathbb{J}_{\crispris}\) is of reduction by $ \widehat{\mathbb{J}}_{\cris}$ along the reduction map $S_{\crispris} \to S_{\widetilde{\cris}}$. 
	\end{theorem}
	\begin{proof}
		For a point $\underline{\bar{x}}=(\underline{R}, \bar{x}: \mathrm{Spec} \bar{R} \to S)$ in $S_{\crispris}$. Write $(M_{\prism}, \varphi_{M_{\prism}})$ for the pair $(\mathrm{H}^1_{\prism}(\sA/\underline{R}), \varphi_{\mathrm{H}^1_{\prism}(\sA/\underline{R})})$, and let $M = \mathbb{D}^*(\sA_{\bar{R}})(R)$ be as in the proof of Lemma \ref{Lem: CrisTens}. Then attached to the Dieudonn\'e crystal is the Dieudonn\'e module $(M, \mathrm{F}, \mathrm{V})$, where 
		\(
		\mathrm{F}: \varphi^* M \to M \) and \( \mathrm{V}: M \to \varphi^* M
		\)
		are $R$-linear maps such that $\mathrm{V} \circ \mathrm{F} = p \cdot \mathrm{id}_{\varphi^*M}$ and $\mathrm{F} \circ \mathrm{V} = p \cdot \mathrm{id}_{M}$.
		We prove the proposition by establishing the isomorphism,
		\begin{equation}\label{Eq: CrisPrisVsCris}
			(M_{\prism}, \mathrm{T}_{\prism})\cong (\Fil M, \mathrm{T}_{\cris}).
		\end{equation}
		
		As the prismatic $F$-crystal 
		\(
		(\mathfrak{M}_{\widehat{\sA}/\widehat{\sS}}, \varphi_{\mathfrak{M}_{\widehat{\sA}/\widehat{\sS}}})
		\)
		is minuscule, there exists a unique $R$-linear map $\psi_{M_{\prism}}: M_{\prism} \to \varphi_R^* M_{\prism}$ such that $\psi_{M_{\prism}} \circ \varphi_{M_{\prism}} = p \cdot \mathrm{id}_{\varphi^*M_{\prism}}$ and $\varphi_{M_{\prism}} \circ \psi_{M_{\prism}} = p \cdot \mathrm{id}_{M_{\prism}}$.
		The relation between the quadruples $(M, \mathrm{F}, \mathrm{V}, \mathrm{T}_\cris)$ and $(M_{\prism}, \varphi_{M_{\prism}}, \psi_{M_{\prism}}, \mathrm{T}_{\prism})$ is that we have the identification, 
		\[
		(M, \mathrm{F}, \mathrm{V}, \mathrm{T}_{\cris}) = \varphi_R^* (M_{\prism}, \varphi_{M_{\prism}}, \psi_{M_{\prism}}, \mathrm{T}_{\prism}).
		\] 
		Moreover, we also have identification of filtered $R$-modules,
		\(
		(\mathrm{Im}(\psi_{M_{\prism}}) \subseteq \varphi^*_R M_{\prism}) = (\Fil M \subseteq M).
		\)
		As the map $\psi_{M_{\prism}}: M_{\prism} \to \varphi_R^* M_{\prism} = M$ intertwines the tensors $\mathrm{T}_{\prism} \subseteq (\Fil M)_{\prism}^\otimes$ with the tensors $\mathrm{T}_{\cris} \subseteq M^\otimes$, it induces the following isomorphism,
		\(
		\psi_{M_{\prism}}: (M_{\prism}, \mathrm{T}_{\prism}) \cong (\Fil M, \mathrm{T}_{\cris}).
		\) 
	\end{proof}
	
	\begin{remark}\label{Rmk: FilMDieuMod}
		It is well-known fact from Dieudonn\'e theory that the Frobenius map $\mathrm{F}: \varphi^* M\to M$ induces a bijection
		\(
		\mathrm{F}|_{\varphi^* \Fil M}: \varphi^* \Fil M \cong pM.  
		\)
		This is well-known (see for example \cite[Lemma 2.9]{Yan18} for a proof). From this one sees that the submodule $\Fil M \subseteq M$ is stable under $\mathrm{F}$ and $\mathrm{V}$ and hence admits a Frobenius structure 
		\[
		\mathrm{F}': \varphi^* \Fil M \to \Fil M, \text{ and } \mathrm{V}': \Fil M \to \varphi^* \Fil M,
		\]
		such that $\mathrm{V}' \circ \mathrm{F}' = p \cdot \mathrm{id}_{\varphi^*\Fil M}$ and $\mathrm{F}' \circ \mathrm{V}' = p \cdot \mathrm{id}_{\Fil M}$. For future reference, we write
		\begin{equation}\label{Eq:zipisom}
			\gamma_{\underline{\bar{x}}}:= \frac{1}{p}\mathrm{F}|_{\varphi^* \Fil M}: \varphi^* \Fil M \cong M.
		\end{equation}
		for this divided Frobenius isomorphism.
	\end{remark}
	
	Denote by $\mathrm{J}$ the special fiber of $\widehat{\mathbb{J}}$ and $\mathrm{J}_{\cris}$ the restriction of $\mathrm{J}$ on $S_{\cris}$ along the inclusion $S_{\cris} \subseteq S_{\mathrm{ZAR}}$.  Denote by $\mathrm{J}_{\crispris} = \mathbb{J}_{\crispris} \otimes_W \kappa$ the presheaf on $S_{\crispris}$ given by associating each point $\underline{\bar{x}} = (\underline{R}, \bar{x})$ of $S_{\crispris}$ the set of global sections of the $\bar{R}$-scheme,
	\[
	\mathrm{J}_{\crispris, \underline{\bar{x}}}:=\underline{\mathrm{Isom}}\big((\Lambda_0 \otimes_{\Zp}R, \mathrm{T}_0 \otimes 1) \otimes_R \bar{R},  (\mathrm{H}^1_{\prism}(\sA_{\bar{x}}/\underline{R}), \mathrm{T}_\prism) \otimes_R \bar{R} \big). 
	\]
	The following corollary is an immediate consequence of Theorem \ref{Thm:KeyRedSpace}. 
	
	\begin{corollary}
		The $S_{\crispris}$-$\Spaces$ is of reduction by 
		\(
		\mathrm{J}_{\cris}
		\) along the reduction map $S_{\crispris} \to S_{\cris}$.
	\end{corollary}

	\section{Reduction of prismatic Frobenius period maps}\label{S:RedFrobPerMap}
	Throughout the remainder of this article, we focus on reductions of the Frobenius period maps established in \S~\ref{S:FrobPerMap}. By definition, reductions of \emph{maps} include those of \emph{$\Spaces$}, the latter having been treated in \S~\ref{S:RepProof} and \S~\ref{S:RedGTorsor}.
	
	\subsection{Reduction of Breuil-Kisin prismatic Frobenius period map}
	
	Set $\mathcal{L}_{\BKcrispris}^{\dagger}\mathcal{G}$ to be the kernel of the following composition of canonical reduction maps of $W_{\crispris}$-sheaves, 
	\[
	\mathcal{L}^+_{\BKcrispris}\mathcal{G} \xrightarrow{\mathrm{Red}^{\BK}_{\cris}} \mathcal{L}^+_{\crispris}\mathcal{G} \xrightarrow{\text{mod } p} \sG_{\crispris}=G_{\crispris}. 
	\]
	Denote by $\mathcal{C}^{\mu}_{\cris}$ and $ (\prescript{}{1}{\mathcal{C}_{ 1}^{\mu}})_{\cris} $ the restriction on $\kappa_{\cris}\subseteq \kappa_{\mathrm{ZAR}}$ of the $\kappa$-schemes $\mathcal{C}^{\mu}$ and $\prescript{}{1}{\mathcal{C}_{1}^{\mu}}$ respectively, and by $(\mathcal{C}^{\mu})_{\crispris}$ and $ (\prescript{}{1}{\mathcal{C}_{ 1}^{\mu}})_{\crispris} $ their further restriction on $W_{\crispris}$. Then clearly the canonical mod $p$ reduction map $\mathbf{C}^{\tilde{\mu}}_{\BKcrispris}\to \mathcal{C}^{\mu}_{\crispris}$ of $W_{\crispris}$-$\Spaces$ induces a map on quotient $\Spaces$: 
	\(\mathcal{L}_{\BKcrispris}^{\dagger}\mathcal{G}\backslash\mathbf{C}_{\BKcrispris}^{\tilde{\mu}}/\mathcal{L}_{\BKcrispris}^{\dagger}\mathcal{G}\to (\prescript{}{1}{\mathcal{C}_{ 1}^{\mu}})_{\crispris},
	\)
	where $\mathcal{L}_{\BKcrispris}^{\dagger}\mathcal{G}\backslash\mathbf{C}_{\BKcrispris}^{\tilde{\mu}}/\mathcal{L}_{\BKcrispris}^{\dagger}\mathcal{G}$ denotes the quotient sheaf of $\mathbf{C}_{\BKcrispris}^{\tilde{\mu}}$ under the action of $\mathcal{L}_{\BKcrispris}^{\dagger}\mathcal{G} \times \mathcal{L}_{\BKcrispris}^{\dagger}\mathcal{G}$ by componentwise multiplication on the left and right.
	
	\begin{lemma}
		The map $\xi_{\BKcrispris}^{\sharp}: \mathbb{J}_{\BKcrispris} \longrightarrow \mathbf{C}_{\BKcrispris}^{\tilde{\mu}}$ induces a morphism of $W_{\crispris}$-$\mathsf{spaces}$, 
		\[
		\rho^{\sharp}_{\BKcrispris}:  \mathrm{J}_{\crispris}\longrightarrow \mathcal{L}_{\BKcrispris}^{\dagger}\mathcal{G}\backslash\mathbf{C}_{\BKcrispris}^{\tilde{\mu}}/\mathcal{L}_{\BKcrispris}^{\dagger}\mathcal{G}\xrightarrow{\text{mod } p} (\prescript{}{1}{\mathcal{C}_{1}^{\mu}})_{\crispris}. 
		\]
		And hence a morphism of $W_{\crispris}$-stacks, 
		\(
		\rho_{\BKcrispris}: S_{\crispris}\longrightarrow [\prescript{}{1}{\mathcal{C}_{1}^{\mu}}/\mathrm{Ad}_{\varphi} G]_{\crispris}.
		\)
	\end{lemma} 
	\begin{proof}
		The proof is similar to that of Lemma \ref{Lem:  PrisPerMapDiagm}, noticing that the canonical projection map $\mathbb{J}_{\BKcrispris} \to \mathrm{J}_{\crispris}$, viewed as a map of $W_{\crispris}$-$\Spaces$ through the structure map $S_{\crispris}\to W_{\crispris}$, is a $\mathcal{L}^{\dagger}_{\BKcrispris}\mathcal{G}$-torsor. 
	\end{proof}
	
	\subsection{Reduction of crystalline prismatic Frobenius period map}\label{S:RedCryPrisTor}
	Write
	$(\prescript{}{1}{ \mathbf{C}_{1}^{\tilde{\mu}}})_{\crispris}$ for the restriction to $W_{\crispris}$ of $(\prescript{}{1}{ \mathbf{C}_{1}^{\tilde{\mu}}})_{\cris}$ along the defrobenius reduction map $W_{\crispris} \to W_{\cris}$. Thanks to Theorem \ref{Thm:ExtdRepthm}, we are allowed to use the notation $(\prescript{}{1}{ \mathbf{C}_{1}^{\tilde{\mu}}})_{\crispris}$ instead of $(\prescript{}{1}{ \mathbf{C}_{1}^{\tilde{\mu}}})_{\tilde{\prism}_\cris}$. Then clearly the $\varphi$-conjugation action of $\mathcal{L}^+_{\crispris}\mathcal{G}$ on $(\prescript{}{1}{ \mathbf{C}_{1}^{\tilde{\mu}}})_{\crispris}$ factors through the quotient map $\mathcal{L}^+_{\crispris}\mathcal{G}\to \mathcal{G}_{\crispris} $. Denote by $[\prescript{}{1}{ \mathbf{C}_{1}^{\tilde{\mu}}}/\mathrm{Ad}_{\varphi} G]_\crispris$  the resulting quotient stack \footnote{Here we write $[\prescript{}{1}{ \mathbf{C}_{1}^{\tilde{\mu}}}/\mathrm{Ad}_{\varphi} G]_\crispris$ instead of $[(\prescript{}{1}{ \mathbf{C}_{1}^{\tilde{\mu}}})_{\crispris}/\mathrm{Ad}_{\varphi} G_{\crispris}]$ since Theorem \ref{Thm:ExtdRepthm} suggests that the quotient stack exists already in the Zariski topology.} over $W_{\crispris}$. 
	
	Since the canonical reduction map $\mathbb{J}_{\crispris}\to \mathrm{J}_{\crispris}$ is a $\mathcal{L}^{1}_{\crispris}\mathcal{G}$-torsor, the following lemma is clear. 
	\begin{lemma}
		The  $\mathcal{L}^+_{\crispris}\mathcal{G}$-equivariant  map $\xi^{\sharp}_{\crispris}: \mathbb{J}_{\crispris}\to \mathbf{C}^{\tilde{\mu}}_{\tilde{\prism}_\cris}$ induces a $\mathcal{G}_{\crispris}$-equivariant morphism \(
		\rho_{\crispris}^{\sharp}:\mathrm{J}_{\crispris}\to (\prescript{}{1}{ \mathbf{C}_{1}^{\tilde{\mu}}})_\crispris
		\) of $W_{\crispris}$-$\mathsf{spaces}$,
		inducing a morphism of $W_{\crispris}$-stacks,
		\[
		\rho_{\crispris}:~S_{\crispris}~\to~ [\prescript{}{1}{ \mathbf{C}_{1}^{\tilde{\mu}}}/\mathrm{Ad}_{\varphi} G]_\crispris.
		\]
	\end{lemma}
	
	
	Denote by $\epsilon_{\crispris}$ the restriction along $W_{\crispris}\!\to\! W_{\cris}$ of the isomorphism $(\prescript{}{1}{\mathbf{C}_1^{\tilde{\mu}}})_{\cris}\cong(\prescript{}{1}{\mathcal{C}_1^{\mu}})_{\cris}$ in \eqref{Eq:2CC1Isom}.
	
	\begin{theorem}
		We have $\rho_{\BKcrispris}=\epsilon_{\crispris}\circ \rho_{\crispris} $.
	\end{theorem}
	\begin{proof}
		The assertion can be verified locally. Hence, it suffices to check that the compositions of $\mathbb{J}_{\BKcrispris}\to \mathrm{J}_{\crispris}$ with $\rho_{\crispris}$ and with $\epsilon_{\crispris}\circ \rho_{\BKcrispris}$ are equal.
		Thus, we are reduced to verifying that the diagram below between maps of   $W_{\crispris}$-$\mathsf{spaces}$ is commutative: this is clear by definition of $\epsilon_{\crispris}$. 
		\[
		\xymatrix{
			\mathbf{C}^{\tilde{\mu}}_{\BKcrispris}\ar[d]^{\text{mod } t}\ar[r]^{\text{mod } p} & \mathcal{C}^{\mu}_{\crispris}\ar[r]^{\Pr} & (\prescript{}{1}{\mathcal{C}^{\mu}_{1}})_{\crispris}\ar[d]^{\epsilon^{-1}_{\crispris}} \\
			\mathbf{C}^{\tilde{\mu}}_{\tilde{\prism}_\cris}\ar[rr]^{\Pr} & & (\prescript{}{1}{ \mathbf{C}_{1}^{\tilde{\mu}}})_{\crispris}.
		}
		\]
	\end{proof}
	
	\subsection{Recovering the zip period map $\zeta$}\label{S:RecoverZipPerMap}
	
	
	

	As recalled in the introduction, there is a canonical morphism of $\kappa$-stacks, \(\delta: [\prescript{}{1}{ \mathcal{C}_{1}^{\mu}}/\mathrm{Ad}_{\varphi}G]\to G\text{-}\mathsf{Zip}^{\mu},\) which becomes an isomorphism after taking perfection \cite[Corollary 6.7]{Yan23zip}.
	\begin{proposition}
		The composition \( \delta_{\crispris} \circ \rho_{\crispris}: S_{\crispris} \to (G\text{-}\textsf{Zip}^{\mu})_{\crispris} \) is of reduction along \( W_{\crispris} \to \kappa_{\cris} \) by \( \zeta_{\cris} \).
	\end{proposition}
	\begin{proof}
		It is enough to show that the morphisms $\delta_{\crispris}\circ \rho_{\crispris}$ and $\zeta_{\crispris}$ are 2-isomorphic. 
		This can be checked locally at crystalline points of $S$; for example, one can use the explicit local construction of $\zeta$ in \cite{Yan18} or the map $\eta^{\sharp}$ in \S \ref{S:PrisGauge} below. 
	\end{proof}
	\begin{remark}\label{Rmk:PrisPerpvsZipMap}
		The proposition implies that the Frobenius period map  $\rho_{\BKcrispris}$ determines the zip period map $\zeta$ \cite{ChaoZhangEOStratification} since $S_{\cris}$ contains sufficiently many crystalline points ($S$ is smooth over $\kappa$).
	\end{remark}

	\section{Zip gauges over $(\mathcal{S}, S)$}\label{S:ZipGauges}
	
	\subsection{The Universal Zip Torsor over $S$}\label{S:UnivZip}
	Consider the following filtered $O_S$-modules,
	\[
	(\mathcal{W}_0 \subseteq \mathcal{W}) := \left( p \mathrm{H}^1_{\mathrm{dR}}(\widehat{\sA}/\widehat{\sS}) \subseteq \Fil \mathrm{H}^1_{\mathrm{dR}}(\widehat{\sA}/\widehat{\sS}) \right) \otimes_{{\mathcal{O}}_{\widehat{\mathcal S}}} {\mathcal{O}}_S.
	\]
	\begin{lemma}\label{Lem:LocSummand}
		The submodule $\mathcal{W}_0\subseteq \mathcal{W}$ is a locally direct summand of $\mathcal{O}_S$-modules.
	\end{lemma}
	\begin{proof}
		This can be checked locally at crystalline points of $S$ (a family of crystalline points forming a Zariski covering of $S$ is sufficient). Take a point $\underline{\bar{x}}\in S_{\crispris}$ and use notations in Lemma \ref{Lem: CrisTens}. 	Write $i: \Fil M \to M$ for the canonical inclusion map and $\bar{i}: \overline{\Fil M} \to \bar{M}$
		its reduction modulo $p$, i.e., $i\otimes_R\mathrm{id}_{\bar{R}}$. Similarly, we write $\overline{\cdot p}: \bar{M} \to  \overline {\Fil M} $ for the reduction mod $p$ of the canonical multiplication by $p$ map $M\xrightarrow {\cdot p} \Fil M$. Write $(\overline{\Fil M})_0:=\mathrm{Im}(\bar{i})$. Let $M=M^1\oplus M^0$ and $\Fil M= M^1\oplus p M^0$ be the decomposition as in Lemma \ref{Lem: CrisTens}, then clearly we have 
		\[
		(\overline{\Fil M})_0= p M^0/p^2 M^0\subseteq M^1/p M^1\oplus p M^0/ p^2 M^0.
		\]
		Thus, $ (\overline{\Fil M})_0$ is a direct summand of $\overline{\Fil M}$. 
	\end{proof}
	
	Now we can define the following subscheme of the $G$-torsor $\mathrm{J}$ over $S$,
	\[
	\mathrm{J}_- := \mathrm{Isom}\big( (( \Lambda^{*,0}_0 \subseteq \Lambda_0^* ) \otimes_{\Zp} \mathcal{O}_S, \mathrm{T}_0 \otimes 1), \, (\mathcal{W}_0 \subseteq \mathcal{W}, \mathrm{T}_{\mathrm{\cris}}) \big).
	\]
	\begin{lemma}\label{Lem:PminusTors}
		The scheme $\mathrm{J}_-$ is a $P_-$-torsor over $S$.
	\end{lemma}
	\begin{proof}
		This can be verified locally at crystalline points using the decomposition $\Fil M = M^1 \oplus p M^0$ as in Lemma \ref{Lem: CrisTens}. To be precise, via the trivialization $\beta$ there, $\tilde{\mu}$ induces the weight decomposition $M=M^1\oplus M^0$, and the following isomorphism between filtered $R$-modules,
		\[
		(\Lambda_0^{*, 0}\subseteq \Lambda_0^*)\otimes_{\Zp}R\xrightarrow[\cong]{\beta}(M^0\subseteq M)\xrightarrow[\cong]{\beta \tilde{\mu}(p)\beta^{-1}} (pM^0\subseteq M^0),
		\]
		whose reduction mod $p$ provides a section of $\mathrm{J}_-(\bar{R})$ (cf. the proof of \ref{Lem:LocSummand}). 
	\end{proof}
	
	Set
	\[
	\mathcal{V} := \mathrm{H}^1_{\cris}(\widehat{\sA}/\widehat{\sS}, \mathrm{T}_{\cris}) \otimes_{{\mathcal{O}}_{\widehat{\mathcal S}}} {\mathcal{O}}_S = \mathrm{H}^1_{\mathrm{dR}}(\widehat{\sA}/\widehat{\sS}, \mathrm{T}_{\mathrm{dR}}) \otimes_{{\mathcal{O}}_{\widehat{\mathcal S}}} {\mathcal{O}}_S
	\]
	and set $\mathcal{V}^1 := \omega_{\mathcal{A}/S} \subseteq \mathcal{V}$ for its Hodge filtration. Then we can form the following subscheme of the $G$-torsor $\mathrm{I}_{\mathrm{dR}} := \mathbb{I}_{\mathrm{dR}} \otimes_{{\mathcal{O}}_{\widehat{\mathcal S}}} {\mathcal{O}}_S$ over $S$,
	\[
	\mathrm{I}_+ := \mathrm{Isom}\big( (\Lambda_0^* \supseteq \Lambda^{*,1}_0) \otimes_{\Zp} \mathcal{O}_S, \mathrm{T}_0 \otimes 1), \, (\mathcal{V} \supseteq \mathcal{V}^1, \mathrm{T}_{\mathrm{dR}}) \big), 
	\]
	where $\Lambda^{*, 1}$ denotes the weight 1 direct summmad in the weight decomposition $\Lambda^*=\Lambda^{*, 1}\oplus \Lambda^{*, 0}$ induced by $\mu: \mathbb{G}_{m, \kappa}\to G\subseteq \mathrm{GL}_{\kappa}(\Lambda^*_0)$. 
	It is a $P_+$-torsor over $S$; see, for example, \cite[Theorem 3.4.1]{ChaoZhangEOStratification}. 
	
	\begin{lemma}[Universal Zip Torsor over $S$]
		The inclusion $\Fil \mathrm{H}^1_{\mathrm{dR}}(\widehat{\sA}/\widehat{\sS}) \subseteq \mathrm{H}^1_{\mathrm{dR}}(\widehat{\sA}/\widehat{\sS})$ induces an isomorphism of $\mathcal{O}_S$-modules,
		\(
		\mathcal{W}/\mathcal{W}_0 \cong \mathcal{V}^1;
		\)
		the multiplication by $p$ map
		\(
		\mathrm{H}^1_{\mathrm{dR}}(\widehat{\sA}/\widehat{\sS}) \xrightarrow{\cdot p} \Fil \mathrm{H}^1_{\mathrm{dR}}(\widehat{\sA}/\widehat{\sS})
		\)
		induces an isomorphism \(
		\mathcal{V}/\mathcal{V}^1 \cong \mathcal{W}_0
		\) of $\mathcal{O}_S$-modules. In particular, we obtain an isomorphism \[
		\mathrm{J}_-/U_- \cong \mathrm{I}_+/U_+,
		\] of $P_+/U_+=M= P_-/U_-$-torsors,
		and hence an $\mathsf{E}_{\mu}$-torsor $\mathsf{Z}_\mu$ over $S$, defined as a subscheme of $\mathrm{J}_-\times_S \mathrm{I}_+$, consisting of pairs of elements from \( \mathrm{J}_- \) and \( \mathrm{I}_+ \) that map to the same point in the common quotient \( \mathrm{J}_-/U_- \cong \mathrm{I}_+/U_+ \).
	\end{lemma}
	\begin{proof}
		Again, this can be verified locally by working with crystalline points and the decomposition $\Fil M = M^1 \oplus p M^0$ from Lemma \ref{Lem: CrisTens}. One verifies that the following sequence of $\mathcal{O}_{\widehat{S}}$-modules, 
		\[
		\Fil \mathrm{H}^1_{\mathrm{dR}}(\widehat{\sA}/\widehat{\sS}) \xrightarrow{i} \mathrm{H}^1_{\mathrm{dR}}(\widehat{\sA}/\widehat{\sS}) \xrightarrow{\cdot p} \Fil \mathrm{H}^1_{\mathrm{dR}}(\widehat{\sA}/\widehat{\sS}) \xrightarrow{i} \mathrm{H}^1_{\mathrm{dR}}(\widehat{\sA}/\widehat{\sS})
		\]
		is exact, from which the lemma follows.
	\end{proof}
	We shall refer to the $\mathsf{E}_{\mu}$-torsor $\mathsf{Z}_{\mu}$ over $S$ as the \emph{universal zip torsor} over $S$. 
	
	\subsection{A Zip-style Description of $\rho_{\prism_\cris^\BK}$}\label{S:PrisGauge}
	Recall that in \cite[Theorem B]{Yan23zip}, we establish a canonical isomorphism of $\kappa$-stacks, 
	\[
	[\prescript{}{1}{\mathcal{C}_{1}^{\mu}}/\mathrm{Ad}_{\varphi} G] \xrightarrow{\epsilon'} [G/\mathsf{R}_{\varphi} \mathsf{E}_{\mu}],
	\]
	where the action of $\mathsf{E}_{\mu}$ on $G$ is given by $g \cdot (p_-, p_+) = p_+^{-1} g \varphi(p_-)$. We now describe the composition 
	\[
	\eta_{\crispris} := \epsilon'_{\crispris}\circ \rho_{\BKcrispris}: S_{\crispris} \longrightarrow [G/\mathsf{R}_{\varphi} \mathsf{E}_{\mu}]_{\crispris},
	\]
	using the universal zip torsor $\mathsf{Z}_{\mu}$ of $S$. Note that defining $\eta_{\crispris}$ is equivalent to providing a $G_{\crispris}$-torsor over $S$, which is naturally given by $(\mathsf{Z}_{\mu})_{\crispris}$, together with an $(\mathsf{E}_{\mu})_{\crispris}$-equivariant map $\eta^{\sharp}:(\mathsf{Z}_{\mu})_{\crispris} \to G_{\crispris}$, which we now describe.
	
	Given a point $\underline{\bar{x}} = (\underline{R}, \bar{x}) \in S_{\crispris}$, the zip isomorphism $\gamma_{\underline{\bar{x}}}$ in \eqref{Eq:zipisom} induces an isomorphism of $\bar{R}$-schemes:
	\[
	\Gamma_{\underline{\bar{x}}}: \mathrm{J}_{\bar{x}}^\varphi \cong \mathrm{I}_{\bar{x}}, \quad \text{with} \quad 
	\mathrm{J}^{\varphi}_{\bar{x}} := \mathrm{Isom}\big( (\Lambda^*_0 \otimes_{\Zp} \bar{R}, \mathrm{T}_0 \otimes 1), \, (\mathcal{W}, \mathrm{T}_{\cris})_{\bar{x}} \otimes_{\bar{R}, \varphi} \bar{R} \big)
	\]
	being the Frobenius twist by $\varphi: \bar{R} \to \bar{R}$ of $\mathrm{J}_{\bar{x}}$. The map $\eta^{\sharp}$ sends an element $(\alpha, \beta)\in \mathrm{J}_-(\bar{R})\times \mathrm{I}_+(\bar{R})$ of $\mathsf{Z}_{\mu}(\bar{R}) $ to \(\beta^{-1} \circ \Gamma_{\underline{\bar{x}}}(\varphi(\alpha)) \in G(\bar{R})\), which is easily verified to be $(\mathsf{E}_{\mu})_{\crispris}$-equivariant. Note that here $\mathrm{J}^{\varphi}$ denotes the pullback of $\mathrm{J}$ along the \emph{absolute Frobenius} map $\varphi: S\to S$. The same principle applies for torsors over any base $\kappa$-scheme.
	
	\begin{remark}
		Although not exactly needed here, note that for a $\kappa$-scheme $T$, the groupoid
		\(
		\left[\prescript{}{1}{\mathcal{C}_{1}^{\mu}}/\mathrm{Ad}_{\varphi} G\right](T)
		\)
		is equivalent to the category whose objects are tuples
		\[
		\Bigl( \mathrm{J}_-,\, \mathrm{I}_+,\, \Gamma\colon (\mathrm{J}_-\times^{P_-}G)^{\varphi}\overset{\sim}{\to}\mathrm{I}_+\times^{P_+}G,\, \epsilon\colon \mathrm{J}_-/U_-\overset{\sim}{\to}\mathrm{I}_+/U_+\Bigr),
		\]
		where $\mathrm{J}_-$ is a $P_-$-torsor, $\mathrm{I}_+$ is a $P_+$-torsor, $\Gamma$ is an isomorphism of $G^{\varphi}=G$-torsors, and $\epsilon$ is an isomorphism of $M$-torsors. Morphisms are torsor maps compatible with additional (i.e., $\Gamma$ and $\epsilon$) structures.
	\end{remark}
	
	\subsection{A double-zip over $(\mathcal{S}, S)$}\label{S:doublezip}
	We refer to the pair $(\mathsf{Z}_\mu, \rho_{\crispris})$ as the \emph{prismatic zip gauge}\footnote{We adopt this terminology, inspired by X.~Shen's talk on his work \cite{ShenGauge}.}  associated with $(\mathcal{S}, S)$, which exists in the prismatic topology. And we call the data 
	\(
	(\mathsf{Z}_{\mu}, \ \zeta: S \to G\text{-}\mathsf{Zip}^{\mu}),
	\)
	equivalently, the tuple, 
	\begin{equation}\label{DoubleZipData}
		\big(\mathrm{J}_-\subseteq \mathrm{J}, \, \mathrm{I}_+\subseteq \mathrm{I}\supseteq \mathrm{I}_-, \, \mathrm{J}_-/U_- \cong \mathrm{I}_+/U_+,\, (\mathrm{I}_+/U_+)^{\varphi} \cong \mathrm{I}_-/U_-^{\varphi}\big),
	\end{equation}  
	a  \emph{double $G$-zip}, 
	where the data 
	\(
	\big(\mathrm{I}_+\subseteq \mathrm{I}\supseteq \mathrm{I}_-, (\mathrm{I}_+/U_+)^{\varphi} \cong \mathrm{I}_-/U_-^{\varphi}\big)
	\)
	is the universal $G$-zip \cite[Theorem 3.4.1]{ChaoZhangEOStratification} corresponding to the zip period map $\zeta$. Clearly, this double $G$-zip exists in the Zariski topology. As mentioned in Remark \ref{Rmk:PrisPerpvsZipMap}, it is essentially determined by the prismatic zip gauge $ (\mathsf{Z}_\mu, \rho_{\crispris})$. 
	
	However, the relationship between our zip gauge and Drinfeld's version of $G$-zip \cite[Definition 4.2.1]{Drinfeld2023shimurian} where he equips a connection on the $P_-^\sigma$-torsor $I_-$ satisfying the Katz condition, as well as the de Rham $F$-gauge with $G$-structure \cite[Definition 3.9]{ShenGauge} remains an interesting question. Additionally, we consider it interesting to study the moduli space of double-zips in \eqref{DoubleZipData}. 
	
	\begin{remark}\label{Rmk:Conn}
	From the double zip in \eqref{DoubleZipData}, we obtain the Frobenius descent result,
		\[
		(\mathrm{J}_-/U_-)^{\varphi} \cong \mathrm{I}_-/U_-^{\varphi}, \eqno{(*)}
		\]
		which corresponds to the Frobenius descent of the attached graded modules of the conjugate filtration defining the torsor $\mathrm{I}_-$ (cf. \cite[Theorem 3.4.1.(iii)]{ChaoZhangEOStratification}).
	\end{remark}

\end{document}